\documentclass[leqno]{amsart}
\usepackage[utf8]{inputenc}
\usepackage[T1]{fontenc}

\usepackage{boilerplate-paramalg-2022}

\usepackage[margin=1in]{geometry}

\usepackage{lmodern}
\usepackage{wrapfig}
\usepackage{csquotes}

\usepackage[shortlabels]{enumitem}

\usepackage{tikz-cd}
\usepackage{mathrsfs}
\usepackage{sseq}
\usepackage{mathdots}
\usepackage{thm-restate}
\usepackage{amsmath}
\usepackage{amsxtra}
\usepackage{amscd}
\usepackage{amsthm}
\usepackage{amsfonts}
\usepackage{amssymb}
\usepackage{mathtools}
\usepackage[scr=rsfs]{mathalfa}
\usepackage{eucal}
\usepackage{bbm}
\usepackage[all,cmtip]{xy}

\usepackage{latexsym}
\usepackage{mdwlist}

\usepackage{xifthen}

\newcommand{\ccK}{\CMcal{K}}
\newcommand{\ccR}{\CMcal{R}}

\title{Parametrized and equivariant higher algebra}

\author{Denis Nardin}
\address{Fakultät für Mathematik, Universität Regensburg, 93040 Regensburg, Germany}
\email{denis.nardin@ur.de}

\author{Jay Shah}
\address{Fachbereich Mathematik und Informatik, WWU Münster, 48149 M\"{u}nster, Germany}
\email{jayhshah@gmail.com}

\begin{document}

\tikzcdset{arrow style=tikz, diagrams={>=stealth}}

\begin{abstract} 
We develop the rudiments of a theory of parametrized $\infty$-operads, including parametrized generalizations of monoidal envelopes, Day convolution, operadic left Kan extensions, results on limits and colimits of algebras, and the symmetric monoidal Yoneda embedding.
\end{abstract}

\date{\today}
\maketitle

\tableofcontents

\section{Introduction}

The goal of this paper is to lay foundations for a theory of \emph{parametrized $\infty$-operads}. To explain the concept, suppose $G$ is a finite group and let us first recall the concept of a $G$-symmetric monoidal $\infty$-category, after Hill--Hopkins \cite{hillhopkins}, Blumberg--Hill \cite{Blumberg2020}, and Bachmann--Hoyois \cite{BACHMANN2021}. Let $\FF_G$ be the category of finite $G$-sets, $\Span(\FF_G)$ the $(2,1)$-category of spans of finite $G$-sets, and $\widehat{\Cat}$ the (huge) $\infty$-category of (large) $\infty$-categories.

\begin{dfn} \label{dfn:g_smc}
A \emph{$G$-symmetric monoidal $\infty$-category} is a product-preserving functor
\[ \cC^{\otimes}: \Span(\FF_G) \to \widehat{\Cat}. \]
\end{dfn}

For example, the $\infty$-category $\Sp^G$ of genuine $G$-spectra extends to a $G$-symmetric monoidal $\infty$-category $(\Sp^G)^{\otimes}$ whose value on $G/H$ is equivalent to $\Sp^H$ and whose covariant functoriality encodes the symmetric monoidal structures on $\{ \Sp^H \}_{H \leq G}$ as well as the \emph{Hill--Hopkins--Ravenel norm functors} $f_{\otimes}: \Sp^H \to \Sp^K$ associated to maps of $G$-orbits $f: G/H \to G/K$ (cf. \cite[\S 9]{BACHMANN2021}). More generally, one can substitute other base $\infty$-categories apart from $\FF_G$ as needed for other applications; in particular, in the motivic context Bachmann and Hoyois work with spans over certain categories of schemes and have extensively investigated the properties of such normed symmetric monoidal $\infty$-categories and their algebras in \cite{BACHMANN2021}.

Just as the theory of symmetric monoidal $\infty$-categories admits a generalization to a theory of $\infty$-operads, we will see that the theory of $G$-symmetric monoidal $\infty$-categories admits a corresponding sort of generalization. Roughly speaking, a \emph{simplicial $G$-operad} should consist of the data of a space of multimorphisms associated to every map of finite $G$-sets, with a composition law then associated to every composite of maps of finite $G$-sets.\footnote{Beware that this isn't the notion of $G$-operad that appears in the work of Blumberg-Hill \cite{MR3406512}.} In fact, just as an $\infty$-operad is really the $\infty$-categorical counterpart of a simplicial colored operad (i.e, a simplicial multicategory), our theory of $G$-$\infty$-operads will encompass both $G$-symmetric monoidal $\infty$-categories and simplicial colored $G$-operads via a suitably defined coherent nerve construction. Abstracting away from the equivariant situation, we will be able to make this idea work under the following hypotheses on our base $\infty$-category, which were first articulated in the first author's work \cite{Exp4} on parametrized stability.

\begin{definition}[{\cite[Def.~4.1]{Exp4}}] Let $\cT$ be a small $\infty$-category. We say that $\cT$ is \emph{orbital} if its finite coproduct completion admits all pullbacks. We say that $\cT$ is \emph{atomic} if it has no non-trivial retracts, so that every map with a left inverse is an equivalence. 
\end{definition}

\begin{exm}
The orbit category $\OO_G$ of a finite group is atomic orbital. Some other examples are enumerated in \cite[Ex.~4.2]{Exp4}.
\end{exm}

\begin{rem}
The condition for an $\infty$-category to be atomic orbital is a highly restrictive one; for example, if $\cT$ is atomic orbital and admits a terminal object, then $\cT$ is equivalent to the nerve of a $1$-category $T$ (\cref{lem:AtomicityEnforcesTruncatedness}).
\end{rem}

At this point, the reader should examine the definition of a simplicial colored $T$-operad (\cref{def:SimplicialColoredTOperad}) to get a conceptual handle on the forthcoming definition of a $\cT$-$\infty$-operad.

\subsection{Summary of results} 

After some preliminaries on the $\cT$-$\infty$-category $\uFinpT$ of pointed finite $\cT$-sets (\cref{def:FiniteTSetsParam}), we give the definition of $\cT$-$\infty$-operad as \cref{dfn:operad} and algebras therein as \cref{dfn:morphism_of_operads}. We  explicate the parametrized Segal condition (\cref{prp:SegalCondition}) and show how the definition of a $\cT$-symmetric monoidal $\cT$-$\infty$-category recovers \cref{dfn:g_smc} (\cref{thm:TwoPresentationsOfTSMCs}). We then study parametrized generalizations of three essential constructions in the theory of $\infty$-operads: monoidal envelopes (\cref{dfn:envelope}), Day convolution (\cref{def:OperadicCoinductionAndDayConvolution}), and operadic left Kan extension (\cref{dfn:operadic_LKE}). Finally, we study $\cT$-(co)limits in $\cT$-$\infty$-categories of $\cO$-algebras, first in the context of a general $\cT$-$\infty$-operad $\cO^{\otimes}$ (\cref{thm:LimitsInAlgebras} and \cref{thm:ColimitsInAlgebras}) and then in the special case of the $\cO^{\otimes} = \uFinpT$ (\cref{thm:IndexedCoproductsAreTensorProducts}), and we establish a $\cT$-symmetric monoidal refinement of the universal property of $\cT$-presheaves (\cref{cor:UMP_presheaves}).

\subsection{Related work}

This paper is part of a larger body of work on parametrized higher category theory and higher algebra \cite{Exp0,Exp1,Exp2,Exp2b,Exp4,nardin}. In particular, all of the conventions, terminology, and notation from \cite{Exp2b} are in force in this paper, and the reader should at least skim the introduction and \S 2 of \cite{Exp2b} before reading this work. Furthermore, the definition of a $\cT$-$\infty$-operad was developed in joint work with Barwick, Dotto, and Glasman circa 2016 and has previously appeared in the first author's thesis \cite[\S 3.1]{nardin}. On the other hand, this paper doesn't otherwise expand on \cite[\S 3]{nardin}; for instance, we will not recapitulate the first author's work on tensor products of $\cT$-presentable $\cT$-$\infty$-categories.

As this paper is intended to play a foundational and supporting role in the literature, we don't discuss many interesting examples or applications here. Horev \cite{horev2019enuine} has used these results in his development of a theory of genuine equivariant factorization homology (see also \cite{hahn2020quivariant}), and he in particular discusses the example of the $G$-$\infty$-operad $\EE_V$ associated to a finite-dimensional real $G$-representation $V$. The second author and Quigley have applied these results in their study of the parametrized Tate construction \cite{quigleyshah_tate} and real cyclotomic spectra \cite{QS21b}. Hilman has introduced similar ideas in his study of parametrized noncommutative motives and equivariant algebraic $K$-theory \cite{hilman2022,hilman2022presentability}.

In a different direction, the theory of $G$-operads in their various guises has a long history that we don't attempt to summarize here; some recent references are \cite{MR3406512,Gutirrez2018,Rubin2021,Bonventre2021,may2021equivariant,guillou2018symmetric}. In terms of the relationship to the $N_{\infty}$-operads of Blumberg--Hill, we discuss $\cT$-indexing systems $\cI$ in our framework in \cref{def:IndexingSystem}, the corresponding commutative $\cT$-$\infty$-operad $\Com^{\otimes}_{\cI}$ in \cref{dfn:IndexingSystemCommOperad}, and how they identify with $G$-indexing systems in the sense of Blumberg--Hill when $\cT = \OO_G$ in \cref{rem:blumberg_hill}. It should be possible to adapt ideas of Hinich from \cite{Hinich} to establish a formal comparison between the $\infty$-category of $\Com_{\cI}$-algebras in our sense and those in the sense of \cite{MR3406512}, but we do not attempt to do this now.

\subsection{Acknowledgements}

We thank Clark Barwick, Emanuele Dotto, and Saul Glasman for valuable discussions in the early stages of this project. J.S. was supported by NSF grant DMS-1547292 and the Deutsche Forschungsgemeinschaft (DFG, German Research Foundation) under Germany’s Excellence Strategy EXC 2044–390685587, Mathematics Münster: Dynamics–Geometry–Structure.

\section{Parametrized \texorpdfstring{$\infty$}{infinity}-operads}

\subsection{First definitions}

We begin by introducing the basic definitions of parametrized higher algebra in parallel to Lurie's development of the foundations of $\infty$-operads \cite[\S 2.1]{HA}. Let $\cT$ be an atomic orbital $\infty$-category, whose objects we refer to as \emph{orbits}, and let $\FinT$ be its finite coproduct completion\footnote{Explicitly, we could take $\FinT \subset \PP(\cT)$ to be the full subcategory spanned by finite coproducts of representables. However, any equivalent choice will suffice.}, which we refer to as the $\infty$-category of \emph{finite $\cT$-sets}.

\begin{definition} For every orbit $V \in \cT$, let
$$\FinT[V] \coloneq (\FinT)^{/V} = \Ar(\FinT) \times_{\FinT} \{V\}$$
(thus fixing a preferred choice of finite coproduct completion of $\cT^{/V}$), and let
$$\FinpT[V] \coloneq (\FinT[V])^{\id_V/} = (\FinT)^{V/ / V}$$
be the $\infty$-category of \emph{finite pointed $\cT^{/V}$-sets}.
\end{definition}

Using that $\cT$ is orbital, we could then define the \emph{$\cT$-$\infty$-category of finite $\cT$-sets} $\uFinT$ as the full $\cT$-subcategory of $\uSpcT$ spanned by the finite $\cT^{/V}$-sets in each fiber $(\uSpcT)_V$ over an orbit $V$, so that as a cocartesian fibration, $\uFinT$ is classified by the assignment $\goesto{V}{(\FinT)^{/V}}$ with functoriality that given by pullback. Similarly, we could define a pointed variant $\uFinpT$ as the full $\cT$-subcategory of $\uSpcpT$ classified by the assignment $\goesto{V}{(\FinT)^{V/ /V}}$.

However, although conceptually transparent, these definitions of $\uFinT$ and $\uFinpT$ are ill-suited to writing down arbitrary morphisms that may interpolate between different fibers. Instead, we will follow the first author's work in \cite[\S 4]{Exp4} and instead define $\uFinT$ and $\uFinpT$ as certain $\infty$-categories of spans, along the lines of the construction of the dual cocartesian fibration in \cite{BGN} as well as the span description of finite pointed sets in terms of finite sets and partially defined maps (cf. \cite[4.11]{Exp4}). 

\begin{definition} \label{def:FiniteTSetsParam} Let
$$\ArOT \coloneqq \Ar(\FinT) \times_{\FinT} \cT,$$
so that the functor $\ev_1: \ArOT \to \cT$ given by evaluation at the target is a cartesian fibration classified by $\goesto{V}{(\FinT)^{/V}}$. Labeling an arbitrary morphism $[\phi: f \to g]$ of $\ArOT$ as
\[ \begin{tikzcd}[row sep=4ex, column sep=4ex, text height=1.5ex, text depth=0.25ex]
U \ar{r}{h} \ar{d}{f} & X \ar{d}{g} \\
V \ar{r}{k} & Y,
\end{tikzcd} \]
we define wide subcategories
$$(\ArOT)^{\tdeg}, \: (\ArOT)^{\si}, \: (\ArOT)^{\cart} \subset \ArOT$$
as containing those morphisms $\phi$ such that $k$ is degenerate, $U \to X \times_Y V$ is a summand inclusion, and $U \to X \times_Y V$ is an equivalence, respectively.\footnote{Note that the ``target degenerate'' morphisms are a subclass of $\ev_1$-cocartesian edges (for which more generally the map $k$ is an equivalence) and the ``cartesian'' morphisms are exactly the $\ev_1$-cartesian edges.} Then the triples
\[  (\ArOT; (\ArOT)^{\cart}, (\ArOT)^{\tdeg}) \textrm{\quad and \quad} (\ArOT; (\ArOT)^{\si}, (\ArOT)^{\tdeg}) \]
are adequate in the sense of \cite[5.2]{M1}.\footnote{Note that we have swapped the order of the wide subcategories from that of \cite{M1}, so that the first subcategory will indicate the backward facing arrows and the second will indicate the forward facing arrows when forming the span $\infty$-category.} Consequently, we may form the associated span $\infty$-categories\footnote{In \cite{M1} and \cite{Exp4}, the term ``effective Burnside $\infty$-category'' $A^{\eff}$ is used as a synonym for ``span $\infty$-category''.}
\[  \uFinT \coloneqq \Span(\ArOT; (\ArOT)^{\cart}, (\ArOT)^{\tdeg}) \textrm{\quad and \quad} \uFinpT \coloneqq \Span(\ArOT; (\ArOT)^{\si}, (\ArOT)^{\tdeg}). \]

We regard $\uFinT$ and $\uFinpT$ as $\cT$-$\infty$-categories via the structure map $\ev_1$ given by evaluation at the target, so that a morphism $\psi$ (for either $\infty$-category)
\[ \begin{tikzcd}[row sep=4ex, column sep=4ex, text height=1.5ex, text depth=0.25ex]
U \ar{d} & Z \ar{l} \ar{d} \ar{r}{m} & X \ar{d} \\
V & Y \ar{l} \ar{r}{=} & Y.
\end{tikzcd} \]
is $\ev_1$-cocartesian if and only if $m: Z \to X$ is an equivalence and $Z \to U \times_V Y$ is an equivalence (cf. \cite[Lem.~4.9 and Def.~4.12]{Exp4}). The canonical inclusion $\uFinT \subset \uFinpT$ of span $\infty$-categories is thus the inclusion of a $\cT$-subcategory. We also have an `identity' cocartesian section $I: \cT^\op \to \uFinT$ that sends $V$ to $[V=V]$.
\end{definition}

\begin{definition} \label{def:InertActiveEdges}
In the notation of \cref{def:FiniteTSetsParam}, we declare a morphism $\psi$ in $\uFinpT$ to be \emph{inert} if $m: Z \to X$ is an equivalence and \emph{active} if $Z \to U \times_V Y$ is an equivalence. Note that a morphism $\psi$ in $\uFinpT$ is both inert and active if and only if $\psi$ is $\ev_1$-cocartesian.
\end{definition}



\begin{rem} \label{rem:UnwindingDefOfFiniteTSets}
Note that $\uFinT$ is by definition the dual cocartesian fibration to $\ArOT$ in the sense of \cite[Def.~3.4]{BGN}. As such, for any orbit $V$ we have an equivalence $$(\ArOT)_V = \FinT[V] \xto{\simeq} \Span(\FinT[V]; (\FinT[V])^{\simeq}, \FinT[V]) = (\uFinT)_V$$ implemented by inclusion.

We next describe $\uFinpT$. Let $\FinT^{\si}$ denote the wide subcategory on the summand inclusions in $\FinT$, so $(\ArOT)^{\si}_V = \FinT[V]^{\si}$. As was noted in \cite[Lem.~4.14]{Exp4}, for any orbit $V$ we have an equivalence
$$(\uFinpT)_V = \Span(\FinT[V]; \FinT[V]^{\si}, \FinT[V]) \xto{\simeq} \FinpT[V],$$
under which an object $[U \xto{f} V]$ is sent to $[U \sqcup V \xtolong{f \sqcup \id}{.8} V]$ pointed at $V$, and a span
$$U \xot{\alpha} W \xto{\beta} U',$$
with $\alpha$ given by the summand inclusion $W \subset W \sqcup W' \simeq U$, is sent to the pointed map
$$U \sqcup V \xto{\gamma} U' \sqcup V$$
with $\gamma|_{W} = \beta$ and $\gamma|_{W'} = {\const}_{V}$. Consequently, we will often refer to an object $f=[U \to V]$ of $\uFinpT$ as $f_+ = [U_+ \to V]$ to emphasize the implicit presence of the basepoint. We will also denote the canonical inclusion $\uFinT \subset \uFinpT$ of span $\infty$-categories by
\[ (-)_+: \into{\uFinT}{\uFinpT} \]
and refer to this as the \emph{pointing} $T$-functor. By \cite[Lem.~4.14]{Exp4}, $(-)_+$ has a `forgetful' right $T$-adjoint which sends $[U_+ \to V]$ to $[U \sqcup V \to V]$. Note also that a morphism $\psi$ in $\uFinpT$ is active if and only if it is in the image of $(-)_+$. 
\end{rem}

\begin{remark} \label{rem:FiberwiseInertActive}
For an orbit $V \in \cT$, we obtain from \cref{def:InertActiveEdges} a definition for inert and active edges in $(\uFinpT)_V$ by restriction to the fiber. Under the equivalence $(\uFinpT)_V \simeq \FinpT[V]$ of \cref{rem:UnwindingDefOfFiniteTSets}, a pointed map $f: U \sqcup V \to U' \sqcup V$ is then inert if and only if its pullback along $U' \subset U' \sqcup V$ is an equivalence, and is active if and only if $f \simeq g_+$ for some $g: U \to U'$ in $\FinT[V]$.
\end{remark}

\begin{definition} Suppose $f_+ = [U_+ \to V]$ is an object of $\uFinpT$. Let $\Orbit(U)$ be the set of orbits of $U$, so that we have an equivalence
\[ U \simeq \coprod_{W \in \Orbit(U)} W \]
in $\FinT$ with each $W$ an object in $\cT$. Given $W \in \Orbit(U)$, the \emph{characteristic morphism}
\[ \chi_{[W \subset U]}: f_+ \to I(W)_+ \]
is defined to be
\[ \begin{tikzcd}[row sep=4ex, column sep=4ex, text height=1.5ex, text depth=0.25ex]
U \ar{d}{f} & W \ar{l} \ar{d}{=} \ar{r}{=} & W \ar{d}{=} \\
V & W \ar{l} \ar{r}{=} & W.
\end{tikzcd} \]
Here we make essential use of our assumption that $\cT$ is atomic to ensure that $W \to U \times_V W$ is a summand inclusion. Clearly, $\chi_{[W \subset U]}$ is inert.
\end{definition}

\begin{dfn} \label{dfn:operad} A \emph{$\cT$-$\infty$-operad} is a pair $(\cC^\otimes, p)$ consisting of a $\cT$-$\infty$-category $\cC^\otimes$ along with a $\cT$-functor $p: \cC^\otimes \to \uFinpT$, which is a categorical fibration and satisfies the following additional conditions:
\begin{enumerate} \item For every inert morphism $\psi: f_+ \to g_+$ of $\uFinpT$ and every object $x \in \cC^\otimes_{f_+}$, there is a $p$-cocartesian edge $x \to y$ in $\cC^\otimes$ covering $\psi$.
\item For any object $f_+ = \left[ U_+ \to V \right]$ of $\uFinpT$, the $p$-cocartesian edges lying over the characteristic morphisms 
\[ \left\{ \chi_{[W \subset U]}: f_+ \to I(W)_+ \;\; |\;\; W \in \Orbit(U) \right\}  \]
together induce an equivalence
\[ \prod_{W \in \Orbit(U)} (\chi_{\left[W \subset U \right]})_! : \cC^\otimes_{f_+} \xto{\sim} \prod_{W \in \Orbit(U)} \cC^{\otimes}_{I(W)_+}.  \]
\item For any morphism
\[ \psi: f_+ = \left[ U_+ \to V \right] \to g_+ = \left[ U'_+ \to V' \right] \]
of $\uFinpT$, objects $x \in \cC^\otimes_{f_+}$ and $y \in \cC^\otimes_{g_+}$, and any choice of $p$-cocartesian edges
\[ \left\{ y \to y_W \;\; |\;\; W \in \Orbit(U') \right\} \]
lying over the characteristic morphisms
\[ \left\{ \chi_{[W \subset U']}: g_+ \to I(W)_+ \;\; |\;\; W \in \Orbit(U') \right\},  \]
the induced map
\[ \Map^{\psi}_{\cC^\otimes}(x,y) \xto{\sim} \prod_{W \in \Orbit(U')} \Map^{\chi_{[W \subset U']} \circ \psi}_{\cC^\otimes}(x, y_W ) \]
is an equivalence.
\end{enumerate}

We will typically omit the structure map $p$ and simply refer to $\cC^\otimes$ as a $\cT$-$\infty$-operad. Given a $\cT$-$\infty$-operad $\cC^\otimes$, its \emph{underlying $\cT$-$\infty$-category} is the fiber product
\[ \cC \coloneq \cT^{\op} \times_{I(-)_+, \uFinpT} \cC^\otimes. \]
\end{dfn}

\begin{dfn} Suppose $(\cC^\otimes,p)$ is a $\cT$-$\infty$-operad. Then an edge of $\cC^\otimes$ is \emph{inert} if it is $p$-cocartesian over an inert edge of $\uFinpT$, and is \emph{active} if it factors as a $p$-cocartesian edge followed by an edge lying over a fiberwise active edge in $\uFinpT$. We let $\cC^\otimes_{\inert}$ be the wide $\cT$-subcategory of $\cC^\otimes$ on the inert edges, and $\cC^\otimes_{\act}$ the wide $\cT$-subcategory of $\cC^\otimes$ on the active edges.
\end{dfn}

\begin{remark} \label{rem:BasepointIndependenceInOperadDefForObjects} Let $\cC^\otimes$ be a $\cT$-$\infty$-operad and $U \in \FinT$. Note that for \emph{any} orbit $V$ and morphism $f: U \to V$, we have an equivalence
$$\cC^\otimes_{f_+} \simeq \cC_U = \prod_{W \in \Orbit(U)} \cC_W.$$
We will often write objects $x \in \cC^\otimes_{f_+}$ as tuples $(x_W)$.
\end{remark}

\begin{remark}[Simplified condition on mapping spaces] \label{rem:alternateOperadDef}
In \cref{dfn:operad}, in view of the inert-fiberwise active factorization system on a $\cT$-$\infty$-operad (\cref{exm:InertActiveFactorizationSystem}) we may replace (3) by the following apparently weaker condition:

(3\textquotesingle) \: Let $\alpha: [U \xto{f} V] \to [U' \xto{g} V]$ be a morphism in $\FinT[V]$, which defines an active edge $\alpha_+$ in $(\uFinpT)_V$. Let $x \in \cC^{\otimes}_{f_+}$, $y \in \cC^\otimes_{g_+}$ be objects, and for each $W \in \Orbit(U')$ let $y \to y_W$ be a $p$-cocartesian edge lifting the characteristic morphism $\chi_{[W \subset U']}$. For every $W \in \Orbit(U')$ we have a commutative square
\[ \begin{tikzcd}[row sep=4ex, column sep=6ex, text height=1.5ex, text depth=0.5ex]
\brac{U_+ \to V} \ar{r}{\alpha_+} \ar{d}{\rho_W} & \brac{U'_+ \to V} \ar{d}{\chi_{\brac{W \subset U'}}} \\\brac{(U \times_{U'} W)_+ \to W} \ar{r}{(\alpha_W)_+} & \brac{W_+ \to W}
\end{tikzcd} \]
where $\rho_W$ is the inert edge corresponding to the summand inclusion $U \times_{U'} W \to U \times_V W $ and $\alpha_W: U \times_{U'} W \to W$ is the pullback of $\alpha: U \to U'$ along $W \subset U'$. (Note that the lower composition is the inert-fiberwise active factorization of the upper composition $\chi_{\brac{W \subset U'}} \circ \alpha_+$.) Let
\[ \left\{ x \to x_W \;\; |\;\; W \in \Orbit(U') \right\} \]
be any choice of $p$-cocartesian edges lying over the morphisms $\rho_W$. Then the induced map
\begin{equation} \label{eq__rem:alternateOperadDef}
 \Map_{\cC^\otimes}^{\alpha_+}(x,y) \xto{\sim} \prod_{W \in \Orbit(U')} \Map_{\cC^\otimes}^{(\alpha_W)_+}(x_W, y_W)
 \end{equation}
is an equivalence.
\end{remark}

\begin{remark}[Spaces of multimorphisms and operadic composition] \label{rem:OperadicCompositionUnwinding} Suppose $\cC^\otimes$ is a $\cT$-$\infty$-operad, $\alpha: U \to U'$ is a morphism in $\FinT$, and $x \in \cC_U$, $y \in \cC_{U'}$ are tuples of objects in $\cC$. For every $W \in \Orbit(U')$, let
$$\alpha_W: U_W = U \times_{U'} W \to W$$
be the pullback of $\alpha$ along the summand inclusion $W \subset U'$. Consider the component $y_W$ of $y$ as an object in $\cC^\otimes_{I(W)_+}$ and the sub-tuple $x_W \in \cC_{U_W}$ of $x$ as an object in $\cC^\otimes_{(\alpha_W)_+}$. Let
\[ \Mul_{\cC}^{\alpha}(x,y) \coloneq \prod_{W \in \Orbit(U')} \Map_{\cC^\otimes}^{(\alpha_{W})_+}(x_W, y_W) \]
be the space of \emph{$(\alpha; x, y)$-multimorphisms} encoded by $\cC^\otimes$. Then for any choice of map $U' \to V$ in $\FinT$ down to an orbit $V$, we have the canonical equivalence
$$\Map_{\cC^\otimes}^{\alpha_+}(x,y) \simeq \Mul_{\cC}^{\alpha}(x,y)$$
of \eqref{eq__rem:alternateOperadDef} (compare \cref{rem:BasepointIndependenceInOperadDefForObjects}).

These spaces of multimorphisms are interrelated by the structure of $\cC^\otimes$. For instance, for every composite morphism $U_0 \xto{\alpha} U_1 \xto{\beta} U_2$ in $\FinT$ and $x_i \in \cC_{U_i}$, $i \in \{ 0,1,2 \}$, any choice of map $\rho: U_2 \to V$ to an orbit $V$ yields a map
\[ \circ: \Mul_{\cC}^{\alpha}(x_0,x_1) \times \Mul_{\cC}^{\beta}(x_1,x_2) \to \Mul_{\cC}^{\beta \circ \alpha}(x_0,x_2) \]
defined by the composition in $\cC^\otimes$, and one may check that this map is independent of the choice of $\rho$. Likewise, for every composition of pullback squares in $\FinT$
\[ \begin{tikzcd}
X \ar{r}{f^\ast \alpha} \ar{d} & X' \ar{r} \ar{d} & W \ar{d}{f} \\
U \ar{r}{\alpha} & U' \ar{r} & V
\end{tikzcd} \]
with $V, W$ orbits, and objects $x \in \cC_U$, $y \in \cC_{U'}$, one has a base-change map
\[ f^\ast: \Mul^{\alpha}_{\cC}(x,y) \to \Mul^{f^\ast \alpha}_{\cC}(f^\ast x, f^\ast y) \]
induced by the cocartesian pushforward in $\cC^\otimes$ along $f$ in $\cT^\op$. Note that these maps extend the functoriality on the underlying $\cT$-$\infty$-category $\cC$. 

Altogether, these maps satisfy homotopy coherent unitality, associativity, and base-change compatibility constraints as encapsulated by $\cC^\otimes$.
\end{remark}

\subsection{Morphisms of operads}

We next introduce morphisms of $\cT$-$\infty$-operads and algebras over $\cT$-$\infty$-operads.

\begin{dfn} \label{dfn:morphism_of_operads}
Suppose $\cC^\otimes$ and $\cD^\otimes$ are two $\cT$-$\infty$-operads. A \emph{morphism of $\cT$-$\infty$-operads} is a $\cT$-functor
\[ A: \cC^\otimes \to \cD^\otimes \]
over $\uFinpT$ that carries inert morphisms to inert morphisms. Conceptually, $A$ is a \emph{$\cC$-algebra valued in $\cD$}.

If $A$ is moreover a categorical fibration, then we call $A$ a \emph{fibration of $\cT$-$\infty$-operads}. Given fibrations of $\cT$-$\infty$-operads $p: \cC^\otimes \to \cO^\otimes$ and $q: \cD^\otimes \to \cO^\otimes$, we let
\[ \Alg_{\cO,\cT}(\cC,\cD) \]
denote the full subcategory of $\Fun_{/\cO^\otimes}(\cC^\otimes,\cD^\otimes)$ spanned by the morphisms of $\cT$-$\infty$-operads, and
\[ \underline{\Alg}_{\cO,\cT}(\cC,\cD) \]
the corresponding full $\cT$-subcategory of the $\cT$-$\infty$-category $\underline{\Fun}_{/\cO^\otimes,\cT}(\cC^\otimes,\cD^\otimes)$ \cite[Notn.~4.7]{Exp2b}.\footnote{For this definition to be sensible, note that for any map $V \to W$ in $\cT$, the pullback of a $\cT^{/W}$-$\infty$-operad along the induced functor $(\cT^{/V})^{\op} \to (\cT^{/W})^{\op}$ is again a $\cT^{/V}$-$\infty$-operad, and likewise for morphisms of $\cT^{/W}$-$\infty$-operads.}

If $p$ is the identity on $\cO^\otimes$, then we will also denote $\Alg_{\cO,\cT}(\cC,\cD)$ as $\Alg_{\cO,\cT}(\cD)$. If $\cO^\otimes = \uFinpT$, then we will also denote $\Alg_{\cO,\cT}(\cC,\cD)$ as $\Alg_{\cT}(\cC,\cD)$. Combining these two cases, if $\cC^\otimes = \cO^\otimes = \uFinpT$, then we will also denote $\Alg_{\cO,\cT}(\cC,\cD)$ as $\CAlg_{\cT}(\cD)$, the $\infty$-category of \emph{$\cT$-commutative algebras} in $\cD$.
\end{dfn}

\begin{wrn} In the case $\cT=\ast$, our notation for $\infty$-categories of algebras conflicts with that of Lurie in \cite[Def.~2.1.3.1]{HA}.
\end{wrn}

\begin{dfn} Suppose $p: \cC^\otimes \to \cO^\otimes$ is a fibration of $\cT$-$\infty$-operads in which $p$ is moreover a cocartesian fibration. In this case, we call $\cC^\otimes$ a \emph{$\cO$-monoidal} $\cT$-$\infty$-category. If $\cO^\otimes = \uFinpT$, we also call $\cC^\otimes$ a \emph{$\cT$-symmetric monoidal} $\cT$-$\infty$-category.\footnote{We may also write ``$\cT$-symmetric monoidal $\infty$-category'' for this notion since there is no potential for ambiguity.} We also refer to $\cC$ as $\cO$-monoidal if the additional structure $(\cC^\otimes,p)$ is understood from context.
\end{dfn}

\begin{ntn} \label{ntn:norm_functors}
Let $\cC^{\otimes}$ be an $\cO$-monoidal $\cT$-$\infty$-category. For an active morphism $f: x \to y$ in $\cO^{\otimes}$, we typically denote the cocartesian pushforward functor associated to $f$ by $f_{\otimes}: \cC^{\otimes}_x \to \cC^{\otimes}_y$ and refer to it as the \emph{norm functor} for $f$. If $\cO^{\otimes} = \uFinpT$, then for any morphism $f: U \to V$ of finite $\cT$-sets with $V$ an orbit, we have a norm functor $f_{\otimes}: \cC_U \to \cC_V$ associated to $f_+: [U_+ \ra V] \to [V_+ \ra V]$. More generally, if $V$ is a finite $\cT$-set with orbit decomposition $\coprod_{i=1}^n V_i$ so that $f = \coprod_{i=1}^n (f_i: U_i \to V_i)$, then we let $f_{\otimes}$ be the product of the functors $\{ (f_i)_{\otimes} \}_{i=1}^n$. (We will also describe in \cref{subsec:bigOperads} how to dispense with the orbit restriction in the formalism by passing to `big' $\cT$-$\infty$-operads.)
\end{ntn}

\begin{dfn} Given two $\cO$-monoidal $\cT$-$\infty$-categories $p,q: \cC^\otimes, \cD^\otimes \to \cO^\otimes$, a $\cT$-functor $F: \cC^\otimes \to \cD^\otimes$ is \emph{lax $\cO$-monoidal} if it is a morphism of $\cT$-$\infty$-operads, and is \emph{(strict) $\cO$-monoidal} if it carries $p$-cocartesian edges to $q$-cocartesian edges. We let
\[ \Fun_{\cO,\cT}^\otimes(\cC,\cD) \]
denote the subcategory of $\Fun_{/\cO^\otimes}(\cC^\otimes,\cD^\otimes)$ spanned by the $\cO$-monoidal $\cT$-functors, and
\[ \underline{\Fun}_{\cO,\cT}^\otimes(\cC,\cD) \]
the corresponding $\cT$-subcategory of $\underline{\Fun}_{/\cO^\otimes,\cT}(\cC^\otimes,\cD^\otimes)$. We will also drop $\cO$ from the notation if $\cO^\otimes = \uFinpT$ and speak of lax and strict $\cT$-symmetric monoidal $\cT$-functors.
\end{dfn}

In the situation of a cocartesian fibration $p: \cC^\otimes \to \cO^\otimes$ over a $\cT$-$\infty$-operad $\cO^\otimes$, we have the following simplification of the conditions for $\cC^\otimes$ to be a $\cT$-$\infty$-operad (and hence $\cO$-monoidal).

\begin{proposition} \label{prop:SimplifyOMonoidalDef} Let $(\cO^\otimes,q)$ be a $\cT$-$\infty$-operad and let $p: \cC^\otimes \to \cO^\otimes$ be a cocartesian fibration of $\cT$-$\infty$-categories. Then $(\cC^\otimes, q \circ p)$ is a $\cT$-$\infty$-operad if and only if for every $f_+ = [U_+ \to V] \in \uFinpT$ and $x \in \cO^\otimes_{f_+}$, the inert edges $\{ x \to x_W \: | \: W \in \Orbit(U) \}$ in $\cO^\otimes$ together induce an equivalence
$$\cC^\otimes_x \xto{\sim} \prod_{W \in \Orbit(U)} \cC^\otimes_{x_W}.$$
\end{proposition}
\begin{proof} The proof is exactly analogous to that of \cite[Prop.~2.1.2.12]{HA}, so we will omit it.
\end{proof}

Lastly, we state the evident notions of $\cT$-suboperad and $\cO$-monoidal $\cT$-subcategory.

\begin{definition}
Let $(\cC^\otimes, p)$ be a $\cT$-$\infty$-operad and let $\cD^\otimes$ be a $\cT$-subcategory of $\cC^\otimes$ with inclusion $\cT$-functor $i$. We say that $\cD^\otimes$ is a \emph{$\cT$-suboperad} of $\cC^\otimes$ if $p \circ i$ exhibits $\cD^\otimes$ as a $\cT$-$\infty$-operad and $i$ is a morphism of $\cT$-$\infty$-operads. If $\cC^\otimes$ is moreover an $\cO$-monoidal $\cT$-$\infty$-category via $q: \cC^\otimes \to \cO^\otimes$, then $\cD^\otimes$ is a \emph{$\cO$-monoidal $\cT$-subcategory} of $\cC^\otimes$ if $\cD^\otimes \subset \cC^\otimes$ is stable under $q$-cocartesian edges, so that $q \circ i$ exhibits $\cD^\otimes$ as an $\cO$-monoidal $\cT$-$\infty$-category and $i$ is an $\cO$-monoidal functor.
\end{definition}


\subsection{Parametrized Segal condition}

We next want to interpret condition (2) of \cref{dfn:operad} as an equivalence of $\cT^{/V}$-$\infty$-categories (i.e., a \emph{$\cT$-Segal condition}). First, we extend our notation for the $\cT$-fibers of a $\cT$-functor.

\begin{ntn} Let $F: \cX \to \cC$ be a $\cT$-functor and let $\sigma: \Delta^n \to \cC$ be a $n$-simplex of $\cC$. Define the \emph{$\cT$-fiber of $\cX$ over $\sigma$} to be
\[ \cX_{\underline{\sigma}} \coloneq \Delta^n \times_{\sigma,\cC,\ev_0} \Ar^\cocart(\cC) \times_{\ev_1,\cC,F} \cX. \] 
\end{ntn}


\begin{cnstr} \label{cnstr:EdgeFunctoriality} For any $\cT$-$\infty$-category $\cX$ and edge $f: x \to y$ in $\cX$, we can construct a $\cT$-functor
\[ \phi: \Delta^1 \times \underline{y} = \Delta^1 \times (\Delta^0 \times_{y, \cX, \ev_0} \Ar^\cocart(\cX)) \to \Delta^1 \times_{f, \cX, \ev_0} \Ar^\cocart(\cX) \]
which fits into the commutative diagram
\[ \begin{tikzcd}[row sep=2em, column sep=2em]
\{0\} \times \underline{y} \ar{r}{f^\ast} \ar{d} & \underline{x} \ar{d} \\
\Delta^1 \times \underline{y} \ar{r}{\phi} & \Delta^1 \times_{f, \cX, \ev_0} \Ar^\cocart(\cX) \\
\{1\} \times \underline{y} \ar{u} \ar{r}{=} & \underline{y}, \ar{u}
\end{tikzcd} \]
where $f^*: \underline{y} \to \underline{x}$ is the $\cT$-functor defined in \cite[Rem.~12.11]{Exp2}, which sends a cocartesian edge $[e: y \ra z]$ to the cocartesian edge $[f^*(e): x \ra z']$ given by the factorization of $[e \circ f]$ as the composite of $f^*(e)$ and a fiberwise edge $\phi(e)_1$.

Explicitly, let $h: \Delta^1 \times \Delta^1 \to \cX$ be given by
\[ \begin{tikzcd}[row sep=2em, column sep=2em]
x \ar{r}{f} \ar{d}{f} & y \ar{d}{=} \\
y \ar{r}{=} & y
\end{tikzcd} \] 
and let 
\[ \cM = \Delta^1 \times_{h, \widetilde{\Fun}_{\Delta^1}(\Delta^1 \times \Delta^1, \cX \times \Delta^1),\ev_0} \widetilde{\Fun}_{\Delta^1}(\Delta^1 \times \Delta^1, \Ar^\cocart(\cX) \times \Delta^1) \times_{\ev_1,\widetilde{\Fun}_{\Delta^1}(\Delta^1 \times \Delta^1, \cT^\op \times \Delta^1),I} (\cT^\op \times \Delta^1), \]
where $I$ denotes the identity section, so that $\cM \to \Delta^1 \times \cT^\op$ is a $\cT$-correspondence with
\begin{align*}
\cM_0 & = \{ 0\} \times_{f,\Ar(\cX), \ev_0}  \Ar(\Ar^{\cocart}(\cX)) \times_{\ev_1,\Ar(\cT^{\op}),I} \cT^{\op}, \\ 
\cM_1 & = \{ 0\} \times_{\id_y,\Ar(\cX),\ev_0}  \Ar(\Ar^{\cocart}(\cX)) \times_{\ev_1,\Ar(\cT^{\op}),I} \cT^{\op}.
\end{align*}
We have a zig-zag of $\cT$-functors over $\Delta^1 \times \cT^\op$
\[ \begin{tikzcd}[row sep=2em, column sep=3em]
\underline{y} \ar[bend left,dotted]{r}{\tau} \times \Delta^1 & \cM \ar{l}[swap]{\pi} \ar{r}{\rho} & \Delta^1 \times_{f,\cX, \ev_0} \Ar^\cocart(\cX) 
\end{tikzcd} \]
where $\pi$ restricts to the trivial fibrations $\cM_0 \to \underline{y} \times \{0\}$, $\cM_1 \to \underline{y} \times \{1\}$ of \cite[Lem.~12.10]{Exp2} and $\rho$ restricts to $\cM_0 \to \underline{x}$, $\cM_1 \to \underline{y}$. Thus $\pi$ is a trivial fibration and we may choose a section $\tau$ which fixes $\underline{y} \times \{1\} \subset \cM_1$. Then we let $\phi = \rho \circ \tau$.
\end{cnstr}


\begin{thm} \label{prp:SegalCondition} Let $\cC^\otimes \to \cO^\otimes$ be a fibration of $\cT$-$\infty$-operads. Let $x \in \cO^\otimes$ be an object over $[f_+: U_+ \rightarrow V] \in \uFinpT$. Let $U \simeq U_1 \coprod ... \coprod U_n$ be an orbit decomposition, let $f_i: U_i \to V$ denote the induced morphisms, and let $e_i: x \to x_i$ be inert edges in $\cO^\otimes$ lifting the characteristic morphisms $\chi_{[U_i \subset U]}$ in $\uFinpT$. Then we have an equivalence of $\cT^{/V}$-$\infty$-categories
\[ \cC^\otimes_{\underline{x}} \simeq \prod_f \left( \coprod_{1 \leq i \leq n} \cC_{\underline{x_i}} \right) \simeq \prod_{1 \leq i \leq n} \left( \prod_{f_i} \cC_{\underline{x_i}} \right). \footnote{In this expression, each $\cC_{\underline{x_i}}$ is a $\cT^{/{U_i}}$-$\infty$-category, their coproduct is a cocartesian fibration over $(\cT^{/U})^\op = \underline{U} \simeq \coprod_{1 \leq i \leq n} \underline{U_i}$, the righthand product is taken in $\cT^{/V}$-$\infty$-categories, and the indexed product $\prod_f$ denotes the right adjoint to pullback along the induced functor $\underline{U} \to \underline{V}$.} \]
\end{thm}
\begin{proof} Let
 \[ h_i: \Delta^1 \times \underline{x_i} \to \Delta^1 \times_{e_i, \cO^\otimes} \Ar^\cocart(\cO^\otimes) \to \cO^\otimes \]
 be the homotopy associated to the edge $e_i$ as defined as in \cref{cnstr:EdgeFunctoriality}. Because $h_i$ lands in $\cO^\otimes_{\inert}$, the pullback $\cC^\otimes \times_{\cO^\otimes} (\Delta^1 \times \underline{x_i}) \to \Delta^1 \times \underline{U_i}$ is a cocartesian fibration. This corresponds to a $\cT^{/U_i}$-functor $\rho^i: f_i^\ast (\cC^\otimes_{\underline{x}}) \to \cC_{\underline{x_i}}$. Taking the coproduct of the $\rho^i$ and taking the adjoint of that, we get a comparison $\cT^{/V}$-functor
\[ \rho: \cC^\otimes_{\underline{x}} \to \prod_f \left( \coprod_{1 \leq i \leq n} \cC_{\underline{x_i}} \right). \]
We claim that $\rho$ is a equivalence of $\cT^{/V}$-$\infty$-categories. We will check that for every object $[g: V' \ra V]$, the fiber $\rho_{g}$ is an equivalence. Consider the pullback square
\[ \begin{tikzcd}[row sep=2em, column sep=2em]
U' \ar{r}{g'} \ar{d}{f'} & U \ar{d}{f} \\
V' \ar{r}{g} & V.
\end{tikzcd} \]
Let $[U_+ \rightarrow V] \to [U'_+ \rightarrow V']$ be the corresponding inert morphism in $\uFinpT$ and let $x \to x'$ be an inert lift of that morphism to $\cO^\otimes$. Also let $U' \simeq U'_1 \coprod ... \coprod U'_m$ be an orbit decomposition and let $e_j': x' \to x'_j$ be inert morphisms lifting the characteristic morphisms $\chi_{[U'_j \subset U']}$.

Note that
\[  (\cC^\otimes_{\underline{x}})_g \simeq (\cC^\otimes_{\underline{x'}})_{\id_{V'}} \simeq \cC^\otimes_{x'} \]
and
\begin{align*}
\left( \prod_f \left( \coprod_{1 \leq i \leq n} \cC_{\underline{x_i}} \right) \right)_g & \simeq \left( \prod_{f'} \left( (g')^\ast \left( \coprod_{1 \leq i \leq n} \cC_{\underline{x_i}} \right) \right) \right)_{\id_{V'}} \\
& \simeq \left( (g')^\ast \left( \coprod_{1 \leq i \leq n} \cC_{\underline{x_i}} \right) \right)_{\id_{U'}} \\
& \simeq \left( \coprod_{1 \leq j \leq m} \cC_{\underline{x'_j}} \right)_{\id_{U'}} \\
& \simeq \coprod_{1 \leq j \leq m} \cC_{x'_j}
\end{align*}

A diagram chase then shows that the functor $\rho_g: \cC^\otimes_{x'} \to \coprod_{1 \leq j \leq m} \cC_{x'_j}$ implements the equivalence of condition (2) in \cref{dfn:operad}.
\end{proof}

\begin{cor}[$\cT$-Segal condition] \label{cor:SymmetricMonoidalSegalCondition} Let $\cC^\otimes \to \uFinpT$ be a $\cT$-$\infty$-operad. Then for every object $[U_+ \xto{f_+} V]$ in $\uFinpT$, we have an equivalence of $\cT^{/V}$-$\infty$-categories
\[ \cC^\otimes_{\underline{f_+}} \simeq \underline{\Fun}_{\cT^{/V}}(\underline{U}, \cC_{\underline{V}}). \]
\end{cor}
\begin{proof} In view of \cref{prp:SegalCondition}, we only need to note that $\underline{\Fun}_{\cT^{/V}}(\underline{U},-) \simeq \prod_f f^\ast$ as endofunctors of $\Cat_{\cT^{/V}}$ and that for an orbit decomposition $U \simeq U_1 \coprod ... \coprod U_n$, $\coprod_{1 \leq i \leq n} \cC_{\underline{U_i}} \simeq \cC_{\underline{U}} \simeq \underline{U} \times_{\underline{V}} \cC_{\underline{V}}$.
\end{proof}

\begin{exm} \label{exm:normParametrizedFunctor}
Let $\cC^{\otimes}$ be a $\cT$-symmetric monoidal $\cT$-$\infty$-category. Then for every morphism $f: U \to V$ of finite $\cT$-sets, the norm functor $f_{\otimes}: \cC_{U} \to \cC_{V}$ of \cref{ntn:norm_functors} canonically refines to a norm $\cT^{/V}$-functor $\underline{\Fun}_{\cT^{/V}}(\underline{U}, \cC_{\underline{V}}) \to \cC_{\underline{V}}$. Indeed, if $V$ is an orbit this is encoded by the cocartesian fibration $\cC^{\otimes} \to \uFinpT$ in view of \cref{cor:SymmetricMonoidalSegalCondition}, and one extends to general $V$ by taking coproducts.
\end{exm}

We also have a reformulation of condition (3) in \cref{dfn:operad} (or rather (3\textquotesingle) in \cref{rem:alternateOperadDef}), whose proof is the same as that of \cref{prp:SegalCondition}. Recall the notion of $\cT$-mapping spaces from \cite[\S 11]{Exp2}.

\begin{ntn} Let $p: \cX \to \cB$ be a $\cT$-fibration, let $\alpha: a \to b$ be a morphism in a fiber $\cB_V$, and let $x, y \in \cX$ so that $p(x) = a$ and $p(y) = b$. Then we define as a pullback of $\cT^{/V}$-spaces
\[ \underline{\Map}^{\alpha}_{\cX}(x,y) \coloneq \underline{\alpha} \times_{\underline{\Map}_{\cB}(a,b)} \underline{\Map}_{\cX}(x,y). \]
\end{ntn}

\begin{prp} \label{prop:MultiMappingTSpace} Let $\cC^\otimes$ be a $\cT$-$\infty$-operad and let notation be as in \cref{rem:alternateOperadDef}. Then we have an equivalence of $\cT^{/V}$-spaces
\[ \underline{\Map}_{\cC^\otimes}^{\alpha_+}(x,y) \xto{\sim} \prod_g \left( \coprod_{W \in \Orbit(U')} \underline{\Map}_{\cC^\otimes}^{(\alpha_W)_+}(x_W, y_W) \right). \]
\end{prp}

Let us now apply the $\cT$-Segal condition to characterize $\cT$-symmetric monoidal $\cT$-$\infty$-categories as $\cT$-commutative monoids in $\uCatT$.

\begin{remark} \label{rem:SMCsAreCommMonoids}
Under the equivalences (given by straightening and \cite[Prop.~3.10]{Exp2}, respectively)
$$\Cat^{\cocart}_{/\uFinpT} \simeq \Fun(\uFinpT, \Cat) \simeq \Fun_T(\uFinpT, \uCatT) $$
we see that $\cT$-symmetric monoidal $\cT$-$\infty$-categories $\cC^\otimes$ correspond to $\cT$-commutative monoids $M$, i.e., $\cT$-semiadditive $\cT$-functors \cite[Def.~5.3]{Exp4}, since by \cref{cor:SymmetricMonoidalSegalCondition}, $M$ transforms
$$[U_+ \to V] \simeq \coprod_{W \in \Orbit(U)} \coprod_{W \rightarrow V} I(W)_+ \in \FinpT[V]$$
into
$$\prod_{W \in \Orbit(U)} \prod_{W \rightarrow V} \cC_{\underline{W}} \simeq \prod_{W \in \Orbit(U)} \underline{\Fun}_{\cT^{/V}}(\underline{W}, \cC_{\underline{V}}) \simeq \underline{\Fun}_{\cT^{/V}}(\underline{U}, \cC_{\underline{V}}) \in \Cat_{\cT^{/V}}.$$
\end{remark}

Furthermore, in \cite[Thm.~6.5]{Exp4} the first author identified $\cT$-commutative monoids with $\cT$-Mackey functors. Using this, we can relate our notion of $\cT$-symmetric monoidal $\cT$-$\infty$-category with that which appears in \cite{BACHMANN2021}, wherein the norm functors of \cref{ntn:norm_functors} appear as the covariant part of categorical Mackey functor.

\begin{theorem} \label{thm:TwoPresentationsOfTSMCs} We have a canonical equivalence of $\infty$-categories
\[ \CatT^\otimes \simeq \Fun^{\times}(\Span(\FinT), \Cat) \]
between the $\infty$-category $\CatT^\otimes$ of $\cT$-symmetric monoidal $\cT$-$\infty$-categories and the $\infty$-category of product-preserving functors from $\Span(\FinT)$ to $\Cat$.
\end{theorem}
\begin{proof} In \cite[Thm.~6.5]{Exp4}, the first author proved that given a $\cT$-$\infty$-category $\cD$ with finite $\cT$-limits, precomposition by the inclusion $\uFinpT \to \underline{\Span}(\FinT) \coloneq \Span(\ArOT; (\ArOT), (\ArOT)^{\tdeg})$ induces an equivalence of $\cT$-$\infty$-categories
\[ \uFun^{\times}_{\cT}(\underline{\Span}(\FinT),\cD) \xto{\sim} \underline{\CMon}_{\cT}(\cD). \]
In particular, if $\cD = \CatT$, then as we just observed $\CMon_{\cT}(\cD) \simeq \CatT^\otimes$, and passing to cocartesian sections we obtain
\[ \Fun^{\times}_{\cT}(\underline{\Span}(\FinT),\uCatT) \simeq \CatT^{\otimes}. \]
Under the equivalence
\[ \Fun_T(\underline{\Span}(\FinT), \uCatT) \simeq  \Fun(\underline{\Span}(\FinT), \Cat) \]
let $\Fun'(\underline{\Span}(\FF_T), \Cat)$ denote the image of $\Fun^{\times}_T(\underline{\Span}(\FinT),\CatT)$. Explicitly, functors in $\Fun'$ send cartesian edges to equivalences and fiberwise products to products, where cartesian edges in $\underline{\Span}(\FinT)$ are given by
\[ \begin{tikzcd}[row sep=4ex, column sep=4ex, text height=1.5ex, text depth=0.25ex]
U \ar{d}{=} & U \ar{r}{=} \ar{d} \ar{l}[swap]{=} & U \ar{d} \\
W & V \ar{r}{=} \ar{l}[swap]{f} & V
\end{tikzcd} \]
in view of the adjunction
\[ \adjunct{f^{\ast}}{\Span(\FinT[W])}{\Span(\FinT[V])}{f_!}.\]
Let
\[ s: \underline{\Span}(\FinT) \to \Span(\FinT) \]
denote the source map. Then we claim that precomposition by $s$ induces an equivalence
\[ \Fun^{\times}(\Span(\FinT), \Cat) \xto{\sim} \Fun'(\underline{\Span}(\FinT), \Cat) \]
with inverse given by right Kan extension.

First note that given a product preserving functor $G: \Span(\FinT) \to \Cat$, $s^\ast G$ evidently sends cartesian edges to equivalences and fiberwise products to products. Conversely, suppose we have a functor $F: \underline{\Span}(\FinT) \to \Cat$ in $\Fun'$. Let $X \in \Span(\FinT)$ be an object. Note that $(\FinT^\op)^{X/} \to \Span(\FinT)^{X/}$ is an initial functor as it admits a right adjoint which sends a span $X \ot Z \to Y$ to $X \ot Z$. Pulling back, we thereby obtain an initial functor
\[ (\ArOT)^\op \times_{\FinT^\op} (\FinT^\op)^{X/} \to \underline{\Span}(\FinT) \times_{\Span(\FinT)} \Span(\FinT)^{X/} \]
and we are interested in computing the limit of the functor
\[ F' = F \circ \pr: (\ArOT)^\op \times_{\FinT^\op} (\FinT^\op)^{X/} \to (\ArOT)^\op \to \Cat.\]
By our assumption on $F$, $F'$ is the right Kan extension of its restriction to $\cT^\op \times_{\FinT^\op} (\FinT^\op)^{X/}$ (where $\cT^\op \to (\ArOT)^\op$ is the identity section).  Indeed, given an object $I = [V \ot U \to X]$ and an orbit decomposition $U \simeq \coprod_{i=1}^n U_i$, the $n$ projection maps $I \to I_i = [U_i = U_i \to X]$ induce an equivalence
\[ F'(I) = F([U \to V]) \xto{\sim} \prod_{i=1}^n F'(I_i) = \prod_{i=1}^n F([U_i = U_i]). \]

We conclude that the limit of $F'$ is $\prod_{i \in I} F([X_i = X_i])$ for some orbit decomposition $X \simeq \coprod_{i \in I} X_i$, so $s_\ast F(X) \simeq \prod_{i \in I} F(\id_{X_i})$. Using this pointwise formula and a simple argument regarding the morphisms, we see that $s_{\ast} F$ preserves products and the counit and unit maps are equivalences.
\end{proof}

\subsection{Examples}

In this subsection, we discuss some basic examples of $\cT$-$\infty$-operads.

\begin{exm}[$\cT$-(co)cartesian $\cT$-symmetric monoidal structures] \label{exm:cartesian_monoidal}
Let $\cC$ be a $\cT$-$\infty$-category and let $\pi: \cC^{\times} \to \FinT$ be the cartesian fibration defined as in \cite[Prop.~5.12]{Exp2}, so $(\cC^{\times})_U \simeq \prod_{W \in \Orbit(U)} \cC_W$ and the functoriality is given by restriction. Suppose $\cC$ admits finite $\cT$-coproducts. Then by \cite[Prop.~5.12]{Exp2}, $\pi$ is a Beck--Chevalley fibration with cocartesian functoriality given by the coinduction functors, and by Barwick's unfurling construction \cite[\S 11]{M1}, $\pi$ straightens to a product-preserving functor $\Span(\FinT) \to \Cat$. Let
\begin{equation*}
\cC^{\amalg} \to \uFinpT
\end{equation*}
denote the resulting $\cT$-symmetric monoidal $\cT$-$\infty$-category under the equivalence of \cref{thm:TwoPresentationsOfTSMCs}. We call $\cC^{\amalg}$ the \emph{$\cT$-cocartesian $\cT$-symmetric monoidal structure} on $\cC$.

Dually, suppose that $\cC$ admits finite $\cT$-products. Then by the dual of \cite[Prop.~5.12]{Exp2}, the vertical opposite $\pi^{\vop}: (\cC^{\times})^{\vop} \to \FinT$ is a Beck--Chevalley fibration with cocartesian functoriality given by the (opposite of the) induction functors, and thus we obtain a product-preserving functor $\Span(\FinT) \to \Cat$. After postcomposing by the opposite automorphism of $\Cat$, let
\begin{equation*}
\cC^{\productop} \to \uFinpT
\end{equation*}
denote the resulting $\cT$-symmetric monoidal $\cT$-$\infty$-category under the equivalence of \cref{thm:TwoPresentationsOfTSMCs}. We call $\cC^{\productop}$ the \emph{$\cT$-cartesian $\cT$-symmetric monoidal structure} on $\cC$.
\end{exm}

\begin{exm}[$G$-spectra]
Let $\Gpd_{\fin}$ be the $(2,1)$-category of finite groupoids and let
\[ \SH^{\otimes}: \Span(\Gpd_{\fin}) \to \CAlg(\Cat^{\sift}) \]
denote the (restriction of the) functor of \cite[\S 9.2]{BACHMANN2021}. Let $G$ be a finite group and let
$$\omega_G: \Span(\FF_G) \to \Span(\Gpd_{\fin})$$
be the action groupoid functor. Let $(\underline{\Sp}^G)^{\otimes}$ be the $G$-symmetric monoidal $G$-$\infty$-category associated to $\SH^{\otimes} \circ \omega_G$ under \cref{thm:TwoPresentationsOfTSMCs}. Then $(\underline{\Sp}^G)^{\otimes}$ is the $G$-symmetric monoidal structure on $\underline{\Sp}^G$ that encodes the Hill--Hopkins--Ravenel norm functors.
\end{exm}


\begin{exm}[Trivial $\cT$-$\infty$-operad] Let $\Triv_{\cT}^\otimes \subset \uFinpT$ be the wide subcategory on the inert edges. Then $\Triv_{\cT}^\otimes$ is a $\cT$-suboperad of $\uFinpT$ such that the identity cocartesian section $I_+: \cT^\op \to \uFinpT$ restricts to a fully faithful functor into $\Triv_{\cT}^\otimes$ and an equivalence onto $\Triv_{\cT}$. We call $\Triv_{\cT}^\otimes$ the \emph{trivial} $\cT$-$\infty$-operad.
\end{exm}

We claim that given any $\cT$-$\infty$-operad $\cC^\otimes$, we have an equivalence
\[ \underline{\Alg}_{\Triv_\cT, \cT}(\cC) \xto{\simeq} \cC \]
implemented by restriction along $I_+$. To show this, we need the following lemma.

\begin{lemma} \label{lem:InertSubcategoryIsRightKanExtended}
Let $(\cO^\otimes, p)$ be a $\cT$-$\infty$-operad and let $$\cO^\otimes_{\inert} \coloneq \Triv_{\cT}^\otimes \times_{\uFinpT} \cO^\otimes$$ be the wide subcategory on the inert edges. Let $F^{\inert}_{\cO}: \Triv_{\cT}^\otimes \to \Cat$ be the functor classifying the cocartesian fibration $p|_{\cO^\otimes_{\inert}}$. Then $F^{\inert}_{\cO}$ is the right Kan extension of its restriction $F_{\cO}$ along $I_+: \cT^\op \to \Triv_{\cT}^\otimes$ (which classifies the underlying $\cT$-$\infty$-category $\cO$).
\end{lemma}
\begin{proof}
Let $f_+ = [U_+ \to V]$ be any object in $\Triv_{\cT}^\otimes$ and let $\cJ^\op = \cT^\op \times_{\Triv_{\cT}^\otimes} (\Triv_{\cT}^\otimes)^{f_+/}$. We need to show that the natural map
\[ F^{\inert}_{\cO}(f_+) \to \lim ( \theta: \cJ^\op \to \cT^\op \xto{F_{\cO}} \Cat) \]
is an equivalence. If we view $\Orbit(U)$ as a discrete category, then we have a functor $\phi: \Orbit(U) \to \cJ^\op$ that sends $W$ to $(W,\chi_{[W \subset U]})$, and by definition we have an equivalence
\[ F^{\inert}_{\cO}(f_+) \xto{\simeq} \lim (\theta \circ \phi). \]
Consequently, it suffices to show that $\phi$ is right cofinal. Since $\Triv_{\cT}^\otimes \simeq ((\ArOT)^{\si})^{\op}$ under the inclusion into $\uFinpT = \Span(\ArOT; (\ArOT)^{\si}, (\ArOT)^{\tdeg})$, we may equivalently show that
$$\phi^\op: \Orbit(U)^{\op} = \Orbit(U) \to \cJ \simeq \cT \times_{(\ArOT)^{\si}} ((\ArOT)^{\si})^{/f}$$
is left cofinal. For this, we will apply Quillen's Theorem A. Let
\[ \overline{\alpha} = \left( \begin{tikzcd}
X \ar{r}{\alpha} \ar{d}{=} & U \ar{d}{f} \\
X \ar{r} & V
\end{tikzcd} \right) \]
be any object in $\cJ$. Then since $\Orbit(U)$ is discrete, we have an equivalence
\[ \Orbit(U) \times_{\cJ} \cJ^{\overline{\alpha}/} \simeq \coprod_{W \in \Orbit(U)} \Map_{\cJ}(\overline{\alpha}, \chi_{[W \subset U]}). \]
If we let $W$ be the orbit in $U$ that $\alpha$ factors through, then for all other $W' \in \Orbit(X)$, we have
$$\Map_{\cJ}(\overline{\alpha}, \chi_{[W' \subset U]}) = \emptyset.$$
To compute the remaining mapping space, observe that we have a homotopy pullback square
\[ \begin{tikzcd}
\Map_{\cJ}(\overline{\alpha}, \chi_{[W \subset U]}) \ar{r} \ar{d} & \Map_{(\ArOT)^{\si}}(\id_X, \id_W) \simeq \Map_{\ArOT}(\id_X, \id_W) \ar{d} \\
\ast \ar{r}{\overline{\alpha}} & \Map_{(\ArOT)^{\si}}(\id_X, f) \simeq \Map_{\ArOT}(\id_X, f)
\end{tikzcd} \]
where for the righthand equivalences we use our assumption that $\cT$ is orbital. But the righthand vertical map identifies with
\[ \Map_{\FinT}(X,W) \to \Map_{\FinT}(X,U) \times_{\Map_{\FinT}(X,V)} \Map_{\FinT}(X,V) \simeq \Map_{\FinT}(X,U) \]
and is hence an equivalence. We conclude that $\Orbit(U) \times_{\cJ} \cJ^{\overline{\alpha}/}$ is contractible, so $\phi^\op$ is left cofinal.
\end{proof}

\begin{corollary} \label{cor:TrivialOperad}
Let $\cC$ be a $\cT$-$\infty$-category, let $F_{\cC}: \cT^\op \to \Cat$ be the functor that classifies $\cC$, and let $q: \Triv_{\cT}(\cC)^\otimes \to \Triv_{\cT}^\otimes$ be the cocartesian fibration classified by the right Kan extension of $F_{\cC}$ along $I_+: \cT^\op \to \Triv_{\cT}^\otimes$.\footnote{Using \cite[Ex.~2.26]{Exp2}, we could give a definition of $\Triv_{\cT}(\cC)^\otimes$ at the level of marked simplicial sets, without passing through straightening and unstraightening.} Then $\Triv_{\cT}(\cC)^\otimes$ is a $\cT$-$\infty$-operad and for any $\cT$-$\infty$-operad $\cO^\otimes$, we have an equivalence of $\cT$-$\infty$-categories
\[ \underline{\Alg}_{\cT}(\Triv_{\cT}(\cC), \cO) \xto{\simeq} \underline{\Fun}_{\cT}(\cC, \cO) \]
implemented by restriction along $I_+$. 
\end{corollary}
\begin{proof}
By \cref{lem:InertSubcategoryIsRightKanExtended}, we have an equivalence of $\infty$-categories between $\Triv_{\cT}$-monoidal $\cT$-$\infty$-categories and the full subcategory of $\Fun(\Triv_{\cT}^\otimes, \Cat)$ spanned by those functors right Kan extended from $\cT^\op$ along $I_+$, so in particular $\Triv_{\cT}(\cC)^\otimes$ is a $\cT$-$\infty$-operad. For the second statement, since restriction along $I_+$ is a $\cT$-functor, it suffices to show an equivalence of fibers over every orbit $V \in \cT$, so after replacing $\cT$ by $\cT^{/V}$ we reduce to the claim that
\[ \Alg_{\cT}(\Triv_{\cT}(\cC), \cO) \simeq \Fun^{\cocart}_{/\Triv_{\cT}^\otimes}(\Triv_{\cT}(\cC)^\otimes, \cO^\otimes_{\inert}) \to \Fun_{\cT}(\cC, \cO) \]
is an equivalence of $\infty$-categories. But this follows immediately from \cref{lem:InertSubcategoryIsRightKanExtended}.
\end{proof}

\begin{exm} \label{exm:TrivialSuboperads}
By \cref{cor:TrivialOperad}, the $\cT$-suboperads of the trivial $\cT$-$\infty$-operad are in bijective correspondence with sieves of $\cT$. For instance, we have that the initial $\cT$-$\infty$-operad $\Triv_{\cT}(\emptyset)^\otimes$ corresponds to $\emptyset \subset \cT$ and identifies with the full subcategory of $\uFinpT$ on objects $[\emptyset_+ \to V]$.
\end{exm}

\begin{exm}[Indexing systems and the family of commutative $\cT$-$\infty$-operads] \label{exm:CommutativeOperads}
Clearly $\uFinpT$ is itself a $\cT$-$\infty$-operad, which deserves to be called the \emph{$\cT$-commutative $\cT$-$\infty$-operad} $\Com_{\cT}^\otimes$. Note that the identity cocartesian section $I_+: \cT^\op \to \Com_{\cT}^\otimes$ restricts to an equivalence $\cT^\op \xto{\simeq} \Com_{\cT}$.

In the parametrized setting, we may further define a family of $\cT$-suboperads of $\Com_{\cT}^\otimes$ so as to encode different flavors of parametrized commutativity. First, define the \emph{minimal} $\cT$-commutative $\cT$-$\infty$-operad $\Com^\otimes_{\cT^\simeq}$ to be the wide subcategory of $\uFinpT$ containing all morphisms
\[ \begin{tikzcd}[row sep=4ex, column sep=4ex, text height=1.5ex, text depth=0.25ex]
U \ar{d} & Z \ar{l} \ar{d} \ar{r}{m} & X \ar{d} \\
V & Y \ar{l} \ar{r}{=} & Y.
\end{tikzcd} \]
where $m$ is a coproduct of fold maps (including possibly empty fold maps). In other words, if we let $\nabla$ denote the collection of fold maps and $(\ArOT)^{\nabla:\tdeg} \subset \ArOT$ the wide subcategory on morphisms with source in $\nabla$ and target degenerate, then
$$\Com^\otimes_{\cT^\simeq} = \Span(\ArOT; (\ArOT)^{\si}, (\ArOT)^{\nabla:\tdeg}),$$
where we use that $\nabla$ is stable under pullback by summand inclusions, and $\Com^{\otimes}_{\cT^\simeq}$ is classified by the functor $\cT^\op \to \Cat$ that sends an orbit $V$ to $\Span(\FinT[V]; \FinT[V]^{\si}, \FinT[V]^{\nabla})$. Since $\Com^{\otimes}_{\cT^\simeq}$ contains all inert edges in $\Com_{\cT}^\otimes$ and $\Com_{\cT^\simeq} \simeq \cT^\op$, to verify that $\Com^\otimes_{\cT^\simeq}$ is a $\cT$-suboperad it only remains to check condition (3\textquotesingle) of \cref{rem:alternateOperadDef}. But this condition is satisfied since a map $\alpha: U \to U'$ of finite $\cT$-sets is a fold map if and only if for all $W \in \Orbit(U')$, the pullback $\alpha_W: U \times_{U'} W \to W$ is a fold map.

Now suppose that we want to define a $\cT$-$\infty$-operad $\cO^\otimes$ such that we have $\cT$-operadic inclusions
$$\Com^\otimes_{\cT^\simeq} \subset \cO^\otimes \subset \Com_{\cT}^\otimes.$$
Since $\Com_{\cT^\simeq} = \Com_{\cT} \simeq \cT^\op$, the only constraint on $\cO^\otimes$ to be a $\cT$-suboperad arises from condition (3\textquotesingle). In other words, to specify $\cO^\otimes$ we may as well specify the morphisms $\alpha: U \to V$ with $V$ an orbit that we wish to be active. We already have that all fold maps are active, so in particular all summand inclusions are active. Furthermore, as observed in \cref{rem:OperadicCompositionUnwinding}, for any orbit $W \subset U$, the composite map $W \subset U \xto{\alpha} V$ yields an operadic composition map
$$\Mul^{W \subset U}_{\cO} \times \Mul^{\alpha}_{\cO} \simeq \ast \times \Mul^{\alpha}_{\cO} \to \Mul^{\alpha|_W}_{\cO}, $$
so if $\alpha$ is active, we must have that $\alpha|_W$ is active for all $W \in \Orbit(U)$. The converse holds by a similar argument since $\alpha$ factors as
\[ \begin{tikzcd}
U \ar{r}{\simeq} & \coprod_{W \in \Orbit(U)} W \ar{r}{\coprod \alpha|_W} & \coprod_{W \in \Orbit(U)} V \ar{r}{\nabla} & V
\end{tikzcd}. \]
We thereby reduce to specifying whether or not morphisms $\alpha: W \to V$ are active for both $W$ and $V$ orbits. Moreover, by examining \cref{rem:OperadicCompositionUnwinding} again we see that the only constraints are:
\begin{enumerate}
\item The active morphisms contain all equivalences and are closed under composition, so assemble to a subcategory $\cI \subset \cT$ such that $\cI$ contains the maximal subgroupoid $\cT^{\simeq}$ of $\cT$.
\item The active morphisms are closed under base-change, in the sense that for any commutative square
\[ \begin{tikzcd}
W' \ar{r} \ar{d}{\alpha'} & W \ar{d}{\alpha} \\
V' \ar{r} & V
\end{tikzcd} \]
such that the map $W' \to V' \times_{V} W$ is a summand inclusion, if $\alpha$ is in $\cI$ then $\alpha'$ is in $\cI$.
\end{enumerate}
\end{exm}

\begin{definition} \label{def:IndexingSystem} A $\cT$-\emph{indexing system} is a subcategory $\cI$ of $\cT$ that satisfies the above conditions (1) and (2).
\end{definition}

\begin{remark} \label{rem:IndexingSystemAlternative}
An indexing system $\cI$ is the same data as a subcategory $\overline{\cI} \subset \FinT$ such that:
\begin{enumerate}
\item $\overline{\cI}$ contains the maximal subgroupoid $\FinT^{\simeq}$.
\item Morphisms in $\overline{\cI}$ are closed under base-change and binary coproducts.
\item $\overline{\cI}$ contains all fold maps (and hence all summand inclusions).
\end{enumerate}
Indeed, the assignment $\goesto{\overline{\cI}}{\cI = \overline{\cI} \times_{\FinT} \cT}$ is seen to identify the two notions, with inverse given by taking the finite coproduct completion of $\cI$.
\end{remark}

\begin{definition} \label{dfn:IndexingSystemCommOperad}
Let $\cI$ be a $\cT$-indexing system and let $(\ArOT)^{\cI:\tdeg} \subset \ArOT$ be the wide subcategory on morphisms with source in $\overline{\cI}$ and target degenerate. We then define the \emph{$\cI$-commutative $\cT$-$\infty$-operad} to be
\[ \Com^\otimes_{\cI} \coloneq \Span(\ArOT; (\ArOT)^{\si}, (\ArOT)^{\cI:\tdeg}). \]
More generally, we define a \emph{commutative} $\cT$-$\infty$-operad to be any $\cT$-suboperad of $\Com^\otimes_{\cT}$ containing $\Com^\otimes_{\cT^\simeq}$.
\end{definition}

The above analysis confirms the following proposition.

\begin{proposition} \label{prp:IndexingSystemsAreCommutativeSuboperads}
The assignment $\goesto{\cI}{\Com^\otimes_{\cI}}$ implements an inclusion-preserving bijection between $\cT$-indexing systems and commutative $\cT$-$\infty$-operads.
\end{proposition}

\begin{remark} \label{rem:blumberg_hill}
Let $\cT = \OO_G$. A $G$-indexing system $I$ in the sense of Blumberg-Hill \cite[Def.~3.22]{MR3406512} as reformulated by Rubin \cite[Def.~2.12]{Rubin2021} is a collection $\{ I(H) \}$ of finite $H$-sets for every subgroup $H \leq G$ such that $I(H)$ contains all finite $H$-sets with trivial $H$-action and satisfies the following closure properties:
\begin{enumerate}
\item If $U \in I(H)$ and $U' \cong U$, then $U' \in I(H)$.
\item For every subgroup $K \leq H$, if $U \in I(H)$ then $\res^H_K U \in I(K)$.
\item For every conjugate $H' = g H g^{-1}$ of $H$, if $U \in I(H)$ then its conjugate $U' \in I(H')$.
\item If $U \in I(H)$ and $U' \subset U$, then $U' \in I(H)$.
\item \label{item:IndexingSystemCoproducts} If $U, U' \in I(H)$, then $U \coprod U' \in I(H)$.
\item For any subgroup $K \leq H$, if $U \in I(K)$ and $H/K \in I(H)$ then $\ind^H_K U \in I(H)$.
\end{enumerate}

In \cite[Thm.~3.17]{MR3406512}, it was shown that $G$-indexing systems are in bijective correspondence with subcategories of $\FF_G$ satisfying the conditions of \cref{rem:IndexingSystemAlternative} (and thus with the $\OO_G$-indexing systems of \cref{def:IndexingSystem}). For the convenience of the reader, we review this correspondence. Note that under the equivalences $\FF_H \simeq \FF_G^{/(G/H)}$, given a subgroup $K \leq H$, induction corresponds to postcomposition by $G/K \to G/H$. Therefore, given a $G$-indexing system $I$, we may define a wide subcategory $\cI \subset \OO_G$ to be the subcategory whose morphisms $f: V \to W \cong G/H$ are such that $f \in I(H)$. The enumerated conditions then imply that $\cI$ is a $\OO_G$-indexing system, using closure under restriction and inclusion to validate the base-change condition.

Conversely, suppose $\cI$ is a $\OO_G$-indexing system and let $\overline{\cI} \subset \FF_G$ be the subcategory generated by $\cI$ as in \cref{rem:IndexingSystemAlternative}. Let $I(H)$ be the subset of objects of $\FF_H$ given by morphisms in $\overline{\cI}$ with target $G/H$ under $\FF_H \simeq \FF_G^{/(G/H)}$. Then one sees that $I$ is a $G$-indexing system and these assignments are mutually inverse --  note that condition (\ref{item:IndexingSystemCoproducts}) holds since given $U, U' \to G/H$, the coproduct in $\FF_H$ is given by the composition $U \coprod U' \to G/H \coprod G/H \to G/H$.

Consequently, by the work of Bonventre--Pereira \cite{Bonventre2021}, Guti\'{e}rrez--White \cite{Gutirrez2018}, and Rubin \cite{Rubin2021}, we see that the commutative $G$-$\infty$-operads are in bijection with the $N_{\infty}$-operads of Blumberg--Hill.
\end{remark}

A straightforward adaptation of the proof of \cref{thm:TwoPresentationsOfTSMCs} shows the following.

\begin{definition}
Given a $\cT$-indexing system $\cI$, let $\Cat^\otimes_{\cI}$ be the $\infty$-category of $\cI$-symmetric monoidal $\cT$-$\infty$-categories and $\cI$-symmetric monoidal $\cT$-functors thereof. 
\end{definition}

\begin{theorem}
Let $\cI$ be a $\cT$-indexing system. We then have a canonical equivalence
\[ \Cat^\otimes_{\cI} \simeq \Fun^{\times}(\Span(\FinT; \FinT, \overline{\cI}), \Cat). \]
\end{theorem}

\begin{corollary} \label{cor:MinimalIndexingSystemIsFiberwiseCommutative}
For the minimal indexing system $\cI = \cT^\simeq$, we have a canonical identification of $\cT^{\simeq}$-symmetric monoidal $\cT$-$\infty$-categories with $\cT^{\op}$-cocartesian families of symmetric monoidal $\infty$-categories (\cite[Def.~4.8.3.1]{HA}).
\end{corollary}

\subsection{\texorpdfstring{$T$}{T}-operadic nerve}

In this subsection, we suppose that $\cT$ is equivalent to the nerve of a $1$-category $T$. For example, by the following proposition we could take $\cT$ to be any atomic orbital $\infty$-category that admits a final object.

\begin{proposition} \label{lem:AtomicityEnforcesTruncatedness} Suppose $\cT$ is an atomic orbital $\infty$-category that admits a final object $\ast$. Then $\cT$ is equivalent to the nerve of a $1$-category.
\end{proposition}
\begin{proof}
By \cite[Prop.~2.3.4.18]{HTT} it suffices to show that the mapping spaces of $\cT$ are $0$-truncated, or equivalently that the essentially unique maps $V \to \ast$ are $0$-truncated for all $V \in \cT$. By \cite[Lem.~5.5.6.15]{HTT}, this occurs if and only if the diagonal $\delta: V \to V \times V$ in $\FinT$ is $(-1)$-truncated. But since $\cT$ is atomic and $\delta$ is split by either projection to $V$, it follows that $\delta$ is a summand inclusion and hence a monomorphism.
\end{proof}

Correspondingly, let $\FF_T$ denote the subcategory of the category of $\Set$-valued presheaves on $T$ spanned by the finite coproducts of representables, so that $\FinT \simeq N(\FF_T)$.

\begin{rem} If $T$ is a $1$-category, then the \emph{a priori} $(2,1)$-category
\[\underline{\FF}_{T,\ast}:=\Span(\FF_T^{\Delta^1}\times_{\FF_T} T, (\FF_T^{\Delta^1})^{\tdeg}, (\FF_T^{\Delta^1})^{\si})\]
is enriched in setoids and therefore equivalent to a $1$-category. Indeed, the only automorphisms of spans of the form
\[\begin{tikzcd}
    U \ar[d] & \tilde U\ar[l]\ar[d]\ar[r] & U'\ar[d]\\
    V & V'\ar[l]\ar[r,-,double] & V'
\end{tikzcd}\]
where the left square is in $(\FF_T^{\Delta^1})^{si}$ are identities. In what follows, we will implicitly make a choice of 1-categorical model for $\underline{\FF}_{T,\ast}$ by picking a representative for every equivalence class of morphisms.
\end{rem}

Our goal is to indicate how to prescribe the data of a $\cT$-$\infty$-operad in terms of the stricter data of a \emph{simplicial colored $T$-operad}, which will be defined along the lines suggested by \cref{rem:OperadicCompositionUnwinding}. To concisely state its definition, we first need to introduce some notation.

\begin{notation} Let $A\FF_T\subseteq \Ar(\FF_T)$ denote the wide subcategory of the arrow category whose morphisms are cartesian squares. Suppose that $\ob O \to \FF_T$ is a Grothendieck fibration fibered in sets, and let $\ob O(U)$ denote the fiber of $\ob O$ over $U \in \FF_T$. We then write
\[A^O\FF_T \coloneqq A\FF_T \times_{(\FF_T\times\FF_T)} (\ob O\times\ob O) \]
for the category whose objects are triples $(f:U\to V,\, x \in \ob O(U),\, y \in \ob O(V))$ and whose morphisms are cartesian squares
 \[\begin{tikzcd}
    U' \ar[d,"f'"] \ar[r,"\phi"] & U \ar[d,"f"]\\
    V' \ar[r,"\psi"] & V
 \end{tikzcd}\]
 such that $x'=\phi^\ast x$ and $y'=\psi^\ast y$. We have functors
 \begin{align*}
 1:\ob O\to A^O\FF_T , \quad &(x \in \ob O(U))\mapsto (\id_U,x,x),\\
 \circ:A^O\FF_T\times_{\ob O} A^O\FF_T\to A^O\FF_T , \quad &(f:U\to V,g:V\to W, x,y,z)\mapsto (gf:U\to W,x,z).
 \end{align*}
\end{notation}

\begin{dfn} \label{def:SimplicialColoredTOperad} A (fibrant) \emph{simplicial colored $T$-operad} $O$ is the data of
\begin{enumerate}
 \item A `$T$-set of colors', given as a Grothendieck fibration fibered in sets
  $$\ob O \to \FF_T $$
 classified by a functor $\FF_T^\op\to\Set$ preserving finite products.

 \item A collection of spaces of multimorphisms, packaged into a functor
 \[\Mul_O:(A^O\FF_T)^\op\to \sSet, \quad (f:U\to V,\, x\in\ob O(U),\, y\in \ob O(V)) \mapsto \Mul_O^f(x,y)\]
 that preserves finite products and is valued in Kan complexes.\footnote{If $V$ is an orbit, we thus obtain a $T^{/V}$-space $\underline{\Mul}_O^f(x,y): (T^{/V})^\op \to \sSet$ by precomposing $\Mul_O$ with the functor $T^{/V} \to A^O\FF_T$  over $\FF_T$ determined by $(f,x,y)$ -- indeed, one has an equivalence $(A^O\FF_T)^{/(f,x,y)} \simeq (\FF_T)^{/V}$.}

 \item A distinguished `identity' for $\Mul_O$, given by a natural transformation
 \[1:\ast\to \Mul_O^{\id_U}(x,x)\]
 of functors $(\ob O)^\op\to \sSet$, where the right hand side is the composition
 \[(\ob O)^\op\xrightarrow{1}A^O\FF_T^\op\xrightarrow{\Mul}\sSet\,.\]
 \item A `composition law' for $\Mul_O$, given by a natural transformation
 \[\circ: \Mul_O^f(x,y)\times\Mul_O^g(y,z)\to \Mul_O^{gf}(x,z)\]
 of functors $(A^O\FF_T\times_{\ob O}A^O\FF_T)^\op\to \sSet$, where the left hand side is the composition
 \[(A^O\FF_T\times_{\ob O}A^O\FF_T)^\op\to (A^O\FF_T\times A^O\FF_T)^\op\xrightarrow{\Mul\times\Mul}\sSet\times\sSet\xrightarrow{\times}\sSet\,,\]
 and the right hand side is the composition
 \[(A^O\FF_T\times_{\ob O}A^O\FF_T)^\op\xrightarrow{\circ}A^O\FF_T^\op\xrightarrow{\Mul}\sSet\,.\]
\end{enumerate}
These data are required to satisfy the following compatibilities:
\begin{itemize}
 \item \textbf{Unitality}: The compositions
 \[\Mul^f_O(x,y)\xrightarrow{(\id,1_y)} \Mul^f_O(x,y)\times\Mul^{\id_V}_O(y,y)\xrightarrow{\circ} \Mul^f_O(x,y)\]
 and
 \[\Mul^f_O(x,y)\xrightarrow{(1_x,\id)} \Mul^{\id_U}_O(x,x)\times\Mul^f_O(x,y)\xrightarrow{\circ} \Mul^f_O(x,y)\]
 are the identity natural transformation.
 \item \textbf{Associativity} The following diagram is commutative
 \[\begin{tikzcd}
    \Mul^f_O(x,y)\times \Mul^g_O(y,z)\times\Mul^h_O(z,t)\ar[r,"(\circ{,}\id)"]\ar[d,"(\id{,}\circ)"] & \Mul^{gf}_O(x,z)\times\Mul^h_O(z,t)\ar[d,"\circ"]\\
    \Mul^f_O(x,y)\times\Mul^{hg}_O(y,t)\ar[r,"\circ"] & \Mul^{hgf}_O(x,t)
 \end{tikzcd}\,.\]
\end{itemize}
\end{dfn}

\begin{construction} From the data of \cref{def:SimplicialColoredTOperad} we will build a simplicial category $O^\otimes$ over $\underline{\FF}_{T,\ast}$ as follows. We let the objects of $O^\otimes$ be the pairs $([U_+ \to V], x)$, where $[U_+ \to V]$ is an object of $\underline{\FF}_{T,\ast}$ and $x \in \ob O(U)$, and we define as mapping simplicial sets
\[\Map_{O^\otimes}(([U_+ \to V],x),([U'_+ \to V'],x')) \coloneqq \coprod_{U\xleftarrow{i}U_0\xrightarrow{f} U'} \Mul^f(i^*x,x') \]
where the coproduct is indexed by the set of all maps $[U_+ \to V]\to [U'_+ \to V']$ in $\underline{\FF}_{T,\ast}$. The identity of $([U_+ \to V],x)$ is given by $1_{(U,x)}\in\Mul^{\id_U}(x,x)$. If 
\[\begin{tikzcd}
    && U_1 \ar[ld,swap,"j'"] \ar[rd,"f'"]&&\\
    & U_0\ar[ld,swap,"i"]\ar[rd,"f"] && U'_0\ar[ld,swap,"j"]\ar[rd,"g"]\\
    U && U' && U''
\end{tikzcd}\]
is a diagram representing a composition in $\underline{\FF}_{T,\ast}$, composition over it is given by
\[\Mul^f(i^*x,x')\times \Mul^g(j^*x',x'')\to \Mul^{f'}((j')^*i^*x,j^*x')\times\Mul^g(j^*x',x'')\to \Mul^{gf'}((i j')^*x,x'')\,.\]
Verifying that this satisfies associativity and unitality is left as an exercise for the reader.
\end{construction}

\begin{prp}
 The map $N(O^\otimes)\to N(\underline{\FF}_{T,\ast}) \simeq \uFinpT$ is a $\cT$-$\infty$-operad.   
\end{prp}
\begin{proof}
 Using \cite[Prop.~2.4.1.10]{HTT} we see that the above map is an inner fibration and that its restriction over the subcategory of inert edges is a cocartesian fibration. Then remaining properties are true because $\ob O$ and $\Mul$ preserve finite products.
\end{proof}

\subsection{Model structures}

In this subsection, we introduce a model structure for $\cT$-$\infty$-operads by means of Lurie's theory of categorical patterns (\cite[Def.~B.0.19]{HA}). For a $\cT$-$\infty$-operad $\cO^\otimes$, let $\Inert \subset (\cO^\otimes)_1$ denote the subset of inert morphisms.



\begin{definition} \label{def:CategoricalPattern}
We define categorical patterns $\mathfrak{P}_{\cT}$ and $\mathfrak{P}^\otimes_{\cT}$ on $\uFinpT$ as follows. For each collection of morphisms $\underline{\alpha} = \{ \alpha_i: U_i \to V \}_{i=1}^n$ in $\cT$, let $\alpha: U = \bigsqcup_{i=1}^n U_i \xto{(\alpha_i)} V$ and define a morphism
$$f_{\alpha}: (n^{\lhd})^\sharp \to (\uFinpT, \Inert)$$
(for $n = \{ 1,...,n \}$ regarded as a discrete category) that sends the cone point $v$ to $\left[ U_+ \ra V \right]$, $i \in n$ to $\left[{U_i}_+ \ra U_i \right]$, and $v \to i$ to the characteristic morphism $\chi_{[U_i \subset U]}$. Then let $A$ be the set of the $\underline{\alpha}$ and let
\begin{align*}
\mathfrak{P}_{\cT} = (\Inert,  \All, \{ f_{\alpha}: n^{\lhd} \to \uFinpT \}_{\underline{\alpha} \in A}), \\
\mathfrak{P}^\otimes_{\cT} = (\All,  \All, \{ f_{\alpha}: n^{\lhd} \to \uFinpT \}_{\underline{\alpha} \in A}).
\end{align*} 

Furthermore, for any $\cT$-$\infty$-operad $\cO^\otimes$, we define categorical patterns $\mathfrak{P}_{\cO}$ and $\mathfrak{P}^\otimes_{\cO}$ on $\cO^\otimes$ by the construction of \cite[Rem.~B.1.5]{HA}. In other words, let $B$ denote the set of pairs $(\underline{\alpha},\overline{f_{\alpha}})$ where $\overline{f_{\alpha}}: (n^{\lhd})^\sharp \to (\cO^\otimes, \Inert)$ is any lift of $f_{\alpha}$, and let
\begin{align*}
\mathfrak{P}_{\cO} = (\Inert, \All, \{ \overline{f_{\alpha}}: n^{\lhd} \to \cO^\otimes \}_{(\underline{\alpha},\overline{f_{\alpha}}) \in B} ), \\
\mathfrak{P}^\otimes_{\cO} = (\Inert, \All, \{ \overline{f_{\alpha}}: n^{\lhd} \to \cO^\otimes \}_{(\underline{\alpha},\overline{f_{\alpha}}) \in B} ).
\end{align*} 

\end{definition}

\begin{conthm} \label{thm:ModelStructureOperads}
The \emph{$\cT$-operadic model structure} on the category $\sSet^+_{/(\uFinpT,\Inert)}$ is that defined by the categorical pattern $\mathfrak{P}_{\cT}$ of \cref{def:CategoricalPattern} according to \cite[Thm.~B.0.20]{HA}. The $\cT$-operadic model structure is left proper, combinatorial, simplicial, and has the following properties:
\begin{enumerate}
\item The cofibrations are precisely the monomorphisms.
\item A marked map $\fromto{X}{Y}$ over $(\uFinpT,\Inert)$ is a weak equivalence if for any $\cT$-$\infty$-operad $\cO^{\otimes}$, the induced map
\[\fromto{\Map_{/(\uFinpT,\Inert)}(Y,(\cO^{\otimes}, \Inert))}{\Map_{/(\uFinpT,\Inert)}}(X,(\cO^{\otimes}, \Inert))\]
is a weak equivalence.
\item An object is fibrant if it is of the form $(\cO^{\otimes}, \Ne)$ for some $\cT$-$\infty$-operad $\cO^{\otimes}$.
\item The fibrations between fibrant objects in $\sSet^+_{/(\uFinpT,\Inert)}$ are exactly given by the fibrations of $\cT$-$\infty$-operads.
\end{enumerate}

Furthermore, for any $\cT$-$\infty$-operad $\cO^{\otimes}$, we define the \emph{$\cT$-operadic model structure} on the category $\sSet^+_{/(\cO^\otimes, \Inert)}$ via the categorical pattern $\mathfrak{P}_{\cO}$, and this coincides with the model structure induced from the $\cT$-operadic model structure on $\sSet^+_{/(\uFinpT,\Inert)}$ by slicing over $(\cO^{\otimes}, \Inert)$.

Finally, we also define the \emph{$\cT$-monoidal model structure} on the category $\sSet^+_{/\cO^\otimes}$ via the categorical pattern $\mathfrak{P}^{\otimes}_{\cO}$. This construction has the same formal properties as with the $\cT$-operadic model structure, but where the fibrant objects are precisely the $\cO$-monoidal $\cT$-$\infty$-categories with the cocartesian edges marked.
\end{conthm}
\begin{proof}
The construction of the $\cT$-operadic model structure on $\sSet^+_{/(\uFinpT,\Inert)}$ and the first three claims about it follows immediately from \cite[Thm.~B.0.20]{HA} and the definition of a $\cT$-$\infty$-operad. The fourth claim and the assertion about the model structure on $\sSet^+_{/(\cO^\otimes, \Inert)}$ follow from \cite[Prop.~B.2.7]{HA}. Finally, the analogous assertions about the $\cT$-monoidal model structure all follow in the same way (cf. \cite[Var.~2.1.4.13]{HA}).
\end{proof}

\begin{dfn}
We define the \emph{$\infty$-category of (small) $\cT$-$\infty$-operads} $$\Op_{\cT} \coloneq N((\sSet^+_{/(\uFinpT,\Inert)})^f)$$ to be the simplicial nerve of the full simplicial subcategory of $\sSet^+_{/(\uFinpT,\Inert)}$ spanned by the fibrant objects in the $\cT$-operadic model structure. We further let $\Cat_{\cT}^\otimes$ denote the subcategory of $\Op_{\cT}$ spanned by the (small) $\cT$-symmetric monoidal $\cT$-$\infty$-categories and $\cT$-symmetric monoidal functors thereof, or equivalently, the simplicial nerve $N((\sSet^+_{/\uFinpT})^f)$ taken with respect to the $\cT$-monoidal model structure.

For a small $\cT$-$\infty$-operad $\cO^\otimes$, we then let
\[ \Op_{\cO,\cT} \coloneq N((\sSet^+_{/(\cO^\otimes,\Inert)})^f), \quad \Cat^\otimes_{\cO,\cT} \coloneq N((\sSet^+_{/\cO^\otimes})^f).\]
Note that $\Op_{\cO,\cT} \simeq (\Op_{\cT})^{/\cO^\otimes}$ and $\Cat^\otimes_{\cO,\cT}$ includes as the subcategory of $\Op_{\cO,\cT}$ on the $\cO$-monoidal $\cT$-$\infty$-categories and morphisms thereof.
\end{dfn}

\begin{corollary}
For any small $\cT$-$\infty$-operad $\cO^\otimes$, the $\infty$-categories $\Op_{\cO,\cT}$ and $\Cat_{\cO,\cT}^\otimes$ are presentable.
\end{corollary}
\begin{proof}
Since the $\cT$-operadic model structure on $\sSet^+_{/(\cO^\otimes,\Inert)}$ and the $\cT$-monoidal model structure on $\sSet^+_{/\cO^\otimes}$ are combinatorial and simplicial by \cref{thm:ModelStructureOperads}, it follows from \cite[Prop.~A.3.7.6]{HTT} that $\Op_{\cO,\cT}$ and $\Cat_{\cO,\cT}^\otimes$ are presentable. 
\end{proof}

Using the theory of categorical patterns, it is easy to construct \emph{cotensors} in the $\infty$-category of $\cT$-$\infty$-operads fibered over a given base.

\begin{construction} \label{constr:OperadCotensors}
Let $\cO^\otimes$ be a $\cT$-$\infty$-operad and let $K$ be a marked simplicial set. By \cite[Prop.~B.1.9]{HA} applied to the trivial categorical pattern on $\sSet^+$ and $\mathfrak{\cP}_{\cO}$ on $\sSet^+_{/(\cO^\otimes, \Ne)}$, the functor
\[ (- \times K) :\sSet^+_{/(\cO^\otimes, \Ne)} \to \sSet^+_{/(\cO^\otimes,\Ne)}, \quad A \mapsto A \times K \]
is left Quillen. We denote its right adjoint on fibrant objects by $(\cC^\otimes,p) \mapsto ((\cC^\otimes,p)^K, p^K)$ and the underlying $\cT$-$\infty$-category by $(\cC,p)^K$. Since this adjunction is also simplicial, for $(\cC^\otimes,p)$ and $(\cD^\otimes,q)$ fibrations of $\cT$-$\infty$-operads over $\cO^\otimes$ we obtain equivalences of $\infty$-categories
\[ \Alg_{\cO, \cT}(\cC, (\cD,q)^K) \simeq \Fun_{/(\cO^\otimes, \Ne)}(K \times (\cC^\otimes, \Ne), (\cD^\otimes, \Ne)) \simeq \Fun(K, \Alg_{\cO, \cT}(\cC,\cD)^{\sim}) \]
where $(-)^{\sim}$ means we take the marking given by the equivalences. In other words, we have constructed the cotensor of $\Op_{\cO,\cT}$ over $\Cat$ at the level of marked simplicial sets. Note also that a fibrant replacement of $K \times (\cC^\otimes, \Ne)$ computes the tensor. Repeating this analysis with the categorical pattern $\mathfrak{\cP}_{\cO}^\otimes$, we see that if $\cC^\otimes$ and $\cD^\otimes$ are $\cO$-monoidal, then $(\cD^\otimes,p)^K$ is moreover $\cO$-monoidal, and we have equivalences of $\infty$-categories
\[ \Fun^\otimes_{\cO,\cT}(\cC, (\cD,q)^K) \simeq \Fun_{/\cO^\otimes}(K \times (\cC^\otimes)^\sharp, (\cD^\otimes)^\sharp) \simeq \Fun(K, \Fun^\otimes_{\cO,\cT}(\cC,\cD)^{\sim}). \]

Moreover, note that if $F: \cO \to \Cat$ denotes the functor classifying $\cD \to \cO$, then $(\cD,q)^K \to \cO$ is classified by the functor $\cO \to \Cat$ given by applying $\Fun(K, (-)^{\sim})$ fiberwise to $F$. We may thus consider $(\cD^\otimes,p)^K$ to be a construction of the \emph{pointwise} $\cO$-monoidal structure on $(\cD,q)^K$.
\end{construction}




\subsection{Big \texorpdfstring{$\cT$-$\infty$}{T-infinity}-operads} \label{subsec:bigOperads}

For certain arguments, it is technically inconvenient that the base $\cT$ does not admit pullbacks. We will thus need to consider an equivalent definition of a $\cT$-$\infty$-operad.

\begin{dfn} \label{dfn:BigFiniteTsets} Let $\Ar(\FinT)^{\tdeg}, \Ar(\FinT)^{\si}$ denote the wide subcategories of $\Ar(\FinT)$ on morphisms 
\[ \sigma = \left( \begin{tikzcd}[row sep=2em, column sep=2em]
U \ar{r} \ar{d} & X \ar{d} \\
V \ar{r} & Y
\end{tikzcd} \right) \]
such that $\ev_1(\sigma): V \to Y$ is a degenerate edge, resp. the induced morphism $U \to V \times_Y X$ is a summand inclusion. Then $(\Ar(\FinT); \Ar(\FinT)^{\si}, \Ar(\FinT)^{\tdeg})$ is a disjunctive triple, and we define
\[ \BigFinT \coloneq \Span(\Ar(\FinT); \Ar(\FinT)^{\si}, \Ar(\FinT)^{\tdeg}). \]
Evaluation at the target defines a structure map $\BigFinT \to \FinT^\op$, which is a cocartesian fibration.
\end{dfn}


\begin{definition} \label{def:BigCatPattern}
We define a categorical pattern $\wt{\mathfrak{P}}_{\cT}$ on $\BigFinT$ as follows. Let $\phi: U \to V$ be a morphism in $\FinT$ and $\Sigma = \{\sigma_1, ..., \sigma_n\}$ be a collection of commutative squares in $\FinT$
\[ \sigma_i = \left( \begin{tikzcd}[row sep=2em, column sep=2em]
U_i \ar{r}{\alpha_i} \ar{d}{\phi_i} & U \ar{d}{\phi} \\
V_i \ar{r}{\beta_i} & V
\end{tikzcd} \right) \]
such that $\alpha_i$ is a summand inclusion and the induced map $U_i \to V_i \times_V U$ is also a summand inclusion. Let $\chi_{\sigma_i}: \Delta^1 \to (\Ar(\FinT)^{\si})^\op \subset \BigFinT$ be the morphism corresponding to $\sigma_i$.

Suppose moreover that the summand inclusions $\alpha_i$ combine to yield an equivalence $\coprod_{1 \leq i \leq n} U_i \simeq U$. Let
\[ f_{\phi,\Sigma}: n^\lhd \to \BigFinT \]
denote the functor which selects the $n$ morphisms $\chi_{\sigma_1}$,  ..., $\chi_{\sigma_n}$. We then let
\[ \wt{\mathfrak{P}}_{\cT} = (\Ne, \All, \{ f_{\phi,\Sigma} \}) \]
where $\phi$ and $\Sigma$ range over all possible choices.
\end{definition}

\begin{definition} \label{def:BigModelStructure}
The \emph{$\cT$-operadic model structure} on the category $\sSet^+_{/(\BigFinT,\Ne)}$ is that defined by the categorical pattern $\wt{\mathfrak{P}}_{\cT}$ of \cref{def:BigCatPattern} according to \cite[Thm.~B.0.20]{HA}. We call the fibrant objects \emph{big $\cT$-$\infty$-operads}.

For any big $\cT$-$\infty$-operad $\widetilde{\cO}^\otimes \to \BigFinT$, we then define the \emph{$\cT$-operadic model structure} on $\sSet^+_{/(\widetilde{\cO}^\otimes,\Ne)}$ via the categorical pattern
$$\wt{\mathfrak{P}}_{\cO} = (\Ne, \All, \{ f_{x,\phi,\Sigma} \}),$$
where the $f_{x,\phi,\Sigma}: n^\lhd \to \widetilde{\cO}^\otimes$ range over all cocartesian sections of the $f_{\phi, \Sigma}$, and the $x$ in the notation denotes the value of $f_{x,\phi,\Sigma}$ on the cone point $v$.
\end{definition}


Given a big $\cT$-$\infty$-operad $\wt{\cO}^\otimes \to \BigFinT \to \FinT^\op$, we will let $\cO^\otimes \to \uFinpT \to \cT^{\op}$ denote its pullback along the inclusion $\cT^\op \subset \FinT^\op$. Clearly, $\cO^\otimes$ is a $\cT$-$\infty$-operad.

\begin{prp} \label{prp:QuillenEquivalenceOfBigAndSmall} Let $\wt{\cO}^\otimes$ be a big $\cT$-$\infty$-operad over $\BigFinT$ and consider the span
\[ \begin{tikzcd}[row sep=2em, column sep=2em]
(\wt{\cO}^\otimes, \Ne) & (\Ar^{\inert}(\wt{\cO}^\otimes) \times_{\wt{\cO}^\otimes} O^\otimes, \Ne) \ar{r}{\pr_{\cO^\otimes}} \ar{l}[swap]{\ev_0} & (\cO^\otimes, \Ne).
\end{tikzcd} \]
Then the adjunction
\[ \adjunct{(\pr_{\cO^\otimes})_! (\ev_0)^\ast}{\sSet^+_{/(\wt{\cO}^\otimes,\Ne)}}{\sSet^+_{/(\cO^\otimes,\Ne)}}{(\ev_0)_\ast (\pr_{\cO^\otimes})^\ast} \]
is a Quillen equivalence.
\end{prp}
\begin{proof} We first show that $(\pr_{\cO^\otimes})_! (\ev_0)^\ast \dashv (\ev_0)_\ast (\pr_{\cO^\otimes})^\ast$ is a Quillen adjunction. For this, it suffices to show that $ (\ev_0)_\ast (\pr_{\cO^\otimes})^\ast$ preserves fibrant objects. Examining the proof of \cite[Prop.~3.5(1)]{Exp2b}, we see that it implies 
\[ \ev_0: \Ar^{\inert}(\wt{\cO}^\otimes) \times_{\wt{\cO}^\otimes} \cO^\otimes \to \wt{\cO}^\otimes\] 
is a cartesian fibration, because any fiberwise active edge with target in $\cO^\otimes$ necessarily has source in $\cO^\otimes$. Therefore, the hypotheses of \cite[Thm.~B.4.2]{HA} are satisfied \emph{excluding those involving the maps $f_{x,\phi, \Sigma}$}. We deduce that $(\ev_0)_\ast (\pr_{\cO^\otimes})^\ast$ sends fibrant objects to fibrations over $\wt{\cO}^\otimes$ which are cocartesian over the inert edges. Given a $\cT$-$\infty$-operad $\cC^\otimes$, let $\wt{\cC}^\otimes = (\ev_0)_\ast (\pr_{\cO^\otimes})^\ast(\cC^\otimes)$. It remains to show that for every $f_{x,\phi,\Sigma}: n^\lhd \to \wt{\cO}^\otimes$,
\begin{enumerate}
	\item[(i)] the functor $n^\lhd \to \Cat$ classifying the cocartesian fibration $n^\lhd \times_{\wt{\cO}^\otimes} \wt{\cC}^\otimes$ is a limit diagram.
	\item[(ii)] For every cocartesian section $n^\lhd \to n^\lhd \times_{\wt{\cO}^\otimes} \wt{\cC}^\otimes$, the composite $n^\lhd \to \wt{\cC}^\otimes$ is a $f$-limit diagram for $f: \wt{\cC}^\otimes \to \wt{\cO}^\otimes$.	
\end{enumerate}

In fact, we only need to consider $f_{x,\phi,\Sigma}$ where $\Sigma$ is given by squares
\[ \begin{tikzcd}[row sep=2em, column sep=2em]
U_i \ar{r} \ar{d}{=} & U \ar{d}{\phi} \\
U_i \ar{r} & V
\end{tikzcd} \]

with $U_i$ an orbit, so we will suppose this in the remainder of the argument.

For (i), recall from \cite[Ex.~2.26]{Exp2} that the right Kan extension of a functor $\cC \to \Cat$ along $\cC \to \cD$ is modeled at the level of cocartesian fibrations by the push-pull construction involving the span
\[ \cD \ot \Ar(\cD) \times_{\cD} \cC \to \cC. \]
Therefore, $\wt{\cC}^\otimes_{\inert}$ is the right Kan extension of $\cC^\otimes_{\inert}$ along $\cO^\otimes_{\inert} \subset \wt{\cO}^\otimes_{\inert}$. In particular, given $x \in \wt{\cO}^\otimes$ over $[f_+: U_+ \ra V] \in \BigFinT$ and an orbit decomposition $V \simeq V_1 \coprod ... \coprod V_m$, let $x \to x_j$ be inert morphisms lifting the cocartesian morphisms $[U_+ \rightarrow V] \to [(U \times_V V_j)_+ \times V_j]$. Then we have an equivalence $\wt{\cC}^\otimes_{x} \simeq \prod_{1 \leq j \leq m} \cC^\otimes_{x_j}$, and postcomposing with the further decompositions of $\cC^\otimes_{x_j}$ given by orbit decompositions of $U \times_V V_j$ verifies (i).

For (ii), it suffices to prove that for every fiberwise active edge $\alpha: x \to y \in \wt{\cO}^\otimes$ over $[U'_+ \rightarrow V] \to [U_+ \rightarrow V]$, objects $\overline{x}$, $\overline{y} \in \wt{\cC}^\otimes$ over $x,y$, and identification $\overline{y} \simeq (\overline{y}_1,...,\overline{y}_n)$ induced by $f_{y,\phi,\Sigma}$, 
\[ \Map^{\alpha}_{\wt{\cC}^\otimes}(\overline{x},\overline{y}) \simeq \prod_{1 \leq i \leq n} \Map^{\alpha_i}_{\wt{\cC}^\otimes}(\overline{x},\overline{y}_i). \]
where $\alpha_i$ is the composition $x \to y \to y_i$.

However, for an orbit decomposition $V \simeq V_1 \coprod ... \coprod V_m$ and corresponding inert morphisms $\overline{x} \to \overline{x}_j$, $\overline{y} \to \overline{y}_j$, we have that
\[ \Map^{\alpha}_{\wt{\cC}^\otimes}(\overline{x},\overline{y}) \simeq \prod_{1 \leq j \leq m} \Map^{\alpha^j}_{C^\otimes}(\overline{x}_j,\overline{y}_j) \]
where $\alpha \simeq (\alpha^j: x_j \to y_j)$ under the same decomposition for mapping spaces in $\wt{\cO}^\otimes$. Using the known decompositions of mapping spaces in $\cC^\otimes$ then yields the claim.

Finally, it is easy to see that the induced adjunction of $\infty$-categories
\[ \adjunctb{(\Op^{\textrm{big}}_{\cT})_{/\wt{\cO}^\otimes}}{(\Op_{\cT})_{/\cO^\otimes}}  \]
is an equivalence because the unit and counit transformations are equivalences. Hence the Quillen adjunction is a Quillen equivalence.
\end{proof}

\begin{cor} \label{cor:BigVsSmallOperads} Let $\wt{\cO}^\otimes$ be a big $\cT$-$\infty$-operad over $\BigFinT$ and let $i: \cO^\otimes \to \wt{\cO}^\otimes$ denote the inclusion. Then we have a Quillen equivalence
\[ \adjunct{i_!}{\sSet^+_{/(\cO^\otimes,\Ne)}}{\sSet^+_{/(\wt{\cO}^\otimes,\Ne)}}{i^\ast} \]
\end{cor}
\begin{proof} $i$ is obviously compatible with the categorical patterns defining the model structures in the sense of \cite[B.2.8]{HA}, so $i_! \dashv i^\ast$ is a Quillen adjunction. Moreover, at the level of the underlying $\infty$-categories $i^\ast$ is left adjoint to the right Quillen functor of \cref{prp:QuillenEquivalenceOfBigAndSmall}. Hence the Quillen adjunction is a Quillen equivalence.
\end{proof}

\subsection{Monoidal envelopes}

In this subsection, we apply the theory of parametrized factorization systems \cite[\S 3]{Exp2b} to construct the $\cO$-monoidal envelope of any fibration of $\cT$-$\infty$-operads $\cC^{\otimes} \to \cO^{\otimes}$. First recall the notion of $\cT$-factorization system from \cite[Def.~3.1]{Exp2b} and the associated `total' factorization system of \cite[Def.~3.2]{Exp2b}.

\begin{exm} \label{exm:InertActiveFactorizationSystem}
We have the inert-active $\cT$-factorization system on $\uFinpT$ given fiberwise by the inert-active factorization system on $\FinpT[V]$ as in \cref{rem:FiberwiseInertActive}. Note that the definition of (possibly non-fiberwise) inert and active edges in $\uFinpT$ given initially in \cref{def:InertActiveEdges} then matches that of \cite[Def.~3.2]{Exp2b}. More generally, we have the inert-active $\cT$-factorization system on any $\cT$-$\infty$-operad $\cO^\otimes$. From this, we obtain the inert-fiberwise active factorization system on $\cO^\otimes$ itself.
\end{exm}

\begin{rem} \label{lem:InertEdgesCancellative} The inert edges in $\uFinpT$ and $\BigFinT$ satisfy the following right cancellation property: if we have a $2$-simplex
\[ \begin{tikzcd}[row sep=2em,column sep=2em]
& x_1 \ar{rd}{g} & \\
x_0 \ar{ru}{f} \ar{rr}{h} & & x_2
\end{tikzcd} \]
such that $f$ is inert, then $g$ is inert if and only if $h$ is inert. The `only if' direction is clear. To see the converse, note that by factoring $f$ as a cocartesian edge followed by a fiberwise inert edge, we may suppose that $f$ is of either form. Then by examining the composition of spans, we see that the claimed assertion reduces to the two-out-of-three property for equivalences in $(\FF_{\cT})^{/U}$, where $x_2$ lies over $U$.

In contrast, the active edges \emph{do not} satisfy the left cancellation property, because cocartesian edges do not. However, fiberwise active edges $\emph{do}$ satisfy the left cancellation property, just as they do in the theory of $\infty$-operads.
\end{rem}

\begin{ntn}
Given a $\cT$-$\infty$-operad $\cO^\otimes$, let $\Ar^\act(\cO^\otimes)$ denote the full $\cT$-subcategory of $\Ar(\cO^\otimes)$ on the active morphisms, and let
$$\Ar^\act_{\cT}(\cO^\otimes) = \cT^\op \times_{\Ar(\cT^\op)} \Ar^\act(\cO^\otimes).$$
\end{ntn}

\begin{dfn} \label{dfn:envelope}  Given a fibration of $\cT$-$\infty$-operads $p: \cC^\otimes \to \cO^\otimes$, the \emph{$\cO$-monoidal envelope} of $p$ is
$$\Env_{\cO,\cT}(\cC)^\otimes \coloneq \cC^\otimes \times_{\cO^\otimes} \Ar^\act_{\cT}(\cO^\otimes) \to \cO^\otimes.$$

If $\cO^\otimes = \uFinpT$, we will abbreviate $\Env_{\cO,\cT}(\cC)^\otimes$ as $\Env_{\cT}(\cC)^\otimes$ and refer to it as the \emph{$\cT$-symmetric monoidal envelope} of $\cC^\otimes$.
\end{dfn}

\begin{rem} For a $\cT$-$\infty$-operad $\cC^\otimes$, the underlying $\cT$-$\infty$-category of $\Env_{\cT}(\cC)^\otimes$ is $\cC^\otimes_{\act}$.
\end{rem}


\begin{prp} \label{prp:EnvelopeCocartesian} Let $p: \cC^\otimes \to \cO^\otimes$ be a fibration of $\cT$-$\infty$-operads. Then $\Env_{\cO,\cT}(\cC)^\otimes$ is a $\cO$-monoidal $\cT$-$\infty$-category.
\end{prp}
\begin{proof} We need to show that
\[ \ev_1: \cC^\otimes \times_{\cO^\otimes} \Ar^\act_{\cT}(\cO^\otimes) \to \cO^\otimes\]
is a cocartesian fibration of $\cT$-$\infty$-operads. By \cite[Prop.~3.5(2)]{Exp2b}, $\ev_1$ is a cocartesian fibration. We now seek to verify the criterion of \cref{prop:SimplifyOMonoidalDef} to finish the proof. Because $\cO^\otimes$ is a $\cT$-$\infty$-operad, for any object $y \in \cO^\otimes_{[U_+ \rightarrow V]}$, orbit decomposition $U \simeq U_1 \coprod ... \coprod U_n$, and inert edges $\rho^i: y \to y_i$ lifting the characteristic morphisms $\chi_{[U_i \subset U]}$, the $\rho^i$ induce an equivalence
\[ ((\cO^\otimes_{\act})_V)^{/y} \simeq \prod_{1 \leq i \leq n} ((\cO^\otimes_{\act})_{U_i})^{/y_i}. \]
Using that $\cC^\otimes \to \cO^\otimes$ is a fibration of $\cT$-$\infty$-operads, we have the further equivalence
\[ (\cC^\otimes_{\act})_V \times_{(\cO^\otimes_{\act})_V} ((\cO^\otimes_{\act})_V)^{/y} \simeq \prod_{1 \leq i \leq n} (\cC^\otimes_{\act})_{U_i} \times_{(\cO^\otimes_{\act})_{U_i}} ((\cO^\otimes_{\act})_{U_i})^{/y_i}. \]

Using that the fiberwise active edges are left cancellative, we identify the lefthand side with $(\cC^\otimes \times_{\cO^\otimes} \Ar^\act_{\cT}(\cO^\otimes))_y$, and similarly for the righthand side. The stated equivalence is then the one induced by the Segal maps, and we conclude that $\ev_1: \cC^\otimes \times_{\cO^\otimes} \Ar^\act_{\cT}(\cO^\otimes) \to \cO^\otimes$ is a cocartesian fibration of $\cT$-$\infty$-operads.
\end{proof}

\begin{prp} \label{prp:UniversalPropertyMonoidalEnvelope} Let $p: \cC^\otimes \to \cO^\otimes$ be a fibration of $\cT$-$\infty$-operads and let $q: \cD^\otimes \to \cO^\otimes$ be a cocartesian fibration of $\cT$-$\infty$-operads. Let $i: \cC^\otimes \subset \Env_{\cO,\cT}(\cC)^\otimes$ denote the inclusion of $\cC^\otimes$ into its $\cO$-monoidal envelope.
\begin{enumerate}
\item Precomposition by $i$ yields an equivalence
\[ i^\ast: \Fun^\otimes_{\cO,\cT}(\Env_{\cO,\cT}(\cC), \cD) \to \Alg_{\cO,\cT}(\cC,\cD). \]
\item We have an adjunction
\[ \adjunct{i_!}{\Alg_{\cO,\cT}(\cC,\cD)}{\Alg_{\cO,\cT}(\Env_{\cO,\cT}(\cC),\cD)}{i^\ast} \]\
where $i_!$ is the fully faithful inclusion of $\Fun^\otimes_{\cO,\cT}(\Env_{\cO,\cT}(\cC), \cD)$ under the equivalence of (1).
\end{enumerate}
\end{prp}
\begin{proof} This follows immediately from \cite[Thm.~3.6]{Exp2b}, using the inert-active $\cT$-factorization system on $\cO^\otimes$.
\end{proof}

\begin{cor} Let $\cO^\otimes$ be a $\cT$-$\infty$-operad. We have an adjunction
\[ \adjunct{\Env^\otimes_{\cO,\cT}}{\Op_{\cO,\cT}}{\Cat^\otimes_{\cO,\cT}}{\mathrm{U}}. \]
\end{cor}

\subsection{Subcategories and localization}

Let $\cC^\otimes \to \cO^\otimes$ be a fibration of $\cT$-$\infty$-operads. Given a full $\cT$-subcategory $\cD \subset \cC$, let $\cD^\otimes$ be the full $\cT$-subcategory of $\cC^\otimes$ on the objects of $\cD$ (using the Segal decompositions of the fibers of $\cC^\otimes$). Clearly, $\cD^\otimes \to \cO^\otimes$ is a fibration of $\cT$-$\infty$-operads, and the inclusion $\cD^\otimes \to \cC^\otimes$ is a morphism of $\cT$-$\infty$-operads over $\cO^\otimes$. In this subsection, we state conditions under which $\cD^\otimes$ inherits an $\cO$-monoidal structure from $\cC^\otimes$. Our presentation of these results parallels and extends those of \cite[\S 2.2.1]{HA}.

\begin{prp} Let $\cC^\otimes \to \cO^\otimes$ be a fibration of $\cT$-$\infty$-operads and let $\cD \subset \cC$ be a full $\cT$-subcategory. Suppose that for every fiberwise active edge $\alpha: x \to y$ in $\cO^\otimes_V$ with $y \in \cO_V$, the pushforward functor $\otimes_\alpha: \cC^\otimes_x \to \cC_y$ restricts to a functor $\cD^\otimes_x \to \cD_y$. Then $\cD^\otimes \to \cO^\otimes$ is a cocartesian fibration and the inclusion $\cD^\otimes \to \cC^\otimes$ is an $\cO$-monoidal $\cT$-functor.
\end{prp}
\begin{proof} This is immediate in light of the inert-fiberwise active factorization system on $\cC^\otimes$ (\cref{exm:InertActiveFactorizationSystem}).
\end{proof}

We say that a $\cT$-functor $L: \cC \to \cC$ is a \emph{$\cT$-localization} if for every object $V \in \cT$, $L_V$ is a localization functor. If we let $\cD$ denote the essential image of $L$, then by \cite[7.3.2.6]{HA} we have a $\cT$-adjunction
\[ \adjunct{L}{\cC}{\cD}{R} \]
where $R: \cD \to \cC$ is the inclusion. Given a $\cT$-localization $L: \cC \to \cC$, a morphism in $\cC$ is an \emph{$L$-equivalence} if it lies in a fiber $\cC_V$ and is a $L_V$-equivalence. Similarly, given a $\cT$-$\infty$-operad $\cC^\otimes$, a morphism in $\cC^\otimes$ is an \emph{$L$-equivalence} if it lies entirely in a fiber $\cC^\otimes_{[U_+ \rightarrow V]}$ and is a product of $L$-equivalences under the Segal decomposition of that fiber.

\begin{thm} \label{thm:CompatibleLocalization} Let $\cO^\otimes$ be a $\cT$-$\infty$-operad and $\cC^\otimes \to \cO^\otimes$ an $\cO$-monoidal $\cT$-$\infty$-category. Let $L: \cC \to \cC$ be a $\cT$-localization and let $\cD \subset \cC$ be its essential image. Suppose that for every fiberwise active edge $\alpha: x \to y \in \cO^\otimes_V$ with $y \in \cO_V$, the pushforward functor $\otimes_{\alpha}: \cC^\otimes_x \to \cC_y$ preserves $L$-equivalences. Then we have a relative adjunction over $\cO^\otimes$
\[ \adjunct{L^\otimes}{\cC^\otimes}{\cD^\otimes}{R^\otimes} \]
with $L^\otimes$ an $\cO$-monoidal functor (i.e., preserving cocartesian edges over $\cO^\otimes$) and $R^\otimes$ a lax $\cO$-monoidal functor (i.e., a morphism of $T$-$\infty$-operads), which prolongs the $\cT$-adjunction $\adjunct{L}{\cC}{\cD}{R}$.
\end{thm}
\begin{proof} This is immediate from the inert-fiberwise active factorization system on $\cC^\otimes$ (\cref{exm:InertActiveFactorizationSystem}) together with the criterion of \cite[Prop.~D.7]{BACHMANN2021}.
\end{proof}

\begin{rem} \label{rem:SMClocalization} In the case where $\cO^\otimes = \uFinpT$, the criterion of \cref{thm:CompatibleLocalization} amounts to
\begin{enumerate} \item For every object $V \in \cT$ and $Z \in \cC_V$, the functor 
\[ - \otimes Z: \cC_V \to \cC_V \]
preserves $L_V$-equivalences.
\item For every morphism $f: V \to W$ in $\cT$, the norm functor
\[ f_{\otimes}: \cC_V \to \cC_W \]
sends $L_V$-equivalences to $L_W$-equivalences.
\end{enumerate}
\end{rem}

\section{Parametrized Day convolution}

In this section, we construct a (partially-defined) internal hom for $\cT$-$\infty$-operads fibered over an arbitrary base $\cT$-$\infty$-operad $\cO^\otimes$: the \emph{$\cT$-Day convolution}. We first introduce the notion of an \emph{$\cO$-promonoidal $\cT$-$\infty$-category} $\cC^\otimes$, which is the analogue of a flat categorical fibration\footnote{Some authors also call this an exponentiable fibration to highlight its key property: a categorical fibration $\pi: \cC \to \cD$ is said to be \emph{flat} if the right adjoint $\pi_\ast$ to the pullback functor $\pi^\ast: \Cat_{/\cD} \to \Cat_{/\cC}$ exists.} in the context of $\cT$-$\infty$-operads. The $\cO$-promonoidal condition ensures the existence of the \emph{$p$-operadic coinduction} functor (\cref{cor:OperadicCoinduction2}), and given a $\cO$-promonoidal $\cT$-$\infty$-category $p: \cC^\otimes \to \cO^\otimes$ and any fibration $\cE^\otimes \to \cO^\otimes$ of $\cT$-$\infty$-operads, we may then use the $p$-operadic coinduction on the pullback of $\cE^\otimes$ over $\cC^\otimes$ to construct the \emph{$\cT$-Day convolution} (\cref{def:OperadicCoinductionAndDayConvolution})
\[ \wt{\Fun}_{\cO,\cT}(\cC,\cE)^\otimes \to \cO^\otimes. \]

We then state conditions on $\cE^\otimes$ under which $\wt{\Fun}_{\cO,\cT}(\cC,\cE)^\otimes$ is $\cO$-monoidal -- these amount to the existence of certain $\cT$-colimits as well as $\cT$-distributivity of the tensor product (\cref{thm:DayConvolutionCocartesian}).

\subsection{\texorpdfstring{$\cO$}{O}-promonoidal \texorpdfstring{$\cT$-$\infty$}{T-infinity}-categories and \texorpdfstring{$\cT$}{T}-Day convolution} \label{subsec:DayConvolution}

\begin{definition} \label{def:PromonoidalCategory}
Let $p: \cC^\otimes \to \cO^\otimes$ be a fibration of $\cT$-$\infty$-operads. We say that $\cC^\otimes$ is \emph{$\cO$-promonoidal} if for every $V \in \cT$, the functor $p_V: (\cC^\otimes_V)_{\act} \to (\cO^\otimes_V)_{\act}$ is a flat categorical fibration.
\end{definition}

\begin{example}
Suppose $(\cC^\otimes,p)$ is a $\cO$-monoidal $\cT$-$\infty$-category, so that $p$ is a cocartesian fibration. $\cC^\otimes$ is then $\cO$-promonoidal since cocartesian fibrations are flat \cite[Ex.~B.3.4]{HA}.
\end{example}

\begin{rem}
To understand the relevance of the $\cO$-promonoidal condition, the reader may find it useful to first review the nonparametrized story from \cite[\S 10]{Exp2b}. To our knowledge, the correct definition of an $\cO$-promonoidal $\infty$-category first appeared in Hinich's work \cite{HINICH2020107129} (and was misstated in \cite{BarwickGlasmanShah}).
\end{rem}

For technical reasons, to construct the $\cT$-Day convolution we will first work in the setting of big $\cT$-$\infty$-operads (i.e., so that the base is $\FinT^\op$ in place of $\cT^\op$). Let $\Ar^{\inert}(\wt{\cO}^\otimes)$ be notation for the full subcategory of $\Ar(\wt{\cO}^\otimes)$ on the inert edges. 

\begin{thm} \label{thm:OperadicCoinduction} Let $\wt{\cO}^\otimes$ and $\wt{\cC}^\otimes$ be big $\cT$-$\infty$-operads over $\BigFinT$ and let $p:\wt{\cC}^{\otimes} \to \wt{\cO}^{\otimes}$ be a fibration of big $\cT$-$\infty$-operads such that the restriction $p':\cC^\otimes \to \cO^\otimes$ is $\cO$-promonoidal. Consider the span of marked simplicial sets
\[ \begin{tikzcd}[row sep=2em, column sep=2em]
(\wt{\cO}^{\otimes},\Ne) & (\Ar^{\inert}(\wt{\cO}^\otimes) \times_{\wt{\cO}^\otimes} \wt{\cC}^\otimes,\Ne) \ar{l}[swap]{\ev_0} \ar{r}{\pr_{\wt{\cC}^\otimes}} & (\wt{\cC}^{\otimes},\Ne)
\end{tikzcd} \]
where by the middle marking $\Ne$ we mean those edges in $\Ar^{\inert}(\wt{\cO}^\otimes) \times_{\wt{\cO}^\otimes} \wt{\cC}^\otimes$ whose source in $\wt{\cO}^\otimes$ is inert and whose projection to $\wt{\cC}^\otimes$ is inert. Then the functor
\[ (\ev_0)_{\ast} \circ (\pr_{\wt{\cC}^\otimes})^{\ast}: \sSet^+_{/(\wt{\cC}^{\otimes},\Ne)} \to \sSet^+_{/(\wt{\cO}^{\otimes},\Ne)} \]
is right Quillen with respect to the $\cT$-operadic model structures of \cref{def:BigModelStructure}.
\end{thm}
\begin{proof} We verify the hypotheses of \cite[Thm.~B.4.2]{HA}.
\begin{enumerate}[leftmargin=*]
\item $\ev_0$ is flat by \cite[Lem.~10.1]{Exp2b} applied to the inert-fiberwise active factorization system on $\wt{\cO}^\otimes$ (\cref{exm:InertActiveFactorizationSystem}), noting that products of flat fibrations are flat in order to promote the flatness condition on $(p')_{f.\act}$ to $p_{f.\act}$.
\item It is obvious that inert edges are closed under composition and contain the equivalences.
\item Vacuously true since the categorical patterns we are looking at contain all $2$-simplices.
\item Suppose $e: x_0 \rightarrow y_0$ is an inert edge in $\wt{\cO}^\otimes$. Then as we saw in \cite[Prop.~3.5(1)]{Exp2b}, given an inert edge $y_0 \rightarrow y_1$ the cartesian lift of $e$ to an edge $e': [x_0 \rightarrow x_1] \rightarrow [y_0 \rightarrow y_1]$ in $\Ar^\inert(\wt{\cO}^\otimes)$ has $\ev_1(e'): x_1 \rightarrow y_1$ an equivalence. It follows that given a further lift of $y_0 \rightarrow y_1$ to an object $(y_0 \rightarrow y_1, c \in \wt{\cC}^\otimes_{y_1})$ in $\Ar^\inert(\wt{\cO}^\otimes) \times_{\wt{\cO}^\otimes} \wt{\cC}^\otimes$, $e$ admits a cartesian lift to an edge $e'': (x_0 \rightarrow x_1, c') \rightarrow (y_0 \rightarrow y_1, c)$ (with $c' \rightarrow c$ an equivalence).
\item Let $f_{x,\phi,\Sigma}: n^\lhd \to \wt{\cO}^\otimes$ be in the categorical pattern defining the $\cT$-operadic model structure on $\sSet^+_{/(\wt{\cO}^{\otimes},\Ne)}$. We claim that the pullback
\[ \pi: n^\lhd \times_{\wt{\cO}^\otimes} \Ar^{\inert}(\wt{\cO}^\otimes) \times_{\wt{\cO}^\otimes} \wt{\cC}^\otimes \to n^\lhd \]
is a cocartesian fibration.
Because inert edges are right cancellative (\cref{lem:InertEdgesCancellative}) and $\wt{\cC}^\otimes$ is cocartesian over the inert edges of $\wt{\cO}^\otimes$, it suffices to show that 
\[ n^\lhd \times_{\wt{\cO}^\otimes_{\inert}} \Ar(\wt{\cO}^\otimes_{\inert}) \to n^\lhd \]
is cocartesian, where $\wt{\cO}^\otimes_{\inert} \subset \wt{\cO}^\otimes$ is the wide subcategory with morphisms restricted to the inert edges. In fact, we will prove the stronger assertion that 
\[ \ev_0: \Ar(\wt{\cO}^\otimes_{\inert}) \to \wt{\cO}^\otimes_{\inert} \]
is cocartesian. For this, by \cite[6.1.1.1]{HTT}, it suffices to show that $\wt{\cO}^\otimes_{\inert}$ admits pushouts. Since the inert edges in $\wt{\cO}^\otimes$ are defined to be the cocartesian lifts of inert edges in $\BigFinT$, we may reduce to the case $\wt{\cO}^\otimes = \BigFinT$. It then suffices to show that $\Ar(\FinT)^\dagger$ (as in \cref{dfn:BigFiniteTsets}) admits pullbacks. So suppose we have a commutative cube
\setlength{\perspective}{2pt}
\[\begin{tikzcd}[row sep={40,between origins}, column sep={40,between origins}]
      &[-\perspective] W \times_U X \ar{rr}\ar{dd}\ar{dl} &[\perspective] &[-\perspective] X\vphantom{\times_{U}} \ar{dd}\ar{dl} \\[-\perspective]
     W \ar[crossing over]{rr} \ar{dd} & & U \\[\perspective]
      & Z \times_V Y  \ar{rr} \ar{dl} & &  Y\vphantom{\times_{V}} \ar{dl} \\[-\perspective]
    Z \ar{rr} && V \ar[from=uu,crossing over]
\end{tikzcd}\]
We want to show that if $W \to Z \times_V U$ is a summand inclusion, then $W \times_U X \to Z \times_Y X$ is a summand inclusion. To see this, consider the diagram
\[ \begin{tikzcd}[row sep=2em, column sep=2em]
W \times_U X \ar{r} \ar{d} & Z \times_Y X \ar{r} \ar{d} & X \ar{d} \\ 
W \ar{r} & Z \times_V U \ar{r} & U
\end{tikzcd} \]
The right square and outer rectangle are both pullback squares, so the left square is as well. Since summand inclusions are stable under pullback, the desired conclusion follows.

\item Let $s: n^\lhd \to n^\lhd \times_{\wt{\cO}^\otimes} \Ar^{\inert}(\wt{\cO}^\otimes) \times_{\wt{\cO}^\otimes} \wt{\cC}^\otimes$ be a cocartesian section of $\pi$ defined as above. Suppose that $s(\{v\}) = (x \xrightarrow{e} y \in \wt{\cO}^\otimes_{\inert}, \: c \in \wt{\cC}^\otimes_{y})$ and $e$ is a cocartesian lift of the inert morphism in $\BigFinT$ defined by the square
\[ \begin{tikzcd} [row sep=2em, column sep=2em]
U' \ar{r} \ar{d}{\phi'} & U \ar{d}{\phi} \\
V' \ar{r} & V
\end{tikzcd} \]
in $\FinT$. If $\Sigma = \{\sigma_1, ..., \sigma_n\}$ is given by squares $\sigma_i$
\[ \begin{tikzcd}[row sep=2em, column sep=2em]
U_i \ar{r} \ar{d} & U \ar{d}{\phi} \\
V_i \ar{r} & V
\end{tikzcd} \]
then let $\Sigma' = \{\sigma'_1, ..., \sigma'_n \}$ be given by the collection of squares $\sigma'_i$
\[ \begin{tikzcd}[row sep=2em, column sep=2em]
U_i \times_U U' \ar{r} \ar{d} & U' \ar{d}{\phi'} \\
V_i \times_V V' \ar{r} & V'.
\end{tikzcd} \]
Because summand inclusions are stable under pullback, the morphisms $U_i \times_U U' \to U'$ are summand inclusions, and clearly induce an equivalence $\coprod_{1 \leq i \leq n} U_i \times_U U' \simeq U'$. Therefore, the data of $(c,\phi',\Sigma')$ defines a morphism $f_{c,\phi',\Sigma'}: n^\lhd \to \wt{\cC}^\otimes$ which is part of the categorical pattern defining the $\cT$-operadic model structure on $\sSet^+_{/(\wt{\cC}^{\otimes},\Ne)}$. Moreover, by the analysis done in (5) we may identify the composite map
\[ n^\lhd \to n^\lhd \times_{\wt{\cO}^\otimes} \Ar^{\inert}(\wt{\cO}^\otimes) \times_{\wt{\cO}^\otimes} \wt{\cC}^\otimes \to \wt{\cC}^\otimes \]
with $f_{c,\phi',\Sigma'}$.

\item We check the following: suppose we have a commutative diagram
\[ \begin{tikzcd}[row sep=2em, column sep = 2em]
x_0 \ar{r} \ar{d} & x_1 \ar{d} \ar{r} & x_2 \ar{d} \\
y_0 \ar{r} & y_1 \ar{r} & y_2
\end{tikzcd} \]
in $\Ar^\inert(\wt{\cO}^\otimes)$ and $c_0 \to c_1 \to c_2$ in $\wt{\cC}^\otimes$ that covers $y_0 \to y_1 \to y_2$, such that $c_1 \to c_2$ is an equivalence (so $y_1 \to y_2$ is an equivalence), $x_1 \to x_2$ is inert, $x_0 \to x_1$ is an equivalence. Then $c_0 \to c_1$ is inert if and only if $c_0 \to c_2$ is inert. But this is clear from the definitions.

\item It suffices to check the following: suppose we have a commutative diagram
\[ \begin{tikzcd}[row sep=2em, column sep = 2em]
x_0 \ar{r} \ar{d} & x_1 \ar{d} \ar{r} & x_2 \ar{d} \\
y_0 \ar{r} & y_1 \ar{r} & y_2
\end{tikzcd} \]
in $\Ar^\inert(\wt{\cO}^\otimes)$ and $c_0 \to c_1 \to c_2$ in $\wt{\cC}^\otimes$ that covers $y_0 \to y_1 \to y_2$, such that $x_0 \to x_1$, $y_0 \to y_1$, and $c_0 \to c_1$ are inert. Then \{$x_1 \to x_2$, $y_1 \to y_2$, $c_1 \to c_2$\} are inert if and only if \{$x_0 \to x_2$, $y_0 \to y_2$, $c_0 \to c_2$\} are inert. But this follows from the right cancellativity of inert morphisms (\cref{lem:InertEdgesCancellative}).

\end{enumerate}
\end{proof}

Having passed to the big $\cT$-$\infty$-operads to construct the $\cT$-operadic coinduction and verify its properties, we now pass back to the usual formulation of $\cT$-$\infty$-operads.

\begin{cor} \label{cor:OperadicCoinduction2} Let $\cO^\otimes$ be a $\cT$-$\infty$-operad and let $(\cC^\otimes,p)$ be an $\cO$-promonoidal $\cT$-$\infty$-category. Consider the span diagram of marked simplicial sets
\[ \begin{tikzcd}[row sep=2em, column sep=2em]
(\cO^{\otimes},\Ne) & (\Ar^{\inert}(\cO^\otimes) \times_{\cO^\otimes} \cC^\otimes,\Ne) \ar{l}[swap]{\ev_0} \ar{r}{\pr_{\cC^\otimes}} & (\cC^{\otimes},\Ne).
\end{tikzcd} \]
Then the functor
\[ (\ev_0)_{\ast} \circ (\pr_{\cC^\otimes})^{\ast}: \sSet^+_{/(\cC^{\otimes},\Ne)} \to \sSet^+_{/(\cO^{\otimes},\Ne)} \]
is right Quillen with respect to the $\cT$-operadic model structures.
\end{cor}
\begin{proof} Combine \cref{thm:OperadicCoinduction}, \cref{prp:QuillenEquivalenceOfBigAndSmall}, and \cref{cor:BigVsSmallOperads}.
\end{proof}

\begin{dfn} \label{def:OperadicCoinductionAndDayConvolution}
In the situation of \cref{cor:OperadicCoinduction2}, given a fibration of $\cT$-$\infty$-operads $\cE^\otimes \to \cC^\otimes$, define the \emph{$p$-operadic coinduction} or \emph{$p$-norm} of $\cE^\otimes$ to be the $\cT$-$\infty$-operad
\[ (\Norm_p \cE)^\otimes \coloneq (\ev_0)_\ast (\pr_{\cC^\otimes})^{\ast} (\cE^{\otimes},\Ne), \]
given as a marked simplicial set fibered over $(\cO^{\otimes}, \Inert)$.

Given a fibration of $\cT$-$\infty$-operads $\cE^\otimes \to \cO^\otimes$, define the \emph{Day convolution} $\cT$-$\infty$-operad to be
\[ \widetilde{\Fun}_{\cO,\cT}(\cC,\cE)^\otimes \coloneq (\Norm_p p^\ast \cE)^\otimes. \]
If $\cO^\otimes = \uFinpT$, we will also denote $\widetilde{\Fun}_{\cO,\cT}(\cC,\cE)^\otimes$ as $\underline{\Fun}_{\cT}(\cC,\cE)^\otimes$.
\end{dfn}

\begin{prp} \label{lem:IdentifyingLeftAdjointOfCoinduction}
Let $p, q: \cD^\otimes \to \cO^\otimes$ be fibrations of $\cT$-$\infty$-operads. Then the functor
\[ \iota: (\cD^\otimes \times_{\cO^\otimes} \cC^\otimes, \Inert) \to (\cD^\otimes \times_{\cO^\otimes} \Ar^{\inert}(\cO^\otimes) \times_{\cO^\otimes} \cC^\otimes, \Ne) \]
induced by the identity section is a homotopy equivalence in $\sSet^+_{/(\cC^\otimes, \Ne)}$. Consequently, for $\cO$-promonoidal $(\cC^\otimes,p)$, $\iota^\ast$ induces an equivalence of $\cT$-$\infty$-categories
\[ \begin{tikzcd}[row sep=2em, column sep=2em]
\underline{\Alg}_{\cO,\cT}(\cD, \Norm_p \cE) \ar{r}{\simeq} & \underline{\Alg}_{\cC,\cT}(\cD \times_{\cO} \cC, \cE)
\end{tikzcd} \]
and an equivalence of $\infty$-categories
\[ \begin{tikzcd}[row sep=2em, column sep=2em]
\Alg_{\cO,\cT}(\cD, \Norm_p \cE) \ar{r}{\simeq} & \Alg_{\cC,\cT}(\cD \times_{\cO} \cC, \cE).
\end{tikzcd} \]
\end{prp}
\begin{proof}
Let $P: \cD^\otimes \times_{\cO^\otimes} \Ar^{\inert}(\cO^\otimes) \to \cD^\otimes$ be a cocartesian pushforward chosen so that $P|_{\cD^\otimes} = \id$, and let
$$P' = P \times \id_{\cC^\otimes}: \cD^\otimes \times_{\cO^\otimes} \Ar^{\inert}(\cO^\otimes) \times_{\cO^\otimes} \cC^\otimes \to \cD^\otimes \times_{\cO^\otimes} \cC^\otimes.$$
Then $P'$ respects the given markings, and as in the proof of \cite[Lem.~3.2(2)]{Exp2} we may construct an explicit homotopy $h: \id \to \iota \circ P'$ such that $h$ sends objects to fiberwise marked edges. This shows that $P$ is a marked homotopy inverse to $\iota$ and hence $\iota$ is a homotopy equivalence in $\sSet^+_{/(\cC^\otimes, \Ne)}$. The consequences then follow from the definition of the left adjoint to $\Norm_p$.
\end{proof}

\begin{cor} The Quillen adjunction of \cref{cor:OperadicCoinduction2} descends to an adjunction of $\infty$-categories
\[ \adjunct{p^\ast}{(\OpT)_{/\cO^\otimes}}{(\OpT)_{/\cC^\otimes}}{\Norm_p}. \]
In particular, if $(\cC^\otimes,p)$ is $\cO$-promonoidal, then the right adjoint to $p^\ast$ exists and is computed by $\Norm_p$.
\end{cor}

\begin{prp} \label{prp:underlyingCategoryOfDayConvolution} The underlying $\cT$-$\infty$-category of $\widetilde{\Fun}_{\cO,\cT}(\cC,\cE)^\otimes$ is the $\cT$-pairing construction $\widetilde{\Fun}_{\cO,\cT}(\cC,\cE)$ of \cite[Constr.~9.1]{Exp2}. In particular, if $\cO^\otimes = \uFinpT$, the underlying $\cT$-$\infty$-category of $\underline{\Fun}_{\cT}(\cC,\cE)^\otimes$ is $\underline{\Fun}_{\cT}(\cC,\cE)$.
\end{prp}
\begin{proof} Consider the diagram
\[ \begin{tikzcd}[row sep=2em, column sep=2em]
\Ar^\cocart(\cO) \times_{\cO} {\cC} \ar{r} \ar{d}{\iota} & \cO \ar{rd} \\
\cO \times_{\cO^\otimes} \Ar^\inert(\cO^\otimes) \times_{\cO^\otimes} \cC^\otimes \ar{r} \ar{d} & \Ar^\inert(\cO^\otimes) \times_{\cO^\otimes} \cC^\otimes \ar{r} \ar{d} & \cO^\otimes \\
\cO \ar{r} & \cO^\otimes
\end{tikzcd} \]
where $\iota$ is defined using the inclusions $\cO \subset \cO^\otimes$ and $\cC \subset \cC^\otimes$. Unwinding the definitions, to prove the claim it suffices to show that the map
\[ \Ar^\cocart(\cO) \times_{\cO} \cC \to \cO \times_{\cO^\otimes} \Ar^\inert(\cO^\otimes) \times_{\cO^\otimes} \cO \times_{\cO} \cC \]
is a homotopy equivalence (in $\sSet^+_{/\leftnat{\cO}}$ via the target map). But this is clear, since the inert edges in $\cO^\otimes$ with source and target in $\cO$ are, up to equivalence, precisely the cocartesian edges in $\cO$.
\end{proof}



\subsection{\texorpdfstring{$\cO$}{O}-monoidality of the \texorpdfstring{$\cT$}{T}-Day convolution}

We now establish $\cO$-monoidality of the $\cT$-Day convolution $\cT$-$\infty$-operad given appropriate conditions on the input $\cT$-$\infty$-operads. For this, we will use repeatedly the following criterion for when a fibration of $\cT$-$\infty$-operads is locally cocartesian or cocartesian.

\begin{lem} \label{lem:LocallyCocartesianCriterionForOperads} Let $p: \cC^\otimes \to \cO^\otimes$ be a fibration of $\cT$-$\infty$-operads. Suppose that for every fiberwise active edge $e: x \to y$ in $\cO^\otimes$ with $y \in \cO$, and $c \in \cO^\otimes$ with $p(c) = x$, there exists a locally cocartesian edge $f: c \to c'$ over $e$. Then $p$ is a locally cocartesian fibration.

Furthermore, suppose that for every composition of fiberwise active edges $x \xrightarrow{e} y \xrightarrow{e'} z$ in $\cO^\otimes$ with $z \in \cO$, and $c \in \cO^\otimes$ with $p(c) = x$, locally cocartesian lifts of $e$ to $f: c \to c'$ and $e'$ to $f':c' \to c''$ compose to yield a locally cocartesian edge $f'': c \to c''$. Then $p$ is a cocartesian fibration.
\end{lem}
\begin{proof} For the first assertion, using the inert-fiberwise active factorization system on $\cO^\otimes$ and that $p$ admits cocartesian lifts over inert edges, we reduce to checking that $p$ admits locally cocartesian lifts over fiberwise active edges. Then given any fiberwise active edge $e: x \to y$ in $\cO^\otimes$, we may use that $\cO^\otimes$ is a $\cT$-$\infty$-operad to obtain a decomposition of $e$ as $(e_i: x_i \to y_i)_{i \in I}$ for $y_i \in \cO$ under the identifications $\cO^\otimes_{y} \simeq \prod_{i \in I} \cO^{\otimes}_{y_i}$, $\cO^\otimes_{x} \simeq \prod_{i \in I} \cO^\otimes_{x_i}$, and $\Map_{\cO^\otimes_\act}(x,y) \simeq \prod_{i \in I} \Map_{\cO^\otimes_\act}(x_i,y_i)$. Suppose $c \in \cC^\otimes$ over $x$. Because $p$ is a fibration of $\cT$-$\infty$-operads, we get $c_i \in \cC^\otimes$ over $x_i$, and we may take the product of locally cocartesian lifts $c_i \to c_i'$ over the $e_i$ to obtain the desired locally cocartesian lift $c \to c'$ of $e$.

For the second assertion, we need to check that locally cocartesian edges compose to yield again a locally cocartesian edge. Note that for any commutative diagram in $\cC^\otimes$
\[ \begin{tikzcd}[row sep=2em, column sep=2em]
c \ar{r}{f} \ar{d}{g} & d \ar{d}{g'} \\
c' \ar{r}{f'} & d'
\end{tikzcd} \]
with $f$ locally cocartesian over a fiberwise active edge and $g, g'$ inert, then $f'$ is necessarily locally cocartesian. Using this, we can reduce to checking that locally cocartesian edges over fiberwise active edges compose. Then as before, we can further reduce to supposing that the last object lies in $\cO$.
\end{proof}

The next proposition allows us to understand $\cT$-Day convolution as a \emph{locally} cocartesian fibration over the base $\cT$-$\infty$-operad $\cO^\otimes$.

\begin{prp} \label{prp:DayConvolutionLocallyCocartesian} Let $(\cC^\otimes,p)$ be an $\cO$-promonoidal $\cT$-$\infty$-category and let $(\cE^\otimes, q)$ be an $\cO$-monoidal $\cT$-$\infty$-category in which for every object $x \in \cO_V$, the parametrized fiber $\cE_{\underline{x}}$ is $\cT^{/V}$-cocomplete.

\begin{enumerate}
\item The $\cT$-Day convolution $\widetilde{\Fun}_{\cO,\cT}(\cC,\cE)^\otimes \to \cO^\otimes$ is a locally cocartesian fibration, and for every $x \in \cO_V$, its parametrized fiber $\widetilde{\Fun}_{\cO,\cT}(\cC,\cE)_{\underline{x}}$ is $\cT^{/V}$-cocomplete.

\item Suppose $\overline{\alpha}$ is a fiberwise active edge in $\widetilde{\Fun}_{\cO,\cT}(\cC,\cE)^\otimes_V$ lifting $\alpha: x \to y$ in $\cO^\otimes_V$ with $y \in \cO_V$, which in turn lifts $f: U \to V$ in $\FinT$ (identified with the fiberwise active edge $[U_+ \rightarrow V] \to [V_+ \rightarrow V]$ in $\uFinpT$). The data of $\overline{\alpha}$ is given by a commutative diagram of $\cT^{/V}$-$\infty$-categories
\[ \begin{tikzcd}[row sep=2em, column sep=2em]
\{ x \} \times_{\cO^\otimes} \Ar^\inert(\cO^\otimes) \times_{\cO^\otimes} \cC^\otimes \ar{r}{F} \ar{d}{i} & \cE^\otimes_{\underline{V}} \ar{d}{q_{\underline{V}}} \\
\Delta^1 \times_{\alpha, \cO^\otimes} \Ar^\inert(\cO^\otimes) \times_{\cO^\otimes} \cC^\otimes \ar{ru}[swap]{H} \ar{r} & \cO^\otimes_{\underline{V}}
\end{tikzcd} \]
where we use the projections of the lefthand $\infty$-categories to $\{ V \} \times_{\cT^\op} \Ar(\cT^\op) \cong (\cT^{/V})^\op$ in order to pullback to the base $(\cT^{/V})^\op$.

Then $\overline{\alpha}$ is a locally cocartesian edge if and only if $H$ is a weak $q_{\underline{V}}$-$\cT^{/V}$-left Kan extension of $F$ in the sense of \cite[Def.~6.1]{Exp2b}.


\item Note that we have inclusions of full $\cT$-subcategories
\begin{align*} \cC^\otimes_{\underline{x}} \coloneq \{ x \} \times_{\cO^\otimes} \Ar^\cocart(\cO^\otimes) \times_{\cO^\otimes} \cC^\otimes & \subset \{ x \} \times_{\cO^\otimes} \Ar^\inert(\cO^\otimes) \times_{\cO^\otimes} \cC^\otimes \\
\cC^\otimes_{\underline{\alpha}} \coloneq \Delta^1 \times_{\alpha, \cO^\otimes} \Ar^\cocart(\cO^\otimes) \times_{\cO^\otimes} \cC^\otimes & \subset \Delta^1 \times_{\alpha,\cO^\otimes} \Ar^\inert(\cO^\otimes) \times_{\cO^\otimes} \cC^\otimes.
\end{align*}

Choose a cocartesian pushforward $P_\alpha: \cE^\otimes_{\underline{\alpha}} \to \cE_{\underline{y}}$ and consider the commutative diagram
\[ \begin{tikzcd}[row sep=2em, column sep=2em]
\cC^\otimes_{\underline{x}} \ar{r}{F} \ar{d}[swap]{i} & \cE^\otimes_{\underline{\alpha}} \ar{r}{P_{\alpha}} & \cE_{\underline{y}} \ar{d} \\
\cC^\otimes_{\underline{\alpha}} \ar{rr} \ar{ru}[swap]{H} & & (\cT^{/V})^\op
\end{tikzcd} \]
(where we abusively continue to write $H$ and $F$ for the canonical lifts of those functors to have codomains $\cE^\otimes_{\underline{\alpha}}$ and $\cE^\otimes_{\underline{x}} \subset \cE^\otimes_{\underline{\alpha}}$ respectively). Then $\overline{\alpha}$ is a locally cocartesian edge if and only if $P_{\alpha} \circ H$ is a $\cT^{/V}$-left Kan extension of $P_\alpha \circ F$.

Furthermore, if we let $G = (P_{\alpha} \circ H)|_{\cC_{\underline{y}}}: \cC_{\underline{y}} \to \cE_{\underline{y}}$, then we have the data of a diagram
\[ \begin{tikzcd}[row sep=2em, column sep=2em]
\cC^\otimes_{\underline{x}} \ar{r}{F} \ar{d}{\quad \Downarrow \eta}[swap]{\alpha_{\otimes}} & \cE^\otimes_{\underline{x}} \ar{r}{\alpha_{\otimes}} & \cE_{\underline{y}} \\
\cC_{\underline{y}} \ar[bend right=15]{rru}[swap]{G}
\end{tikzcd} \]
in which the natural transformation $\eta$ exhibits $G$ as a $\cT^{/V}$-left Kan extension of $\alpha_{\otimes} \circ F$ along $\alpha_{\otimes}$.
\end{enumerate}
\end{prp}
\begin{proof} (2): By definition, $\overline{\alpha}$ is a locally cocartesian edge if and only if $H$ is initial in the space of all such fillers. But if $H$ is a weak $q_{\underline{V}}$-$\cT^{/V}$-left Kan extension of $F$, then it is in particular initial. Conversely, provided that we know a weak $q_{\underline{V}}$-$\cT^{/V}$-left Kan extension of $F$ exists, then it necessarily coincides with the filler defined by $\overline{\alpha}$.

(1): We now show existence of these weak $q_{\underline{V}}$-$\cT^{/V}$-left Kan extensions. Let
\begin{align*} K & \coloneq \{ x \} \times_{\cO^\otimes} \Ar^\inert(\cO^\otimes) \times_{\cO^\otimes} \cC^\otimes \\
L & \coloneq \Delta^1 \times_{\alpha, \cO^\otimes} \Ar^\inert(\cO^\otimes) \times_{\cO^\otimes} \cC^\otimes.
\end{align*}

Note that any object in $L$ that is not in $K$ is either of the form $(y \ra z = p(c) \in \Ar^{\cocart}(\cO), \: c \in \cC)$ or $(y \ra z = p(c) \in \Ar^{\inert}(\cO^\otimes), \: c \in \cC^\otimes)$ with $z$ over $[\emptyset_+ \ra W] \in \uFinpT$. Let $L_0 \subset L$ denote the full $\cT$-subcategory excluding the second type of objects. Because the fiber of $\cE^\otimes$ over any object $[\emptyset_+ \ra W] \in \uFinpT$ is contractible, we may replace $L$ with $L_0$ and instead consider fillers $H_0: L_0 \to \cE^\otimes$. Moreover, note that there are no inert edges in $L_0$ not either cocartesian over $\cT^{\op}$ or in $K$, so any extension $H_0$ necessarily defines an edge of $\widetilde{\Fun}_{\cO,\cT}(\cC,\cE)^\otimes$.

Because both the additional objects in $L_0$ and morphisms between those objects lie over $\underline{y} \to \cO \subset \cO^\otimes$, by \cite[Thm.~6.2]{Exp2b} in conjunction with \cite[Prop.~5.8]{Exp2b} it suffices to have $\cE_{\underline{y}}$ be $\cT^{/V}$-cocomplete for a weak $q_{\underline{V}}$-$\cT^{/V}$-left Kan extension of $F$ to exist. This is ensured by our hypothesis.

By \cref{lem:LocallyCocartesianCriterionForOperads}, we see that the case we just considered suffices to show that $\widetilde{\Fun}_{\cO,\cT}(\cC,\cE)^\otimes \to \cO^\otimes$ is a locally cocartesian fibration. In addition, by \cref{prp:underlyingCategoryOfDayConvolution} and \cite[Prop.~9.7]{Exp2}, for every $x \in \cO_V$ we have an equivalence
$$\widetilde{\Fun}_{\cO,\cT}(\cC,\cE)_{\underline{x}} \simeq \underline{\Fun}_{\underline{V}}(\cC_{\underline{x}},\cE_{\underline{x}}),$$
and the latter $\cT^{/V}$-$\infty$-category is $\cT^{/V}$-cocomplete by the pointwise computation of $\cT^{/V}$-colimits in $\cT^{/V}$-functor categories.

(3): Observe that for $l = (y \ra z, c) \in \cC^\otimes_{\underline{\alpha}} \subset L$, $K \times_L L^{/\underline{l}} \simeq \cC^{\otimes}_{\underline{x}} \times_{\cC^\otimes_{\underline{\alpha}}} (\cC^\otimes_{\underline{\alpha}})^{/\underline{l}}$. By the pointwise formula for $\cT^{/V}$-left Kan extensions, the first part of the claim follows. The second part then follows by \cite[Rem.~2.14]{Exp2b}.
\end{proof}

As with ordinary Day convolution, we need an additional distributivity hypothesis on the target for $\cT$-Day convolution to be $\cO$-monoidal. We first recall the definition of a distributive functor (as originally formulated by the first author).

\begin{dfn}[{\cite[Def.~8.18]{Exp2b}}] 
Let $f: U \to V$ be a map of finite $\cT$-sets, let $\cC$ be a $\cT^{/U}$-$\infty$-category, and let $\cD$ be a $\cT^{/V}$-$\infty$-category. Let $F: \prod_f \cC = f_* \cC \to \cD$ be a $\cT^{/V}$-functor. Then we say that $F$ is \emph{distributive} if for every pullback square
\[ \begin{tikzcd}
U' \ar{r}{f'} \ar{d}{g'} & V' \ar{d}{g} \\
U \ar{r}{f} & V
\end{tikzcd} \]
of finite $\cT$-sets and $\cT^{/U'}$-colimit diagram $\overline{p}: \cK^{\underline{\rhd}} \to g'^* \cC$, the $\cT^{/V'}$-functor
\[ (f'_* \cK)^{\underline{\rhd}} \xto{\can} f'_* (\cK^{\underline{\rhd}}) \xto{f'_* \overline{p}} f'_* g'^* \cC \simeq g^* f_* \cC \xto{g^* F} g^* \cD \]
is a $\cT^{/V'}$-colimit diagram.
\end{dfn}

\begin{dfn} \label{def:DistributiveMonoidalCategory}
Let $\cC^\otimes$ be a $\cO$-monoidal $\cT$-$\infty$-category and suppose that for all $y \in \cO_V$, $\cC_{\underline{y}}$ is $\cT^{/V}$-cocomplete. We say that $\cC^\otimes$ is \emph{distributive} if for every fiberwise active edge $\alpha: x \to y$ lifting $[U_+ \rightarrow V] \to [V_+ \rightarrow V]$ (corresponding to $f: U \to V$ in $\FinT$), the associated pushforward $\cT^{/V}$-functor
\[ \alpha_{\otimes}: \cC^\otimes_{\underline{x}} \to \cC_{\underline{y}} \]
is distributive. Here, for this condition to be sensible, we use that for an orbit decomposition $U \simeq U_1 \coprod ... \coprod U_n$ and $n$ cocartesian morphisms $x \to x_i$ lifting the characteristic maps $\chi_{[U_i \subset U]}: [U_+ \rightarrow V] \to [(U_i)_+ \rightarrow U_i]$, the parametrized fiber $\cC^\otimes_{\underline{x}}$ is identified with $\prod_f \left( \cC_{\underline{x_1}} \coprod ... \coprod \cC_{\underline{x_n}} \right)$ by the $\cT$-Segal condition (cf. \cref{exm:normParametrizedFunctor}).

Also note that in particular, for each morphism $\alpha: x \to y$ in $\cO_V$, the condition that the pushforward $\cT^{/V}$-functor $\alpha_{\otimes}: \cC_{\underline{x}} \to \cC_{\underline{y}}$ is distributive is equivalent to $\alpha_{\otimes}$ strongly preserving all small $\cT^{/V}$-colimits.
\end{dfn}


The following proposition furnishes some examples of distributive $\cT$-symmetric monoidal $\cT$-$\infty$-categories.

\begin{prp}
Let $\cC$ be a cocomplete $\infty$-category with finite products such that the products commute with colimits separately in each variable. Let $f: U \to U'$ be a morphism of finite $\cT$-sets. Then the product $\cT^{/U'}$-functor
\[ \mu: f_* f^* \underline{\cC}_{\cT^{/U'}} \to \underline{\cC}_{\cT^{/U'}} \]
is distributive. Consequently, the $\cT$-cartesian $\cT$-symmetric monoidal structure on $\underline{\cC}_{\cT}$ is $\cT$-distributive.
\end{prp}
\begin{proof}
By the universal property of the category of $\cT^{/U'}$-objects (\cite[Prop.~3.10]{Exp2}), $\mu$ can be identified with a functor
\[ \mu^{\dagger}: f_* f^* \underline{\cC}_{\cT^{/U'}} \to \cC\]
such that its restriction to the fiber over $[W \ra U'] \in \cT^{/U'}$ is the functor
\[\prod_{V\in\mathrm{Orbit}(U\times_{U'}W)} \Fun(\underline{V},\cC ) \xrightarrow{\prod \ev_V} \prod_{V\in \mathrm{Orbit}(U\times_{U'}W)} \cC \xrightarrow{\times} \cC\,.\]
For the proof of distributivity we can assume without loss of generality that $U'$ is an orbit. Suppose $p: \cK^{\underline{\rhd}} \to \underline{\cC}_{\cT^{/U}}$ is a $\cT^{/U}$-colimit diagram and let $p^{\dagger}: \cK^{\underline{\rhd}} \to \cC$ be the functor under the equivalence of \cite[Prop.~3.10]{Exp2}. By \cite[Prop.~5.5]{Exp2}, any such $p$ is $\cT^{/U}$-colimit diagram if and only if for every $[V \ra U]\in \cT^{/U}$ the functor $p^{\dagger}_V :\cK_{[V\ra U]}^{\rhd} \to \cC$ given by restriction to the fiber is a colimit diagram. By \cite[Prop.~5.5]{Exp2} again, it suffices to prove that the diagram
\[ (f_* \cK)^{\underline{\rhd}} \to f_* \underline{\cC}_{\cT^{/U}} \xto{\mu^{\dagger}} \cC \]
is a colimit diagram when restricted to each fiber. But the restriction to the fiber above $[W \ra U']\in \cT^{/U'}$ is
        \[\left(\prod_{V\in\mathrm{Orbit}(U\times_{U'}W)} \cK_V \right)^\triangleright \to \prod_{V\in\mathrm{Orbit}(U\times_{U'}W)} \cK_V^\triangleright \xrightarrow{\prod p_V} \prod_{V\in\mathrm{Orbit}(U\times_{U'}W)} \cC \xrightarrow{\times} \cC, \]
        which is a colimit diagram since the cartesian product in $\cC$ commutes with colimits separately in each variable.
\end{proof}

We now return to $\cT$-Day convolution and prove the main result of this subsection. 

\begin{thm} \label{thm:DayConvolutionCocartesian}
In the situation of \cref{prp:DayConvolutionLocallyCocartesian}, suppose moreover that $q: \cE^\otimes \to \cO^\otimes$ is distributive. Then $\widetilde{\Fun}_{\cO,\cT}(\cC,\cE)^\otimes \to \cO^\otimes$ is a cocartesian fibration and $\widetilde{\Fun}_{\cO,\cT}(\cC,\cE)^\otimes$ is distributive.
\end{thm}
\begin{proof} Let
\[ \begin{tikzcd}[row sep=2em, column sep=2em]
& y \ar{rd}{\beta} & \\
x \ar{ru}{\alpha} \ar{rr}{\gamma} & & z
\end{tikzcd} \]
be a $2$-simplex $\sigma$ of fiberwise active edges in $\cO^\otimes_W$ with $z \in \cO_W$, which covers
\[ \begin{tikzcd}[row sep=2em, column sep=2em]
& V \ar{rd}{g} & \\
U \ar{ru}{f} \ar{rr}{h} & & W
\end{tikzcd} \]
in $\FinT$ (viewed as a $2$-simplex in $(\uFinpT)^{\act}_W$). We need to verify that given a lift of $\sigma$ to a $2$-simplex $\overline{\sigma}$
\[ \begin{tikzcd}[row sep=2em, column sep=2em]
& \overline{y} \ar{rd}{\overline{\beta}} & \\
\overline{x} \ar{ru}{\overline{\alpha}} \ar{rr}{\overline{\gamma}} & & \overline{z}
\end{tikzcd} \]
in $\widetilde{\Fun}_{\cO,\cT}(\cC,\cE)^\otimes$, if $\overline{\alpha}$ and $\overline{\beta}$ are locally cocartesian edges then $\overline{\gamma}$ is a locally cocartesian edge.


Suppose that $V \simeq \coprod_{1 \leq i \leq n} V_i$ is an orbit decomposition of $V$ with respect to which $\alpha$ decomposes as $\{ \alpha_i: x_i \to y_i\}_{1 \leq i \leq n}$. Then the locally cocartesian edge $\overline{\alpha}$ corresponds to $n$ commutative diagrams of $\cT^{/V_i}$-$\infty$-categories
\[ \begin{tikzcd}[row sep=2em, column sep=2em]
\cC^\otimes_{\underline{x_i}} \ar{r}{F_{x_i}} \ar{d} & \cE_{\underline{y_i}} \\
\cC^\otimes_{\underline{\alpha_i}} \ar{ru}[swap]{F_{\alpha_i}} 
\end{tikzcd} \]
in which $F_{\alpha_i}$ is a $\cT^{/V_i}$-left Kan extension of $F_{x_i}$. From this, we obtain the commutative diagram of $\cT^{/W}$-$\infty$-categories 
\[ \begin{tikzcd}[row sep=2em, column sep=2em]
\prod_g (\coprod_{1 \leq i \leq n} \cC^\otimes_{\underline{x_i}}) \simeq \cC^\otimes_{\underline{x}} \ar{r}{F_x} \ar{d} & \prod_g (\coprod_{1 \leq i \leq n} \cE_{\underline{y_i}}) \simeq \cE^\otimes_{\underline{y}} \ar{r}{\beta_{\otimes}} & \cE_{\underline{z}} \\
\prod_g (\coprod_{1 \leq i \leq n} \cC^\otimes_{\underline{\alpha_i}}) \simeq \cC^\otimes_{\underline{\alpha}} \ar{ru}{F_{\alpha}} \ar[bend right=10]{rru}[swap]{\beta_{\otimes} \circ F_\alpha}
\end{tikzcd} \]
in which $\beta_{\otimes} \circ F_{\alpha}$ is a $\cT^{/W}$-left Kan extension of $\beta_{\otimes} \circ F_x$, invoking the hypothesis that $\cE^\otimes$ is distributive.

On the other hand, the locally cocartesian edge $\overline{\beta}$ corresponds to a commutative diagram of $\cT^{/W}$-$\infty$-categories
\[ \begin{tikzcd}[row sep=2em, column sep=2em]
\cC^\otimes_{\underline{y}} \ar{r}{F_y} \ar{d} &  \cE_{\underline{z}} \\
\cC^\otimes_{\underline{\beta}} \ar{ru}[swap]{F_\beta}
\end{tikzcd} \]
in which $F_y$ is the restriction of $\beta_{\otimes} \circ F_{\alpha}$ along the inclusion $\cC^\otimes_{\underline{y}} \subset \cC^\otimes_{\underline{\alpha}}$ and $F_\beta$ is a $\cT^{/W}$-left Kan extension of $F_y$. Combining these two diagrams, we obtain a commutative diagram of $\cT^{/W}$-$\infty$-categories
\[ \begin{tikzcd}[row sep=2em, column sep=3em]
& \cC^\otimes_{\underline{x}} \ar{r}{F_x} \ar{d} & \cE^\otimes_{\underline{y}} \ar{d}{\beta_{\otimes}} \\
\cC^\otimes_{\underline{y}} \ar{r} \ar{d} & \cC^\otimes_{\underline{\alpha}} \ar{r}{\beta_{\otimes} \circ F_\alpha} \ar{d} & \cE_{\underline{z}} \\
\cC^\otimes_{\underline{\beta}} \ar{r} \ar{rru}[near start]{F_\beta} & \cC^\otimes_{\underline{\sigma}} \ar[dotted]{ru}[swap]{F_\sigma}
\end{tikzcd} \]
where $\cC^\otimes_{\underline{\sigma}} \coloneq \Delta^2 \times_{\sigma, \cO^\otimes} \Ar^\cocart(\cO^\otimes) \times_{\cO^\otimes} \cC^\otimes$. Because $\cC^\otimes$ is $\cO$-promonoidal and $\sigma$ is a $2$-simplex of fiberwise active edges, the lefthand square is a pushout square of $\cT^{/W}$-$\infty$-categories, and the dotted $\cT^{/W}$-functor $F_{\sigma}$ is obtained from gluing together $\beta_{\otimes} \circ F_\alpha$ and $F_\beta$. Indeed, $F_\sigma$ corresponds to the $2$-simplex $\overline{\sigma}$ in $\widetilde{\Fun}_{\cO,\cT}(\cC,\cE)^\otimes$. By \cref{lem:KanExtendToPushout}, $F_\sigma$ is a $\cT^{/W}$-left Kan extension of $\beta_{\otimes} \circ F_\alpha$. By transitivity of $\cT^{/W}$-left Kan extensions, $F_\sigma$ is a $\cT^{/W}$-left Kan extension of $\beta_{\otimes} \circ F_x$, and the restriction $F_\gamma$ of $F_\sigma$ to $\cC^\otimes_{\underline{\gamma}} \subset \cC^\otimes_{\underline{\sigma}}$ is also a $\cT^{/W}$-left Kan extension of $\beta_{\otimes} \circ F_x$. But this exactly means that $\overline{\gamma}$ is a locally cocartesian edge. Finally, by \cref{lem:LocallyCocartesianCriterionForOperads} this suffices to show that $\widetilde{\Fun}_{\cO,\cT}(\cC,\cE)^\otimes \to \cO^\otimes$ is a cocartesian fibration.

To show that $\widetilde{\Fun}_{\cO,\cT}(\cC,\cE)^\otimes$ is distributive, we check the definition. So suppose that $\alpha: x \to y$ is a fiberwise active edge in $\cO^\otimes_V$ with $y \in \cO_V$, which lifts $f: U \to V$ in $\FinT$. Let $U \simeq U_1 \coprod ... \coprod U_n$ be an orbit decomposition and suppose we have $\cT^{/U_i}$-colimits
\[ \begin{tikzcd}[row sep=2em, column sep=3em]
\cK_i \ar{r}{p_i} \ar{d} & \underline{\Fun}_{\cT^{/U_i}}(\cC_{\underline{x_i}}, \cE_{\underline{x_i}}) \\
(\cT^{/U_i})^\op \ar{ru}[swap]{q_i}
\end{tikzcd} \]

(Here and throughout we suppress the data of the natural transformations.) By \cite[Prop.~9.17]{Exp2} and its proof, these are adjoint to $\cT^{/U_i}$-left Kan extensions
\[ \begin{tikzcd}[row sep=2em, column sep=3em]
\cK_i \times_{(\cT^{/U_i})^\op} \cC_{\underline{x_i}} \ar{r}{p'_i} \ar{d} & \cE_{\underline{x_i}} \\
\cC_{\underline{x_i}} \ar{ru}[swap]{q_i'}
\end{tikzcd} \]

Let $\cK = \coprod_{1 \leq i \leq n} \cK_i$, $p = \coprod_{1 \leq i \leq n} p_i$, and $q = \coprod_{1 \leq i \leq n} q_i$, and the same for $\cK',p',q'$. We need to show that
\[ \begin{tikzcd}[row sep=2em, column sep=3em]
\prod_f \cK \ar{r}{\prod_f p} \ar{d} & \widetilde{\Fun}_{\cO,\cT}(\cC,\cE)^\otimes_{\underline{x}} \ar{r}{\alpha_{\otimes}} & \underline{\Fun}_{\cT^{/V}}(\cC_{\underline{y}}, \cE_{\underline{y}}) \\
(\cT^{/V})^\op \ar[bend right=5]{rru}[swap]{\alpha_{\otimes} \circ \prod_f q}
\end{tikzcd} \]

is a $\cT^{/V}$-colimit. Equivalently, we need to show that in the diagram

\[ \begin{tikzcd}[row sep=2em, column sep=3em]
\prod_f \cK \times_{(\cT^{/V})^\op} \cC^\otimes_{\underline{x}} \ar{r}{(\prod_f p)'} \ar{d} & \cE^\otimes_{\underline{x}} \ar{r}{\alpha_{\otimes}} & \cE_{\underline{y}} \\
\prod_f \cK \times_{(\cT^{/V})^\op} \cC^\otimes_{\underline{\alpha}}  \ar[bend right=10]{rru}{(\alpha_{\otimes} \circ \prod_f p)'} \ar{d} \\
\cC^\otimes_{\underline{\alpha}} \ar[bend right=15]{rruu}[swap]{(\alpha_{\otimes} \circ \prod_f q)'}
\end{tikzcd} \]

the $\cT^{/V}$-functor $(\alpha_{\otimes} \circ \prod_f q)'$ is a $\cT^{/V}$-left Kan extension of $(\alpha_{\otimes} \circ \prod_f p)'$. Here we change the domain from $\cC^\otimes_{\underline{y}}$ to $\cC^\otimes_{\underline{\alpha}}$ because we are not supposing that $\cC^\otimes$ is $\cO$-monoidal. Note that $(\alpha_{\otimes} \circ \prod_f p)'$ is by definition a $\cT^{/V}$-left Kan extension of $\alpha_{\otimes} \circ (\prod_f p)'$, so it suffices to show that $(\alpha_{\otimes} \circ \prod_f q)'$ is a $\cT^{/V}$-left Kan extension of $\alpha_{\otimes} \circ (\prod_f p)'$.

Because $\cE^\otimes$ is distributive, we have that
\[ \begin{tikzcd}[row sep=2em, column sep=3em]
\prod_f \cK \times_{(\cT^{/V})^\op} \cC^\otimes_{\underline{x}} \ar{r}{\prod_f (p')} \ar{d} & \cE^\otimes_{\underline{x}} \ar{r}{\alpha_{\otimes}} & \cE_{\underline{y}} \\
\cC^\otimes_{\underline{x}} \ar[bend right=10]{rru}[swap]{\alpha_{\otimes} \circ \prod_f (q')}
\end{tikzcd} \]
is a $\cT^{/V}$-left Kan extension. By definition, $(\alpha_{\otimes} \circ \prod_f q)'$ is a $\cT^{/V}$-left Kan extension of $\alpha_{\otimes} \circ \prod_f (q')$ along the inclusion $\cC^\otimes_{\underline{x}} \subset \cC^\otimes_{\underline{\alpha}}$. By transitivity of $\cT^{/V}$-left Kan extensions, we are done.
\end{proof}

\begin{lem} \label{lem:KanExtendToPushout} Suppose we have a diagram of $\cT$-$\infty$-categories
\[ \begin{tikzcd}[row sep=2em, column sep=2em]
\cA \ar{r}{i} \ar{d}{\phi} & \cB \ar{r}{F} \ar{d}{\phi} & \cE \\
\cC \ar{r}{i} & \cD \ar{ru}[swap]{H}
\end{tikzcd} \]
such that the lefthand square is a pushout square in which every $\cT$-functor is an inclusion. Further suppose that the relevant $\cT$-colimits exists for the $\cT$-left Kan extensions of $F$ along $\phi$ and of $F \circ i$ along $\phi$ to exist. Then $H$ is a $\cT$-left Kan extension of $F$ if and only if $H \circ i$ is a $\cT$-left Kan extension of $F \circ i$.
\end{lem}
\begin{proof} Consider the commutative diagram of $\cT$-$\infty$-categories
\[ \begin{tikzcd}[row sep=2em, column sep=3em]
\cP \ar{r} \ar{d} & \underline{\Fun}_{\cT}(\cD,\cE) \ar{r} \ar{d} & \underline{\Fun}_{\cT}(\cC,\cE) \ar{d} \\
\cT^{\op} \ar{r}{\sigma_F} \ar[bend left,dotted]{u}{\sigma} \ar[bend right=15]{rr}[swap]{\sigma_{F i}} & \underline{\Fun}_{\cT}(\cB,\cE) \ar{r} & \underline{\Fun}_{\cT}(\cA,\cE)
\end{tikzcd} \]
in which every square is a pullback square. Then $\cP$ is identified with
\[ \cT^{\op} \times_{\sigma_{F i}, \underline{\Fun}_{\cT}(\cA,\cE)} \underline{\Fun}_{\cT}(\cC,\cE) \simeq \cT^{\op} \times_{\sigma_F, \underline{\Fun}_{\cT}(\cB,\cE)} \underline{\Fun}_{\cT}(\cD,\cE) \]
and the cocartesian section $\sigma$ is equivalently determined by a $\cT$-functor $H$ extending $F$ or a $\cT$-functor $G$ extending $F \circ i$. Moreover, because of our hypothesis on the existence of the relevant $\cT$-colimits, $\sigma$ is an $\cT$-initial object in $\cP$ if and only if $H$ is a $\cT$-left Kan extension of $F$ along $\phi$ if and only if $G$ is a $\cT$-left Kan extension of $F \circ i$ along $\phi$.
\end{proof}

\begin{exm}[Smash product]
Define a functor
\[ (\underline{\Delta^1})^\otimes: \Span(\FinT) \to \text{ho} \Span(\FinT) \to \Cat_1 \]
by sending an object $U = \coprod_I U_i$ (decomposed as a disjoint union of orbits) to $\prod_I \Delta^1$, and a morphism $f: \coprod_I U_i \to \coprod_J V_j$, $\phi: I \to J$ contravariantly to the restriction functor $\phi^\ast: \prod_J \Delta^1 \to \prod_I \Delta^1$ and covariantly to the product over $J$ of
\[ \text{min}: \prod_{I_j} \Delta^1 \to \Delta^1, \: (x_i) \mapsto \text{min}(x_i)  \]
if $I_j$ is nonempty, and $1: \Delta^0 \to \Delta^1$ otherwise. (One easily verifies the base-change condition, so this indeeds defines a functor.)

Let $(\Delta^1 \times \cT^\op)^\otimes$ denote the resulting $\cT$-symmetric monoidal $\cT$-$\infty$-category under the equivalence of \cref{thm:TwoPresentationsOfTSMCs}. Let $\cC^\otimes$ be a distributive $\cT$-symmetric monoidal $\cT$-$\infty$-category. Then by \cref{thm:DayConvolutionCocartesian}, we have that $\underline{\Fun}_{\cT}(\Delta^1 \times \cT^\op, \cC)^\otimes$ is a distributive $\cT$-symmetric monoidal $\cT$-$\infty$-category. Moreover, the fiberwise tensor products admit a simple description because the underlying $\cT$-category of the source is constant. Namely, for the fold map $\nabla: U \coprod U \to U$, the tensor product
\[ \otimes: \Fun(\Delta^1, \cC_U) \times \Fun(\Delta^1, \cC_U) \to \Fun(\Delta^1, \cC_U) \]
is given by taking the cartesian product into $\Fun(\Delta^1 \times \Delta^1, \cC_U \times \cC_U)$, postcomposition by $\otimes: \cC_U \times \cC_U \to \cC_U$, and then left Kan extension along $\text{min}: \Delta^1 \times \Delta^1 \to \Delta^1$ (which is computed by taking colimits fiberwise because $\text{min}$ is a cocartesian fibration).


Now suppose in addition that $\cC^\otimes$ is \emph{$\cT$-cartesian} $\cT$-symmetric monoidal, so the tensor product on the fibers of $\underline{\Fun}_{\cT}(\Delta^1 \times \cT^\op, \cC)$ is the pushout product. Let $\cC_{\ast}$ denote the full $\cT$-subcategory of $\underline{\Fun}_{\cT}(\Delta^1 \times \cT^\op, \cC)$ given over an object $U \in \cT$ by those functors $\Delta^1 \to \cC_U$ which take $0$ to a final object $\ast$ of $\cC_U$. The inclusion $\cC_\ast \subset \underline{\Fun}_{\cT}(\Delta^1 \times \cT^\op, \cC)$ admits a left adjoint $L$, which over an object $U \in \cT$ takes a functor $F: \Delta^1 \to \cC_U$ to
\[ L(F): \Delta^1 \to \cC_U, \: \goesto{[0 \rightarrow 1]}{[\ast \rightarrow  F(1)/F(0)]}. \]

$L$ is a $\cT$-localization functor, so to descend the cartesian $\cT$-symmetric monoidal structure on $\cC$ to the \emph{smash product} on $\cC_{\ast}$, we can check the criterion of \cref{thm:CompatibleLocalization} (or rather, \cref{rem:SMClocalization}):
\begin{enumerate} \item For the fold map $\nabla: U \coprod U \to U$ and an object $[z_0 \to z_1] \in \Fun(\Delta^1, \cC_U)$,
\[ - \otimes [z_0 \to z_1]: \Fun(\Delta^1, \cC_U) \to \Fun(\Delta^1, \cC_U) \]
preserves $L_U$-equivalences. Indeed, let
\[ \begin{tikzcd}[row sep=4ex, column sep=4ex, text height=1.5ex, text depth=0.25ex]
x_0 \ar{r} \ar{d} & y_0 \ar{d} \\
x_1 \ar{r} & y_1 
\end{tikzcd} \]
be an $L_U$-equivalence, i.e. $x_1/x_0 \to y_1/y_0$ is an equivalence in $\cC_U$. We have an equivalence
\[ \left( x_1 \times z_1 \right) / \left(x_0 \times z_1 \cup_{x_0 \times z_0} x_1 \times z_0 \right) \simeq \left( \frac{x_1 \times z_1}{x_0 \times z_1} \right) / \left( \frac{x_1 \times z_0}{x_0 \times z_0} \right). \]
Using that we have a diagram of pushout squares
\[ \begin{tikzcd}[row sep=4ex, column sep=4ex, text height=1.5ex, text depth=0.25ex]
x_0 \times z_j \ar{r} \ar{d} & z_j \ar{r} \ar{d} & \ast \ar{d} \\
x_1 \times z_j \ar{r} & x_1/x_0 \times z_j \ar{r} & (x_1 \times z_j)/(x_0 \times z_j)
\end{tikzcd} \]
(and similarly for $y$) by cartesian closedness, we deduce that 
\[ \left( \frac{x_1 \times z_1}{x_0 \times z_1} \right) / \left( \frac{x_1 \times z_0}{x_0 \times z_0} \right) \to \left( \frac{y_1 \times z_1}{y_0 \times z_1} \right) / \left( \frac{y_1 \times z_0}{y_0 \times z_0} \right) \]
is an equivalence, as desired.

\item For a map $f: U \to V$ in $\cT$,
\[ f_{\otimes}: \Fun(\Delta^1, \cC_U) \to \Fun(\Delta^1,\cC_V) \]
sends $L_U$-equivalences to $L_V$-equivalences: to show this, suppose
\[ \theta: [x_0 \rightarrow x_1] \to [y_0 \rightarrow y_1] \]
is a $L_U$-equivalence in $\Fun(\Delta^1,\cC_U)$. Equivalently, we have a left Kan extension
\[ \begin{tikzcd}[row sep=4ex, column sep=4ex, text height=1.5ex, text depth=0.25ex]
\Lambda^2_0 \times \Delta^1 \ar{r} \ar{d} & \cC_U \\
(\Lambda^2_0)^\rhd \times \Delta^1 \ar{ru}
\end{tikzcd} \]
where restriction to the cone point is sent to an equivalence $x_1/x_0 \xto{\sim} y_1/y_0$. Using distributivity, we get a $\cT^{/V}$-left Kan extension diagram
\[ \begin{tikzcd}
\prod_{f} ((\Lambda^2_0 \times \Delta^1) \times (\cT^{/U})^\op) \ar{r} \ar{d} & \prod_{f} \cC_{\underline{U}} \ar{r}{\prod_{f}} & \cC_{\underline{V}} \\
((\Lambda^2_0)^\rhd \times \Delta^1) \times (\cT^{/V})^\op \ar{rru}
\end{tikzcd} \]
The vertical arrow factors as
\[ \prod_{f} ((\Lambda^2_0 \times \Delta^1) \times (\cT^{/U})^\op) \to (\Lambda^2_0 \times \Delta^1) \times (\cT^{/V})^\op \to ((\Lambda^2_0)^\rhd \times \Delta^1) \times (\cT^{/V})^\op \]
where the first arrow is induced by the symmetric monoidal structure on $\Delta^1 \times \cT^\op$. Therefore, the left Kan extension corresponding to $f_{\otimes} (\theta)$
\[ \begin{tikzcd}[row sep=4ex, column sep=4ex, text height=1.5ex, text depth=0.25ex]
\Lambda^2_0 \times \Delta^1 \ar{r} \ar{d} & \cC_V \\
(\Lambda^2_0)^\rhd \times \Delta^1 \ar{ru}
\end{tikzcd} \]
restricts on the cone point to an equivalence, so $f_{\otimes} (\theta)$ is a $L_V$-equivalence.
\end{enumerate}

In particular, suppose that $\cT = \OO_G$ and $\cC = \underline{\Spc}_G$. Then we obtain the smash product $G$-symmetric monoidal structure on pointed $G$-spaces $\underline{\Spc}_{G,\ast}$, and given a map of orbits $f: G/H \to G/K$ corresponding to an inclusion of subgroups $K \subset H$ and a real $K$-representation $R$, the norm functor $f_{\otimes}$ sends the representation sphere $S^R$ to the representation sphere $S^{\Ind^H_K R}$. 
\end{exm}

\subsection{Pointwise \texorpdfstring{$\cO$}{O}-monoidal structure}

In this brief subsection, we indicate how to adapt the construct of the $\cT$-Day convolution so as to construct the cotensor of $\Op_{\cO,\cT}$ over $\Cat_{\cT}$. In other words, given a fibration of $\cT$-$\infty$-operads $p: \cD^{\otimes} \to \cO^{\otimes}$  and a $\cT$-$\infty$-category $\cK$, we have a $\cT$-$\infty$-operad $\widetilde{\Fun}_{\cO,\cT}(\cK \times_{\cT^{\op}} \cO, \cD)^{\otimes}$ that satisfies the universal mapping property
\begin{equation} \label{eq:UMP_cotensor}
 \Alg_{\cO,\cT}( \cC, \widetilde{\Fun}_{\cO,\cT}(\cK \times_{\cT^{\op}} \cO, \cD)) \simeq \Fun_{\cT}(\cK, \underline{\Alg}_{\cO, \cT}(\cC,\cD))
 \end{equation} 
for all fibrations of $\infty$-operads $q:\cC^{\otimes} \to \cO^{\otimes}$.

For the following construction, observe the isomorphism
\[ \Ar^{\inert}(\cO^\otimes) \times_{\ev_1,\cO^{\otimes}, \pr} (\cO^{\otimes} \times_{\cT^{\op}} \cK) \cong \Ar^{\inert}(\cO^\otimes) \times_{\cT^{\op}} \cK. \]

\begin{conthm} \label{conthm:cotensor_operads}
Let $\cO^{\otimes}$ be a $\cT$-$\infty$-operad and $\cK$ a $\cT$-$\infty$-category. Consider the span diagram of marked simplicial sets
\[ \begin{tikzcd}[row sep=2em, column sep=2em]
(\cO^{\otimes},\Ne) & (\Ar^{\inert}(\cO^\otimes) \times_{\cT^{\op}} \cK, \Ne) \ar{l}[swap]{\ev_0} \ar{r}{\ev_1} & (\cO^{\otimes},\Ne).
\end{tikzcd} \]
where by the middle marking $\Ne$ we mean those edges whose source and target in $\cO^{\otimes}$ are inert and whose projection to $\cK$ is a cocartesian edge. Then the functor
\[ (\ev_0)_* \circ (\ev_1)^*: \sSet^+_{/(\cO^{\otimes}, \Ne)} \to \sSet^+_{/(\cO^{\otimes}, \Ne)} \]
is right Quillen with respect to the $\cT$-operadic model structures. For a fibration of $\infty$-operads $p: \cC^{\otimes} \to \cO^{\otimes}$, we let
\[ \widetilde{\Fun}_{\cO,\cT}(\cK \times_{\cT^{\op}} \cO, \cC)^{\otimes} \coloneq (\ev_0)_* (\ev_1)^* (\cC^{\otimes},\Ne). \]
If $\cO \simeq \cT^{\op}$, then we more simply write
\[ \underline{\Fun}_{\cT}(\cK,\cC)^{\otimes} \coloneq (\ev_0)_* (\ev_1)^* (\cC^{\otimes},\Ne). \]
This construction satisfies the universal property \eqref{eq:UMP_cotensor} and its underlying $\cT$-$\infty$-category is as the notation indicates.
\end{conthm}
\begin{proof}
This follows along the same lines as our proofs of the analogous results for $\cT$-Day convolution in \cref{subsec:DayConvolution}.
\end{proof}

\begin{rem}
In the situation of \cref{conthm:cotensor_operads}, a fibrant replacement of $(\cO^{\otimes}, \Ne) \times_{\cT^{\op}} \leftnat{\cK}$ in the $\cT$-operadic model structure on $\sSet^+_{/(\cO^{\otimes}, \Ne)}$ computes the \emph{tensor} of $\Op_{\cO,\cT}$ over $\CatT$.
\end{rem}

If we then suppose that $\cC^{\otimes}$ is an $\cO$-monoidal $\cT$-$\infty$-category, we obtain the \emph{pointwise} $\cO$-monoidal structure on $\underline{\Fun}_{\cT}(\cK,\cC)$. In contrast to the $\cT$-Day convolution, we don't need to impose any further hypotheses on $\cC^{\otimes}$ for this to exist.

\begin{conthm}
Let $\cO^{\otimes}$ be a $\cT$-$\infty$-operad and $\cK$ a $\cT$-$\infty$-category. Consider the span diagram of marked simplicial sets
\[ \begin{tikzcd}[row sep=2em, column sep=2em]
(\cO^{\otimes})^{\sharp} & \Ar^{\inert}(\cO^\otimes)^{\sharp} \times_{\cT^{\op}} \leftnat{\cK} \ar{l}[swap]{\ev_0} \ar{r}{\ev_1} & (\cO^{\otimes})^{\sharp}.
\end{tikzcd} \]
Then the functor
\[ (\ev_0)_* \circ (\ev_1)^*: \sSet^+_{/\cO^{\otimes}} \to \sSet^+_{/\cO^{\otimes}} \]
agrees with the construction of \cref{conthm:cotensor_operads} on underlying simplicial sets and is right Quillen with respect to the $\cT$-monoidal model structures. Given any $\cO$-monoidal $\cT$-$\infty$-categories $\cC^{\otimes}, \cD^{\otimes}$, we then have a natural equivalence
\[ \Fun^{\otimes}_{\cO,\cT}(\cC, \widetilde{\Fun}_{\cO,\cT}(\cK \times_{\cT^{\op}} \cO, \cD)) \simeq \Fun_{\cT}(\cK, \underline{\Fun}^{\otimes}_{\cO, \cT}(\cC, \cD)). \]
\end{conthm}
\begin{proof}
That $(\ev_0)_* \circ (\ev_1)^*$ is right Quillen follows by \cite[Thm.~B.4.2]{HA} as in the proof of \cref{thm:OperadicCoinduction}. The only differences to note are regarding conditions (4) and (7). For (4), by \cite[Prop.~3.5(1)]{Exp2b} the $\ev_0$-cartesian edges in $\Ar^{\inert}(\cO^{\otimes})$ are fiberwise active in the target, hence $\ev_0: \Ar^{\inert}(\cO^{\otimes}) \times_{\cT^{\op}} \cK \to \cO^{\otimes}$ is a cartesian fibration whose cartesian edges project to equivalences in $\cK$. (7) then follows by inspection.

The universal mapping property then follows by restricting the equivalence \eqref{eq:UMP_cotensor}.
\end{proof}

\begin{exm}
Suppose $\cO^{\otimes} = \uFinpT$, let $\cC^{\otimes}$ be a $\cT$-symmetric monoidal $\cT$-$\infty$-category, and let $\cK$ be a $\cT$-$\infty$-category. Then we may unwind the pointwise $\cT$-symmetric monoidal structure on $\underline{\Fun}_{\cT}(\cK, \cC)$ as follows:

\begin{itemize}
\item[($\ast$)] Let $f: U \to V$ be a map of finite $\cT$-sets. Then the norm functor
\[ f_{\otimes}: \Fun_{\cT^{/U}}(\cK_{\underline{U}}, \cC_{\underline{U}}) \to \Fun_{\cT^{/V}}(\cK_{\underline{V}}, \cC_{\underline{V}}) \]
sends a $\cT^{/U}$-functor $F: \cK_{\underline{U}} \to \cC_{\underline{U}}$ to the $\cT^{/V}$-functor
\[ \cK_{\underline{V}} \xto{\eta} \prod_f \cK_{\underline{U}} \simeq \Fun_{\cT^{/V}}(\underline{U}, \cK_{\underline{V}}) \xtolong{\prod_f F}{1.3} \prod_f \cC_{\underline{U}} \simeq \Fun_{\cT^{/V}}(\underline{U}, \cC_{\underline{V}}) \xto{f_{\otimes}} \cC_{\underline{V}} \]
where $\eta$ is the diagonal $\cT^{/V}$-functor (i.e., the unit of the restriction-coinduction adjunction) and the norm $\cT^{/V}$-functor $f_{\otimes}$ is defined as in \cref{exm:normParametrizedFunctor}.
\end{itemize}

In particular, note that if $\cC^{\otimes}$ is distributive, then $\underline{\Fun}_{\cT}(\cK, \cC)^{\otimes}$ is also distributive.
\end{exm}

\section{Parametrized operadic left Kan extensions}

In this section, we construct $\cT$-operadic left Kan extensions, implementing in the operadic context the strategy that the second author used to construct $\cT$-left Kan extensions in \cite[§§9-10]{Exp2}.

\begin{rem} The strategy of our construction of $\cT$-operadic left Kan extensions will be to first show that the $\cT$-colimit of a lax $\cO$-monoidal functor canonically inherits a $\cO$-algebra structure, and to then reduce to this case via the $\cO$-monoidal envelope. If we let $\cT = \Delta^0$, this gives a new construction of Lurie's operadic left Kan extension (\cite[\S 3.1.2]{HA}). See also \cite{chu2021free}.
\end{rem}

We first begin with some necessary preliminaries on generalized $\cT$-$\infty$-operads and the $\cT$-operadic join.

\subsection{Generalized \texorpdfstring{$\cT$-$\infty$}{T-infinity}-operads}

Given a map of finite $\cT$-sets $\phi: U \to V$, let $\sigma_1$, $\sigma_2$ be two squares in $\FinT$
\[ \begin{tikzcd}[row sep=2em, column sep=2em]
U_i \ar{r}{\alpha_i} \ar{d}{\phi_i} & U \ar{d}{\phi} \\
V_i \ar{r}{\beta_i} & V
\end{tikzcd} \]
such that $\alpha_i$ is a summand inclusion, the induced map $U_i \to V_i \times_V U$ is also a summand inclusion, and moreover $(\alpha_1, \alpha_2): U_1 \sqcup U_2 \to U$ is an epimorphism. Let $U_{12} = U_1 \times_U U_2$ and $V_{12} = V_1 \times_V V_2$. Then we have an induced map
\[ g_{\phi,\sigma_1,\sigma_2}: (\Lambda^2_2)^{\lhd} \to \BigFinT \]
which selects the square
\[ \begin{tikzcd}[row sep=2em, column sep=2em]
\left[ U_+ \rightarrow V \right] \ar{r} \ar{d} & \left[ (U_1)_+ \rightarrow V_1 \right] \ar{d} \\
\left[ (U_2)_+ \rightarrow V_2 \right] \ar{r} & \left[ (U_{12})_+ \rightarrow V_{12} \right] 
\end{tikzcd} \]
in which every morphism is inert. 

We define the \emph{generalized $\cT$-operadic model structure} on $\sSet^+_{/(\BigFinT,\Ne)}$ to be the model structure defined using the categorical pattern $(\Ne, \text{All}, \{g_{\phi,\sigma_1,\sigma_2}\} )$ on $\BigFinT$, letting the $\phi$ and $\{ \sigma_1, \sigma_2\}$ range over all possible choices. We call the fibrant objects for this model structure \emph{generalized $\cT$-$\infty$-operads}. Note that $\cT$-$\infty$-operads are fibrant in this model structure, so are examples of generalized $\cT$-$\infty$-operads. However, the converse is not true. Indeed, let $\sigma_0: \FinT^\op \to \BigFinT$ be the cocartesian section which selects $[\emptyset_+ \rightarrow V]$ in each fiber and define $\wt{\cC}_0 \coloneq \FinT^\op \times_{\sigma_0, \BigFinT} \wt{\cC}^\otimes$. Then if $\wt{\cC}^\otimes \to \BigFinT$ is a generalized $\cT$-$\infty$-operad, $\wt{\cC}_0$ is not necessarily the terminal $\cT$-$\infty$-category.

We then say that $\cC^\otimes \to \uFinpT$ is a \emph{generalized $\cT$-$\infty$-operad} if it is the pullback of a generalized $\cT$-$\infty$-operad $\wt{\cC}^\otimes \to \BigFinT$ under the inclusion $\uFinpT \to \BigFinT$. Let $\cC_0$ be the corresponding pullback of $\wt{\cC}_0$.  The $\cT$-functor $\cT^\op \times \Delta^1 \to \uFinpT$ which selects the inert edge $[V_+ \rightarrow V] \to [\emptyset_+ \rightarrow V]$ in each fiber induces a $\cT$-functor $\cC \to \cC_0$.


\begin{lem} \label{lem:FiberwiseDescriptionForGeneralizedOperads} Suppose $\cC^\otimes \to \uFinpT$ is a generalized $\cT$-$\infty$-operad. Let $f: U \to V$ be a morphism in $\FinT$ and suppose that we have an orbit decomposition $U \simeq \sqcup_{1 \leq i \leq n} U_i$. Abbreviate $U_i \times_V U_j$ as $U_{ij}$ and let $f_i: U_i \to V$, $f_{ij}: U_{ij} \to V$ denote the induced maps. Then we have an equivalence of $\cT^{/V}$-$\infty$-categories
\[ \cC^\otimes_{\underline{\left[ U_+ \rightarrow V \right]}} \to \prod_{f_1} \cC_{\underline{U_1}} \: \underset{\underset{f_{12}}{\prod} (\cC_0)_{\underline{U_{12}}}}{\times} \: \prod_{f_2} \cC_{\underline{U_2}} \: \underset{ \underset{f_{23}}{\prod} (\cC_0)_{\underline{U_{23}}} }{\times} \: ... \: \underset{ \underset{f_{(n-1)n}}{\prod} (\cC_0)_{\underline{U_{(n-1)n}}} }{\times} \: \prod_{f_n} \cC_{\underline{U_n}} \]
\end{lem}
\begin{proof} Note that $U_{ij}$ need not be an orbit, so the notation $(\cC_0)_{\underline{U_{ij}}}$ means $\coprod_{\cO \subset U_{ij}} (\cC_0)_{\underline{O}} \to \underline{U_{ij}}$ for any orbit decomposition of $U_{ij}$. Without loss of generality we may suppose $n=2$. We have a pullback square in $\FinT$
\[ \begin{tikzcd}[row sep=2em, column sep=2em]
U_{12} \ar{r}{g_2} \ar{d}{g_1} & U_1 \ar{d}{f_1} \\
U_2 \ar{r}{f_2} & V.
\end{tikzcd} \]
The unit maps for the restriction-coinduction adjunction yield $\cT^{/V}$-functors
\[ \begin{tikzcd}[row sep=2em, column sep=3em]
\prod_{f_1} \cC_{\underline{U_1}} \ar{r} & \prod_{f_{12}} \cC_{\underline{U_{12}}} & \prod_{f_1} \cC_{\underline{U_1}} \ar{l}
\end{tikzcd} \]
and postcomposing with
\[ \prod_{f_{12}} \cC_{\underline{U_{12}}} \to \prod_{f_{12}} (\cC_0)_{\underline{U_{12}}} \]
we obtain the $\cT^{/V}$-functors which define the pullback. Using the `fiberwise' definition of generalized $\cT$-$\infty$-operad, we can use the same argument as in the proof of \cref{prp:SegalCondition} to accomplish the proof.
\end{proof}

\subsection{\texorpdfstring{$\cT$}{T}-operadic join}

\begin{dfn} Let $\wt{\cC}^\otimes, \wt{\cD}^\otimes, \wt{\cO}^\otimes$ be generalized $\cT$-$\infty$-operads over $\BigFinT$ and let $p,q: \wt{\cC}^\otimes, \wt{\cD}^\otimes \to \wt{\cO}^\otimes$ be categorical fibrations preserving the inert edges. Then the \emph{$\cT$-operadic join} of $p$ and $q$ is defined to be
\[ \widetilde{(\cC \star_{\cO} \cD)}^\otimes \coloneq \wt{\cC}^\otimes \star_{\wt{\cO}^\otimes} \wt{\cD}^\otimes \to \wt{\cO}^\otimes. \]

Similarly, given $\cC^\otimes, \cD^\otimes \to \cO^\otimes$ fibrations of generalized $\cT$-$\infty$-operads over $\uFinpT$, we define the $\cT$-operadic join $(\cC \star_\cO \cD)^\otimes$ to be $\cC^\otimes \star_{\cO^\otimes} \cD^\otimes \to \cO^\otimes$.
\end{dfn}

Note that the underlying $\cT$-$\infty$-category of a $\cT$-operadic join $(\cC \star_{\cO} \cD)^\otimes$ is the $\cO$-join of the underlying $\cT$-$\infty$-categories of the factors (by the base-change property \cite[Lem.~4.4]{Exp2}).

\begin{prp} $\widetilde{(\cC \star_{\cO} \cD)}^\otimes$ is a generalized $\cT$-$\infty$-operad and the structure map to $\wt{\cO}^\otimes$ is a fibration of generalized $\cT$-$\infty$-operads.
\end{prp}
\begin{proof} By the proof of \cite[Prop.~4.7(2)]{Exp2}, $\pi: \widetilde{(\cC \star_{\cO} \cD)}^\otimes \to \BigFinT$ admits cocartesian lifts over the inert edges. Moreover, an edge $\Delta^1 \to \widetilde{(\cC \star_{\cO} \cD)}^\otimes$ is $\pi$-cocartesian if and only if it factors through either $\wt{\cC}^\otimes$ or $\wt{\cD}^\otimes$ and is inert there.

The fiber of $\widetilde{(\cC \star_{\cO} \cD)}^\otimes$ over an object $[U_+ \rightarrow V]$ of $\BigFinT$ is given by $\wt{\cC}^\otimes_{[U_+ \rightarrow V]} \star_{\wt{\cO}^\otimes_{[U_+ \rightarrow V]}} \wt{\cD}^\otimes_{[U_+ \rightarrow V]}$. Moreover, the relative join is functorial in the following sense: given commutative diagrams
\[ \begin{tikzcd}[row sep=2em, column sep=2em]
X \ar{r} \ar{d} & X' \ar{d} \\
B \ar{r} & B'
\end{tikzcd}, \quad
 \begin{tikzcd}[row sep=2em, column sep=2em]
Y \ar{r} \ar{d} & Y' \ar{d} \\
B \ar{r} & B'
\end{tikzcd} \]
we have an induced map $X \star_B Y \to X' \star_{B'} Y'$ covering $B \times \Delta^1 \to B' \times \Delta^1$. From our explicit description of the $\pi$-cocartesian edges, we see that the Segal maps for $\widetilde{(\cC \star_{\cO} \cD)}^\otimes$ are obtained in this way. Consequently, it is clear that they are equivalences.

It remains to check that for all of the defining maps $g_{\phi, \sigma_1, \sigma_2}: (\Lambda^2_2)^\lhd \to \BigFinT$, we have that for every lift $G: (\Lambda^2_2)^\lhd \to \widetilde{(\cC \star_{\cO} \cD)}^\otimes$ where all edges are sent to $\pi$-cocartesian edges, $G$ is a $\pi$-limit diagram. For this, there are two cases to consider. Either $G$ factors through $\wt{\cC}^\otimes$, in which case the assertion obviously follows from $G$ being a $\pi|_{\wt{\cC}^\otimes}$-limit diagram, or $G$ factors through $\wt{\cD}^\otimes$, in which case the assertion is a consequence of \cref{lem:JoinCommutingWithSlice}.

Finally, the second assertion is obvious given the first.
\end{proof}

\begin{lem} \label{lem:JoinCommutingWithSlice} Let $K$, $X$, $Y$ and $B$ be $\infty$-categories, let $f_1, f_2: X,Y \to B$ be categorical fibrations and let $p: K \to Y$ be a functor. Also let $p$ denote the composition $K \xrightarrow{p} Y \subset X \star_B Y$. Then we have an equivalence of $\infty$-categories over $B^{/f_2 p} \times \Delta^1$
\[ (X \star_B Y)^{/p} \simeq (X \times_{B} B^{/f_2 p}) \star_{B^{/f_2 p}} Y^{/p}. \]
Consequently, if $\overline{p}: K^\lhd \to Y$ is a $f_2$-limit diagram, then $\overline{p}: K^\lhd \to Y \subset X \star_B Y$ is a $f$-limit diagram (where $f$ denotes the structure map $X \star_B Y \to B$).
\end{lem}
\begin{proof} Let $q$ denote $K \to X \star_B Y \to B \times \Delta^1$ and note that $q$ factors as $K \xrightarrow{f_2 p} B \times \{1\} \subset B \times \Delta^1$. We first place $(X \star_B Y)^{/p}$ into the diagram of pullback squares
\[ \begin{tikzcd}[row sep=2em, column sep=2em]
(X \star_B Y)^{/p} \ar{r} \ar{d} & Z \ar{r} \ar{d} & \Fun(K^{\lhd}, X \star_B Y) \ar{d} \\
\{p \} \ar{r} & \Fun_{/B}(K,Y) \ar{r} \ar{d} & \Fun(K, X \star_B Y) \ar{d} \\
& \{ q \} \ar{r} & \Fun(K, B \times \Delta^1).
\end{tikzcd} \]
In addition, using that the composition $\Fun(K^\lhd, X \star_B Y) \to \Fun(K, X \star_B Y) \to \Fun(K, B \times \Delta^1)$ agrees with $\Fun(K^\lhd, X \star_B Y) \to \Fun(K^\lhd, B \times \Delta^1) \to \Fun(K, B \times \Delta^1)$, $Z$ fits into the diagram of pullback squares
\[ \begin{tikzcd}[row sep=2em, column sep=2em]
Z \ar{r} \ar{d} & \Fun(K^\lhd, X \star_B Y) \ar{d} \\
(B \times \Delta^1)^{/q} \ar{r} \ar{d} & \Fun(K^\lhd, B \times \Delta^1) \cong \Fun(K^\lhd, B) \times \Fun(K^\lhd, \Delta^1) \ar{d} \\
\{ q \} = \{ f_2 p\} \times \{ \text{const}_1 \} \ar{r} & \Fun(K, B \times \Delta^1) \cong \Fun(K,B) \times \Fun(K,\Delta^1).
\end{tikzcd} \]
Because $(\Delta^1)^{/\text{const}_1} \simeq \Delta^1$, we get that $(B \times \Delta^1)^{/q} \simeq B^{/f_2 p} \times \Delta^1$. Consequently,
\[ (X \star_B Y)^{/p} \simeq \{p \} \times_{\Fun_{/B}(K,Y)} \left( \Fun(K^\lhd, X \star_B Y) \times_{\Fun(K^\lhd, B \times \Delta^1)} B^{/f_2 p} \times \Delta^1 \right). \]

Let $A \to B^{/f_2 p} \times \Delta^1$ be any functor. We will identify $\Fun_{/(B^{/f_2 p} \times \Delta^1)} (A, (X \star_B Y)^{/p})$ with $\Fun_{/B}(A_0,X) \times \Fun_{/(B^{/f_2 p})}(A_1,Y^{/p})$ and thereby prove the claim. Let
\[ A \times K^{\lhd} \to B^{/f_2 p} \times \Delta^1 \times K^\lhd \to B \times \Delta^1 \]
be the composite of the given map and the map adjoint to $B^{/f_2 p} \times \Delta^1 \to \Fun(K^\lhd, B \times \Delta^1)$. Note that $(A \times K^\lhd)_0 \cong A_0$ and $(A \times K^\lhd)_1 \cong (A \times K) \cup_{A_1 \times K} A_1 \times K^\lhd$.
We then have the chain of equivalences
\begin{align*} \Fun&_{/(B^{/f_2 p} \times \Delta^1)} (A, (X \star_B Y)^{/p}) \simeq \{ p \circ \pr_K \} \times_{\Fun(A \times K, X \star_B Y)} \Fun_{/(B \times \Delta^1)}(A \times K^\lhd, X \star_B Y) \\
&\simeq \{ p \circ \pr_K \} \times_{\Fun(A \times K, X \star_B Y)} (\Fun_{/B}(A_0,X) \times \Fun_{/B}((A \times K) \cup_{A_1 \times K} A_1 \times K^\lhd, Y)) \\
&\simeq \Fun_{/B}(A_0,X) \times \left( \{ p \circ \pr_K \} \underset{\Fun_{/B}(A \times K, Y)}{\times} \Fun_{/B}(A \times K, Y) \underset{\Fun_{/B}(A_1 \times K,Y)}{\times} \Fun_{/B}(A_1 \times K^\lhd, Y) \right) \\
&\simeq \Fun_{/B}(A_0,X) \times \left( \{ p \circ \pr_K \} \underset{\Fun_{/B}(A_1 \times K,Y)}{\times} \Fun_{/B}(A_1 \times K^\lhd, Y) \right)
\end{align*}
and finally
\[ \{ p \circ \pr_K \} \underset{\Fun_{/B}(A_1 \times K,Y)}{\times} \Fun_{/B}(A_1 \times K^\lhd, Y) \simeq \Fun_{/(B^{/f_2 p})}(A_1,Y^{/p}) \]
because both sides compute the total fiber of the punctured cube
\[\begin{tikzcd}[row sep={40,between origins}, column sep={60,between origins}]
      &[-\perspective]  &[\perspective] &[-\perspective]  \Fun(A_1 \times K^\lhd, Y) \vphantom{\times_{U}} \ar{dd}\ar{dl} \\[-\perspective]
     \Delta^0 \ar[crossing over]{rr} \ar{dd} & & \Fun(A_1 \times K, Y) \\[\perspective]
      & \Delta^0  \ar{rr} \ar{dl} & &  \Fun(A_1 \times K^\lhd, B) \vphantom{\times_{V}} \ar{dl} \\[-\perspective]
    \Delta^0 \ar{rr} && \Fun(A_1 \times K, B) \ar[from=uu,crossing over]
\end{tikzcd}\]

For the last assertion, we need to show that
\[ \begin{tikzcd}[row sep=2em, column sep=2em]
(X \star_B Y)^{/\overline{p}} \ar{r} \ar{d} & (X \star_B Y)^{/p} \ar{d} \\
B^{/ f \overline{p}} \ar{r} & B^{/f p}
\end{tikzcd} \]
is a homotopy pullback square. But by the first part of the lemma this is equivalent to
\[ \begin{tikzcd}[row sep=2em, column sep=2em]
(X \times_{B} B^{/ f \overline{p}}) \star_{B^{/ f \overline{p}}} Y^{/\overline{p}} \ar{r} \ar{d} & (X \times_{B} B^{/ f p}) \star_{B^{/f p}} Y^{/p} \ar{d} \\
B^{/ f \overline{p}} \ar{r} & B^{/f p}.
\end{tikzcd} \]
One easily deduces that this is a homotopy pullback square as a consequence of the two squares
\[ \begin{tikzcd}[row sep=2em, column sep=2em]
X \times_{B} B^{/ f \overline{p}} \ar{r} \ar{d} & X \times_{B} B^{/ f p} \ar{d} \\
B^{/ f \overline{p}} \ar{r} & B^{/f p}
\end{tikzcd}, \quad
 \begin{tikzcd}[row sep=2em, column sep=2em]
Y^{/\overline{p}} \ar{r} \ar{d} & Y^{/p} \ar{d} \\
B^{/ f \overline{p}} \ar{r} & B^{/f p}
\end{tikzcd} \]
being homotopy pullback squares (the second by our hypothesis). 
\end{proof}

\subsection{Construction of \texorpdfstring{$\cT$}{T}-operadic left Kan extensions}

We now turn towards constructing $\cT$-operadic left Kan extensions. We will need a variant of the $\cT$-Day convolution for our proofs, where we allow the source $\cT$-$\infty$-operad to be generalized.

\begin{vrn} Suppose that $\cC^\otimes$ is a generalized $\cT$-$\infty$-operad. Then the proofs of \cref{thm:OperadicCoinduction} and \cref{cor:OperadicCoinduction2} still go through to show that
$$\widetilde{\Fun}_{\cO,\cT}(\cC,\cE)^\otimes \to \cO^\otimes$$
is a (non-generalized) $\cT$-$\infty$-operad. Moreover, if $\cC^\otimes, \cE^\otimes \to \cO^\otimes$ are cocartesian fibrations, then the proof of \cref{prp:DayConvolutionLocallyCocartesian} goes through to show that $\widetilde{\Fun}_{\cO,\cT}(\cC,\cE)^\otimes \to \cO^\otimes$ is a locally cocartesian fibration. However, the proof of \cref{thm:DayConvolutionCocartesian} doesn't directly apply because if $\cC^\otimes$ is generalized, we have a different formula for $\cC^\otimes_{\underline{x}}$ involving fiber products instead of products (cf. \cref{lem:FiberwiseDescriptionForGeneralizedOperads}).
\end{vrn}

For the following results, let $\cO^\otimes$ be a $\cT$-$\infty$-operad, $p: \cC^\otimes \to \cO^\otimes$ an $\cO$-monoidal $\cT$-$\infty$-category and $q: \cE^\otimes \to \cO^\otimes$ a distributive $\cO$-monoidal $\cT$-$\infty$-category.

\begin{lem} \label{lem:JoinDayConvolutionCocartesian} Consider the commutative diagram
\[ \begin{tikzcd}[row sep=2em, column sep=2em]
\widetilde{\Fun}_{\cO,\cT}(\cC \star_{\cO} \cO, \cE)^\otimes \ar{r}{\rho} \ar{d}{\lambda} & \cE^\otimes \ar{d}{q} \\
\widetilde{\Fun}_{\cO,\cT}(\cC, \cE)^\otimes \ar{r}{\pi} & \cO^\otimes.
\end{tikzcd} \]
where $\lambda$ is given by restriction along $\cC^\otimes \subset (\cC \star_{\cO} \cO)^\otimes$, $\rho$ is given by restriction along $\cO^\otimes \subset (\cC \star_{\cO} \cO)^\otimes$ followed by the equivalence $\widetilde{\Fun}_{\cO,\cT}(\cO, \cE)^\otimes \xto{\simeq} \cE^\otimes$ induced by precomposition with the identity section $\iota: \cO^\otimes \to \Ar^{\inert}(\cO^\otimes)$ (cf. \cref{lem:IdentifyingLeftAdjointOfCoinduction}), and $\pi$ is the structure map.
\begin{enumerate} \item An edge $e$ in $\widetilde{\Fun}_{\cO,\cT}(\cC \star_{\cO} \cO, \cE)^\otimes$ is locally $\pi \lambda$-cocartesian if and only if $\lambda(e)$ is locally $\pi$-cocartesian and $\rho(e)$ is locally $q$-cocartesian. Consequently, $\pi \lambda$ is cocartesian.
\item $\lambda$ is $\cO^\otimes$-cartesian and $\rho$ is cocartesian.
\end{enumerate}
\end{lem}
\begin{proof} (1): $\rho$ and $\lambda$ are fibrations of $\cT$-$\infty$-operads by the functoriality of the Day convolution, hence preserve inert edges. By \cref{lem:LocallyCocartesianCriterionForOperads}, it suffices to consider a locally $\pi \lambda$-cocartesian edge $e$ over a fiberwise active edge $f: U \to V \in \FinT$ (identified with $[U_+ \rightarrow V] \to [V_+ \rightarrow V]$ in $\uFinpT$). Suppose $e$ covers $\alpha: x \to y$ in $\cO^\otimes_V$. Then by \cref{prp:DayConvolutionLocallyCocartesian}(3), $e$ corresponds to a $\cT^{/V}$-left Kan extension
\[ \begin{tikzcd}[row sep=2em, column sep=2em]
(\cC^\otimes_{\underline{x}})^{\underline{\rhd}} \ar{r}{F}[name=F,below,near start]{} \ar{d}[swap]{(\alpha_{\otimes})^{\underline{\rhd}}} & \cE^\otimes_{\underline{x}} \ar{r}{\alpha_{\otimes}} & \cE_{\underline{y}} \ar{d} \\
(\cC_{\underline{y}})^{\underline{\rhd}} \ar{rru}[swap]{G}[name=G,near start]{} \ar{rr} & & (\cT^{/V})^\op.
\arrow[Rightarrow,to path={(F) -- (G)}]
\end{tikzcd} \]
Examining the pointwise formula defining $\cT^{/V}$-left Kan extensions, we see that $G$ is a $\cT^{/V}$-left Kan extension of $\alpha_{\otimes} \circ F$ along $\alpha_{\otimes} \star \id$ if and only if $G|_{\cC_{\underline{y}}}$ is a $\cT^{/V}$-left Kan extension of $\alpha_{\otimes} \circ F|_{\cC^\otimes_{\underline{x}}}$ along $\alpha_{\otimes}$ and $G|_{\underline{V}} \simeq \alpha_{\otimes} \circ F|_{\underline{V}}$ (for the latter, using that the inclusion of the right $\cT$-cone point is fiberwise cofinal and \cite[Thm.~6.7]{Exp2}). This implies the first part of the claim. For the consequence, we only need to check that the composition of locally cocartesian edges is again locally cocartesian, and this is clear using the claim and \cref{thm:DayConvolutionCocartesian}.

(2): Taking the fiber over an object $y \in \cO_V$, we get a bifibration
\[  \Fun_{\cT^{/V}}((\cC_{\underline{y}})^{\underline{\rhd}}, \cE_{\underline{y}}) \to \Fun_{\cT^{/V}}(\cC_{\underline{y}}, \cE_{\underline{y}}) \times \cE_{\underline{y}}. \]
Combining this observation with the product decomposition over a general object $x \in \cO^\otimes$ obtained by the $\cT$-Segal condition, we deduce that $\lambda$ is fiberwise cartesian and $\rho$ is fiberwise cocartesian (over $\cO^\otimes$). It remains to show that for $\lambda$ the cocartesian pushforward of fiberwise cartesian edges remain fiberwise cartesian, and the dual statement for $\rho$. This is obvious over inert edges, so it suffices to consider a fiberwise active edge $\alpha: x \to y$ in $\cO^\otimes_V$. Also, without loss of generality suppose $y \in \cO_V$. Let $\theta: F_0 \to F_1$ be an edge in $\widetilde{\Fun}_{\cO,\cT}(\cC, \cE)^\otimes_x$, which corresponds to a natural transformation
\[ \theta: \Delta^1 \times \cC^\otimes_{\underline{x}} \to \cE^\otimes_{\underline{x}}. \]
 A fiberwise cartesian edge in $\widetilde{\Fun}_{\cO,\cT}(\cC \star_{\cO} \cO, \cE)^\otimes_x$ over $\theta$ is given by
 \[ \theta': \Delta^1 \times (\cC^\otimes_{\underline{x}})^{\underline{\rhd}} \to \cE^\otimes_{\underline{x}} \]
 which restricts to $\theta$ and is a constant natural transformation when restricted to the right $\cT$-cone point. The cocartesian pushforward $\alpha_! \theta'$ is given by the $\cT^{/V}$-left Kan extension of $\alpha_{\otimes} \circ \theta'$ along $\id \times \alpha_{\otimes}$. Clearly, this is still a constant natural transformation when restricted to the right $\cT$-cone point, which proves that $\alpha_! \theta'$ is a fiberwise cartesian edge lifting $\alpha_! \theta$. A similar argument handles the fiberwise cocartesian edges.
\end{proof}

The next proposition is a very general and parametrized form of the following observation: the colimit of a lax symmetric monoidal functor canonically inherits the structure of a commutative algebra.

\begin{prp} \label{prp:localOperadicKanExtn} Let $F: \cC^\otimes \to \cE^\otimes$ be a lax $\cO$-monoidal $\cT$-functor and let $\sigma_F: \cO^\otimes \to \widetilde{\Fun}_{\cO,\cT}(\cC,\cE)^\otimes$ be the associated section (which is an $\cO$-algebra map). Then there exists a $\cO$-algebra lift of $\sigma_F$ to
\[ \sigma_{\overline{F}} : \cO^\otimes \to \widetilde{\Fun}_{\cO,\cT}(\cC \star_{\cO} \cO, \cE)^\otimes \]
such that the resulting $\cO$-algebra
\[ A = \overline{F}|_{\cO^\otimes}: \cO^\otimes \to \cE^\otimes \]
has underlying section $A|_{\cO}: \cO \to \cE$ computed as the $q|_{\cE}$-$\cT$-left Kan extension of $F|_{\cC}: \cC \to \cE$ along the inclusion $\cC \to \cC \star_{\cO} \cO$.
\end{prp}
\begin{proof} Factoring $\sigma_F$ through the $\cO$-monoidal envelope $\Ar^{\act}_{\cT}(\cO^\otimes)$ of $\id: \cO^\otimes \to \cO^\otimes$, we obtain a pullback square of $\cO$-monoidal $\cT$-$\infty$-categories and (strong) $\cO$-monoidal $\cT$-functors
\[ \begin{tikzcd}[row sep=2em, column sep=2em]
X \ar{r} \ar{d}{\lambda_F} & \widetilde{\Fun}_{\cO,\cT}(\cC \star_{\cO} \cO, \cE)^\otimes \ar{d}{\lambda} \\
\Ar^{\act}_{\cT}(\cO^\otimes) \ar{r}{\tau_F} & \widetilde{\Fun}_{\cO,\cT}(\cC, \cE)^\otimes.
\end{tikzcd} \]
We proceed to identify the fibers of $\lambda_F$. By definition, for any object $x \in \cO^\otimes$, $\sigma_F(x)$ is given by the functor $F_{\underline{x}}: \cC^\otimes_{\underline{x}} \to \cE^\otimes_{\underline{x}}$, and for any fiberwise active edge $\alpha: x \to y$, $\tau_F(\alpha)$ is given by the cocartesian pushforward $\alpha_! F_{\underline{x}}: \cC^\otimes_{\underline{y}} \to \cE^\otimes_{\underline{y}}$. If $\alpha$ decomposes via the $\cT$-$\infty$-operad axioms as $(\alpha_i: x_i \to y_i)_{1 \leq i \leq n}$ for $y_i \in \cO_{V_i}$ (induced by an orbit decomposition $V \simeq V_1 \sqcup ... \sqcup V_n$ if $y$ covers $[V_+ \to W]$ in $\uFinpT$), then $\alpha_! F_{\underline{x}}$ may be explicitly identified as the collection of $\cT^{/V_i}$-left Kan extensions $(\alpha_i)_! F_{\underline{x_i}}$ of $(\alpha_i)_{\otimes} \circ F_{\underline{x_i}}: \cC^\otimes_{\underline{x_i}} \to \cE_{\underline{y_i}}$ along $(\alpha_i)_{\otimes}: \cC^\otimes_{\underline{x_i}} \to \cC_{\underline{y_i}}$. Therefore, the fiber of $\lambda_F$ over $\{\alpha\}$ is given by
\[ \prod_{1 \leq i \leq n} \cE_{\underline{y_i}}^{((\alpha_i)_! F_{\underline{x_i}}, \cT^{/V_i})/} \: .\]

Because each $\cE_{\underline{y_i}}$ is $\cT^{/V_i}$-cocomplete by assumption, these fibers all have initial objects, which are moreover preserved by the cocartesian edges over $\cT^\op$ (i.e., by restriction). We claim that restricting $\lambda_F$ to the full $\cT$-subcategory $X_0$ on these initial objects yields a trivial Kan fibration $\lambda'_F: X_0 \to \Ar^{\act}_{\cT}(\cO^\otimes)$. By (2) of \cref{lem:JoinDayConvolutionCocartesian}, $\lambda$ is $\cO^\otimes$-cartesian, hence the pulled back map $\lambda_F$ is $\cO^\otimes$-cartesian. Therefore, it suffices to check that the cocartesian edges in $X$ over $\cO^\otimes$ preserve initial objects. This is obvious for cocartesian edges over inert edges in $\cO^\otimes$, so it suffices to consider the case of a fiberwise active edge $\beta: y \to z$ in $\cO^\otimes_W$ with $z \in \cO_W$. Then for our fiberwise active edge $\alpha: x \to y \in \cO^\otimes_W$ above, $\beta_!(\alpha) = \beta \circ \alpha: x \to z$ computes the cocartesian pushforward in $\Ar^{\act}_{\cT}(\cO^\otimes)$. As we have seen, an initial object in $X$ covering $\alpha$ corresponds to a collection of $\cT^{/V_i}$-colimits of $(\alpha_i)_! F_{\underline{x_i}}$
\[ \begin{tikzcd}[row sep=2em, column sep=3em]
\cC_{\underline{y_i}} \ar{r}{(\alpha_i)_! F_{\underline{x_i}}}[name=F,below,near start]{} \ar{d} & \cE_{\underline{y_i}} \ar{d} \\
(\cT^{/V_i})^\op  \ar{r}{=} \ar{ru}[name=G,above,near start]{} & (\cT^{/V_i})^\op
\arrow[Rightarrow,to path={(F) -- (G)}]
\end{tikzcd} \]
By our assumption that $\cE^\otimes \to \cO^\otimes$ is distributive, applying $\prod_\beta$ and postcomposing with $\otimes_{\beta}: \cE^\otimes_{\underline{y}} \to \cE_{\underline{z}}$ yields a $\cT^{/W}$-colimit diagram
\[ \begin{tikzcd}[row sep=2em, column sep=3em]
\cC^\otimes_{\underline{y}} \ar{r}[name=F,below,near start]{} \ar{d} & \cE_{\underline{z}} \ar{d} \\
(\cT^{/W})^\op \ar{ru}[name=G,above,near start]{} \ar{r}{=} & (\cT^{/W})^\op
\arrow[Rightarrow,to path={(F) -- (G)}]
\end{tikzcd} \]
Factoring $\cC^\otimes_{\underline{y}} \to (\cT^{/W})^\op$ through $\cC_{\underline{z}}$ and using the transitivity of $\cT^{/W}$-left Kan extensions, this further implies that the diagram
\[ \begin{tikzcd}[row sep=2em, column sep=3em]
\cC_{\underline{z}} \ar{r}{(\beta \circ \alpha)_! F_{\underline{x}}}[name=F,below,near start]{} \ar{d} & \cE_{\underline{z}} \ar{d} \\
(\cT^{/W})^\op \ar{r}{=} \ar{ru}[name=G,above,near start]{} & (\cT^{/W})^\op
\arrow[Rightarrow,to path={(F) -- (G)}]
\end{tikzcd} \]
is a $\cT^{/W}$-colimit diagram, where inspecting the definitions reveals that the top horizontal arrow may be identified with $(\beta \circ \alpha)_! F_{\underline{x}}$. This is an initial object covering $\beta \circ \alpha$, as desired.

Choosing a section of $\lambda'_F$ and postcomposing by the map $X_0 \to \widetilde{\Fun}_{\cO,\cT}(\cC \star_{\cO} \cO, \cE)^\otimes$, we obtain the desired extension $\sigma_{\overline{F}}$. Finally, the assertion about $A|_{\cO}$ is clear from the construction if we consider only those objects $x$, $\alpha = \id_x$, and edges $\beta$ entirely in $\cO$.
\end{proof}

We can then promote \cref{prp:localOperadicKanExtn} to a global existence result.

\begin{thm} \label{thm:globalOperadicKanExtn} We have $\cO^\otimes$-adjunctions
\[\begin{tikzcd}[row sep=2em, column sep=2em]
\widetilde{\Fun}_{\cO,\cT}(\cC,\cE)^\otimes \ar[shift left]{r} & \widetilde{\Fun}_{\cO,\cT}(\cC \star_{\cO} \cO,\cE)^\otimes \ar[shift left]{r}{\ev_{\cO}} \ar[shift left]{l}{\ev_{\cC}} & \cE^\otimes \ar[shift left]{l}{\delta}.
\end{tikzcd} \]
Consequently, on passing to $\infty$-categories of $\cO$-algebras, we obtain the adjunction
\[ \adjunct{p_!}{\Alg_{\cO,\cT}(\cC,\cE)}{\Alg_{\cO,\cT}(\cE)}{p^\ast} \]
where $p_!$ is computed as in \cref{prp:localOperadicKanExtn} and $p^\ast$ is restriction along $p$.
\end{thm}
\begin{proof} The only subtlety involves the first $\cO^\otimes$-adjunction. We may invoke \cite[7.3.2.12]{HA} because the second condition there is ensured by distributivity in $\cE^\otimes$, using the same argument as in the proof of \cref{prp:localOperadicKanExtn}. We can then extract the adjunction involving $\infty$-categories of $\cO$-algebra maps by pulling back along the structure map $\ev_1: \Ar^{\act}_{\cT}(\cO^\otimes) \to \cO^\otimes$ of the $\cO$-monoidal envelope and taking cocartesian sections.
\end{proof}


Now suppose that $p: \cC^\otimes \to \cO^\otimes$ is only a fibration of $\cT$-$\infty$-operads and consider the factorization of $p$ through its $\cO$-monoidal envelope $\Env_{\cO,\cT}(\cC)^\otimes$. In view of \cref{prp:UniversalPropertyMonoidalEnvelope} and \cref{thm:globalOperadicKanExtn}, we have the composite adjunction
\[\begin{tikzcd}[row sep=2em, column sep=2em]
p_!: \Alg_{\cO,\cT}(\cC,\cE) \ar[shift left]{r} & \Alg_{\cO,\cT}(\Env_{\cO,\cT}(\cC), \cE) \ar[shift left]{r} \ar[shift left]{l} & \Alg_{\cO,\cT}(\cE) : p^\ast \ar[shift left]{l}.
\end{tikzcd} \]

\begin{dfn} \label{dfn:operadic_LKE}
Given a $\cO$-algebra map $F: \cC^\otimes \to \cE^\otimes$, we define the \emph{$\cT$-operadic left Kan extension} of $F$ to be $p_! F$.
\end{dfn}

\begin{rem} Given a fibration of $\cT$-$\infty$-operads $\cO^\otimes \to \cP^\otimes$ and $\cE^\otimes \to \cP^\otimes$ a distributive $\cP$-monoidal $\cT$-$\infty$-category, we will also speak of the $\cT$-operadic left Kan extension of a $\cP$-algebra map $F: \cC^\otimes \to \cE^\otimes$ along $p$ in the obvious way (note that distributivity is stable under pullback). In other words, we also have an adjunction
\[ \adjunct{p_!}{\Alg_{\cP,\cT}(\cC,\cE)}{\Alg_{\cP,\cT}(\cO,\cE)}{p^\ast}. \]
\end{rem}

Note that the underlying $\cT$-functor $p_!(F)|_{\cO}: \cO \to \cE$ is computed by first extending $F$ to
$$i_! F: \Env_{\cO,\cT}(\cC)^\otimes \to \cE^\otimes$$
and then taking the $\cT$-left Kan extension of $i_! F|_{\Env_{\cO,\cT}(\cC)}$ along the structure map to $\cO$. 

\begin{exm} Suppose that $\cC^\otimes = \Triv^\otimes = (\uFinpT)_{\inert}$ and $\cO^\otimes = \uFinpT$. Then the $\cT$-symmetric monoidal envelope of $\Triv^\otimes$ is $(\Triv^\otimes)_{\act} = (\uFinpT)_{\cocart}$, the maximal sub-left fibration of $\uFinpT \to \cT^\op$ obtained by taking the wide subcategory spanned by the cocartesian edges.

In the case that $\cT = \OO_G$ is the orbit category of a finite group, we can identify this with something familiar. Namely, let $\Sigma_n$ be the symmetric group on $n$ letters, and let $\OO_{G \times \Sigma_n, \Gamma_n}$ be the full subcategory of the orbit category of $G \times \Sigma_n$ on the $\Sigma_n$-free transitive $G \times \Sigma_n$-sets. (Recall that every object in this subcategory is isomorphic to an orbit $G \times \Sigma_n / \Gamma_\phi$, where $\Gamma_\phi$ is the graph of a homomorphism $\phi: H \to \Sigma_n$ for some subgroup $H$ of $G$.) Define a functor
\[ -/\Sigma_n: \OO_{G \times \Sigma_n, \Gamma_n} \to \OO_G, \: \goesto{U}{U/\Sigma_n}. \]
This is left adjoint to restriction along the projection $G \times \Sigma_n \to G$ and sends $(G \times \Sigma_n)/ \Gamma_{\phi}$ to $G/H$. Then $(-/\Sigma_n)^\op$ is a left fibration and exhibits $\OO_{G \times \Sigma_n, \Gamma_n}^{\op}$ as a $G$-$\infty$-category. In fact, $\OO_{G \times \Sigma_n, \Gamma_n}^{\op}$ is a model for the $G$-space $B_G \Sigma_n$ which classifies $G$-equivariant principal $\Sigma_n$-bundles (\cite[Rem.~3.17]{quigleyshah_tate}).

Now define a $G$-functor $F: \OO_{G \times \Sigma_n, \Gamma_n}^{\op} \to (\underline{\FF}_G)_{\cocart}$ which sends an object $U$ to the morphism of $G$-sets $(U \times n)/ \Sigma_n \to U/\Sigma_n$ and a morphism $U \ot V$ to
\[ \begin{tikzcd}[row sep=4ex, column sep=4ex, text height=1.5ex, text depth=0.25ex]
(U \times n)/\Sigma_n \ar{d} & (V \times n)/\Sigma_n \ar{l} \ar{d} \ar{r}{=} & (U \times n)/\Sigma_n \ar{d} \\
U/ \Sigma_n & V/ \Sigma_n \ar{l} \ar{r}{=} & V/ \Sigma_n
\end{tikzcd} \]
where we note that the left square is a pullback square of $G$-sets. (Note that this suffices to define a functor since $\underline{\FF}_G$ is equivalent to a $1$-category.) It follows from $\Sigma_n$-freeness and an elementary argument that $F$ is fully faithful. Moreover, taking the disjoint union over all $n \geq 0$ and postcomposing with $(-)_+$, we obtain an equivalence of $G$-$\infty$-categories
\[ \coprod_{n \geq 0} \OO_{G \times \Sigma_n, \Gamma_n}^{\op} \xto{\sim} (\underline{\FF}_G)_{\cocart} \xto{\sim} (\underline{\FF}_{G,\ast})_{\cocart}. \]

Therefore, for a $G$-symmetric monoidal $\infty$-category $\cC^\otimes$, the free $G$-commutative algebra on an object $x: \OO_G^\op \to \cC$ is computed by the $G$-colimit of the induced functor
\[ \coprod_{n \geq 0} \OO_{G \times \Sigma_n, \Gamma_n}^{\op} \to \cC. \]
\end{exm}

\begin{wrn} In the proofs of \cref{lem:JoinDayConvolutionCocartesian}, \cref{prp:localOperadicKanExtn} and \cref{thm:globalOperadicKanExtn}, the results would fail if we replaced $(\cC \star_{\cO} \cO)^\otimes$ by $(\cO \star_{\cO} \cC)^\otimes$. This corresponds to there generally being no $\cT$-operadic \emph{right} Kan extension.
\end{wrn}

\section{\texorpdfstring{$\cT$-$\infty$}{T-infinity}-categories of \texorpdfstring{$\cO$}{O}-algebras}

Let $\cO$ be a $\cT$-$\infty$-operad. In this section, we study $\cT$-limits and $\cT$-colimits in the $\cT$-$\infty$-category of $\cO$-algebras within an $\cO$-monoidal $\cT$-$\infty$-category. Our results are straightforward generalizations of Lurie's results in \cite[\S 3.2]{HA} and overlap with similar work of Bachmann--Hoyois undertaken in \cite[\S 7]{BACHMANN2021} (in particular, compare \cite[Prop.~7.6]{BACHMANN2021}). Before reading this section, the reader should first review \cite[Thm.~4.16 and Cor.~4.17]{Exp2b} on how to compute $\cT$-(co)limits in a $\cT$-$\infty$-category of sections.

\subsection{Parametrized (co)limits in general}

In this section, let $\cC^\otimes \to \cO^\otimes$ be an $\cO$-monoidal $\cT$-$\infty$-category, let $\cP^\otimes \to \cO^\otimes$ be a fibration of $\cT$-$\infty$-operads, and let $\CMcal{K} = \{ \CMcal{K}_V : V \in \cT \} $ be a collection of classes $\CMcal{K}_V$ of small $\cT^{/V}$-$\infty$-categories closed with respect to base-change in $\cT$. We are interested in criteria for when the $\cT$-$\infty$-category $\underline{\Alg}_{\cO,\cT}(\cP,\cC)$ of algebras strongly admits $\CMcal{K}$-indexed $\cT$-limits and colimits. To solve this problem, we will first need to understand how to compute $\cT$-limits and $\cT$-colimits in an indexed product.


\begin{lemma} \label{lem:CoinductionCreatesParamCoLimits}
Let $f: \cT_0 \to \cT_1$ be a categorical fibration of $\infty$-categories, let $\cC$ be a $\cT_0$-$\infty$-category, and let $\cK$ be a $\cT_1$-$\infty$-category. Let
\[ \adjunct{f^\ast}{\Cat_{\cT_1}}{\Cat_{\cT_0}}{f_\ast} \]
denote the restriction-coinduction adjunction.
\begin{enumerate}
\item Let $\left( \adjunct{F}{\cC}{\cD}{G} \right)$ be a $\cT_0$-adjunction. Then $\left( \adjunct{f_\ast F}{f_\ast \cC}{f_\ast \cD}{f_\ast G} \right)$ is a $\cT_1$-adjunction.
\item We have a canonical equivalence
$$\underline{\Fun}_{\cT_1}(\cK, f_\ast \cC) \simeq f_\ast \underline{\Fun}_{\cT_0}(f^\ast \cK, \cC)$$
under which $\delta_{\cK} \simeq f_{\ast}(\delta_{f^\ast \cK})$ as $\cT_1$-functors with common domain $f_\ast \cC$, where $\delta_{\cK}$, resp. $\delta_{f^\ast \cK}$ is the constant $\cK$-diagram $\cT_1$-functor, resp. constant $f^\ast \cK$-diagram $\cT_0$-functor.
\item Let $p: \cK \to f_\ast \cC$ be a $\cT_1$-functor and let $q: f^\ast \cK \to \cC$ be its adjoint $\cT_0$-functor. Then $p$ admits a $\cT_1$-limit if and only if $q$ admits a $\cT_0$-limit, and moreover an extension $\overline{p}: \cK^{\underline{\lhd}} \to f_\ast \cC$ is a $\cT_1$-limit diagram if and only if the adjoint extension $\overline{q}: (f^\ast \cK)^{\underline{\lhd}} \simeq f^\ast (\cK^{\underline{\lhd}}) \to \cC$ is a $\cT_0$-limit diagram. The analogous statements also hold for parametrized colimits.
\end{enumerate}
\end{lemma}
\begin{proof}
\noindent (1): It suffices to show that for all $t \in \cT_1$, $(f_\ast F)_t \dashv (f_\ast G)_t$ is an adjunction. In fact, since $(f_\ast \cC)_{t} \simeq \Fun_{\cT_1}((\cT_1^{t/})^\op, f_\ast \cC)$, we can more generally verify that $\Fun_{\cT_1}(\cK,f_\ast F) \dashv \Fun_{\cT_1}(\cK,f_\ast G)$ is an adjunction for every $\cT_1$-$\infty$-category $\cK$. But this holds since
\[ \adjunct{\Fun_{\cT_0}(f^\ast \cK,F)}{\Fun_{\cT_0}(f^\ast \cK, \cC)}{\Fun_{\cT_0}(f^\ast \cK, \cD)}{\Fun_{\cT_0}(f^\ast \cK,G)} \]
is an adjunction by \cite[Prop.~8.2]{Exp2}.

\noindent (2): We check the claimed equivalence at the level of representable functors:
\begin{align*}
\Fun_{\cT_1}(\cL, \underline{\Fun}_{\cT_1}(\cK,f_\ast \cC)) \simeq \Fun_{\cT_1}(\cL \times_{\cT_1^\op} \cK,f_\ast \cC) \simeq \Fun_{\cT_0}(f^\ast \cL \times_{\cT_0^\op} f^\ast \cK, \cC) \\
\simeq \Fun_{\cT_0}(f^\ast \cL, \underline{\Fun}_{\cT_0}(f^\ast \cK, \cC)) \simeq \Fun_{\cT_0}(\cL, f_\ast \underline{\Fun}_{\cT_0}(f^\ast \cK, \cC))
\end{align*}
The assertion about constant functors is shown in a similar manner.

\noindent (3): This follows from combining (1) and (2).
\end{proof}

\begin{corollary} \label{cor:ParametrizedFibersInOperadAdmitCoLimits}
Let $\cC^\otimes \to \cO^\otimes$ be a fibration of $\cT$-$\infty$-operads, let $f: U \simeq \coprod_{1 \leq i \leq n} U_i \to V$ be a morphism in $\FinT$ with $U_i$ and $V$ orbits, let $x \in \cO^\otimes_{f_+}$, and let $e_i: x \to x_i$ be inert edges in $\cO^\otimes$ lifting the characteristic morphisms $\chi_{[U_i \subset U]}$ in $\uFinpT$. Suppose $\cC_{\underline{x_i}}$ admits all $\CMcal{K}_{U_i}$-indexed $\cT^{U_i/}$-(co)limits for each $1 \leq i \leq n$. Then $\cC^\otimes_{\underline{x}}$ admits all $\CMcal{K}_V$-indexed $\cT^{V/}$-(co)limits.
\end{corollary}
\begin{proof}
Combine \cref{lem:CoinductionCreatesParamCoLimits}(3) and the Segal equivalence $\cC^\otimes_{\underline{x}} \simeq \prod_{1 \leq i \leq n} \prod_{f_i} \cC_{\underline{x_i}}$ of \cref{prp:SegalCondition}.
\end{proof}

We may now prove our main result on parametrized limits.

\begin{theorem} \label{thm:LimitsInAlgebras}
Let $\cC$ be an $\cO$-monoidal $\cT$-$\infty$-category and let $\cP^\otimes \to \cO^\otimes$ be a fibration of $\cT$-$\infty$-operads. Let $\CMcal{K} = \{ \CMcal{K}_V : V \in \cT \} $ be a collection of classes $\CMcal{K}_V$ of small $\cT^{/V}$-$\infty$-categories closed with respect to base-change in $\cT$ (e.g., we could take $\CMcal{K}_V$ to be all small $\cT^{/V}$-$\infty$-categories for each $V \in \cT$). Suppose for all $x \in \cO_V$ that $\cC_{\underline{x}}$ admits all $\CMcal{K}_V$-indexed $\cT^{/V}$-limits. Then:
\begin{enumerate}
\item Both $\underline{\Alg}_{\cO,\cT}(\cP,\cC)$ and $\underline{\Fun}_{/\cO^\otimes,\cT}(\cP^\otimes, \cC^\otimes)$ strongly admit all $\CMcal{K}$-indexed $\cT$-limits, and the inclusion
\[ \underline{\Alg}_{\cO,\cT}(\cP,\cC) \subset \underline{\Fun}_{/\cO^\otimes,\cT}(\cP^\otimes, \cC^\otimes) \]
strongly preserves all $\CMcal{K}$-indexed $\cT$-limits.
\item $\underline{\Fun}_{/\cO,\cT}(\cP,\cC)$ strongly admits all $\CMcal{K}$-indexed $\cT$-limits, and the forgetful $\cT$-functor
$$\mathrm{U}: \underline{\Alg}_{\cO,\cT}(\cP,\cC) \to \underline{\Fun}_{/\cO,\cT}(\cP,\cC)$$
strongly creates all $\CMcal{K}$-indexed $\cT$-limits.
\end{enumerate}
\end{theorem}
\begin{proof}
By \cite[Cor.~4.17]{Exp2b} and \cref{cor:ParametrizedFibersInOperadAdmitCoLimits}, $\underline{\Fun}_{/\cO^\otimes,\cT}(\cP^\otimes, \cC^\otimes)$ strongly admits all $\CMcal{K}$-indexed $\cT$-limits. Moreover, by the explicit formula for parametrized limits in $\underline{\Fun}_{/\cO^\otimes,\cT}(\cP^\otimes, \cC^\otimes)$ given in \cite[4.16(3)]{Exp2b}, we see the remaining claims follow from the observation that for every fiberwise inert edge $e: x \to x'$ in $\cO^\otimes_V$, the associated pushforward functor $e_!: \cC^\otimes_{\underline{x}} \to \cC^\otimes_{\underline{x'}}$ is identified with projection to a subset of factors in a fiber product under the Segal equivalence of \cref{prp:SegalCondition}, so in particular preserves all $\CMcal{K}_V$-indexed $\cT^{/V}$-limits. This proves (1). (2) then follows by invoking \cite[4.16(3)]{Exp2b} once more.
\end{proof}


Next, we handle the case of parametrized colimits, for which we will need an additional distributivity assumption on $\cC^\otimes \to \cO^\otimes$. Let $\CMcal{K}_V$ be the class $\CMcal{K}^{\sift}_V$ of $\cT^{/V}$-sifted $\cT^{/V}$-colimits and write $\CMcal{K}^{\sift} = \CMcal{K}$.



\begin{theorem} \label{thm:ColimitsInAlgebras}
Suppose $\cC$ is a distributive $\cO$-monoidal $\cT$-$\infty$-category (\cref{def:DistributiveMonoidalCategory}), and let $\cP^\otimes \to \cO^\otimes$ be a fibration of $\cT$-$\infty$-operads. Then:
\begin{enumerate}
\item Both $\underline{\Alg}_{\cO,\cT}(\cP,\cC)$ and $\underline{\Fun}_{/\cO^\otimes,\cT}(\cP^\otimes, \cC^\otimes)$ strongly admit all $\CMcal{K}^{\sift}$-indexed $\cT$-colimits, and the inclusion
\[ \underline{\Alg}_{\cO,\cT}(\cP,\cC) \subset \underline{\Fun}_{/\cO^\otimes,\cT}(\cP^\otimes, \cC^\otimes) \]
strongly preserves all $\CMcal{K}^{\sift}$-indexed $\cT$-colimits.
\item $\underline{\Fun}_{/\cO,\cT}(\cP,\cC)$ is $\cT$-cocomplete and the forgetful $\cT$-functor
$$\mathrm{U}: \underline{\Alg}_{\cO,\cT}(\cP,\cC) \to \underline{\Fun}_{/\cO,\cT}(\cP,\cC)$$
strongly creates all $\CMcal{K}^{\sift}$-indexed $\cT$-colimits.
\item $\underline{\Alg}_{\cO,\cT}(\cP,\cC)$ is $\cT$-cocomplete.
\item Suppose in addition that $\cC$ is fiberwise presentable. Then $\underline{\Alg}_{\cO,\cT}(\cP,\cC)$ is fiberwise presentable.
\end{enumerate}
\end{theorem}
\begin{proof}
(1) and (2): To show the claim for $\underline{\Fun}_{/\cO^\otimes,\cT}(\cP^\otimes, \cC^\otimes)$, we verify the criterion of \cite[Thm.~4.16(4)]{Exp2b}. Since a fiberwise morphism $\alpha:x \to y$ in $\cO^\otimes_V$ factors as the composite of a fiberwise inert edge and a fiberwise active edge, and the pushforward functor associated to a fiberwise inert edge is a projection, we may suppose $\alpha$ is fiberwise active. Moreover, using again the Segal equivalence of \cref{prp:SegalCondition}, we may suppose that $\alpha$ covers a fiberwise active edge $f_+: [U_+ \ra V] \to [V_+ \ra V]$ in $\uFinpT$ defined by a map $f: U \to V$ of finite $\cT$-sets. Then using the distributive hypothesis on $\cC^\otimes$ together with \cite[Prop.~8.19]{Exp2b}, we have that the pushforward $\cT^{/V}$-functor $\alpha_! = \otimes_{\alpha}: \cC^\otimes_{\underline{x}} \to \cC_{\underline{y}}$ preserves all $\CMcal{K}^{\sift}_V$-indexed colimits, so $\underline{\Fun}_{/\cO^\otimes,\cT}(\cP^\otimes, \cC^\otimes)$ strongly admits all $\CMcal{K}^{\sift}$-indexed $\cT$-colimits. Similarly, we see that $\underline{\Fun}_{/\cO,\cT}(\cP,\cC)$ is $\cT$-cocomplete. The remaining claims then follow as in the proof of \cref{thm:LimitsInAlgebras}, now using \cite[Thm.~4.16(4)]{Exp2b}.

(3): By part (1) and \cite[Cor.~12.15]{Exp2} (or \cite[Thm.~8.6]{Exp2b}), it suffices to check that $\underline{\Alg}_{\cO,\cT}(\cP,\cC)$ admits finite $\cT$-coproducts. For this, we employ the same strategy as in the proof of \cite[Cor.~3.2.3.3]{HA}. Pulling back $\cC^\otimes \to \cO^\otimes$ via $\cP^\otimes \to \cO^\otimes$, we may suppose $\cP^\otimes = \cO^\otimes$ without loss of generality. Now let $\cO^\otimes_{\inert} = \cO^\otimes \times_{\uFinpT} \Triv_{\cT}^\otimes$ and note that by \cref{lem:InertSubcategoryIsRightKanExtended} and \cref{cor:TrivialOperad}, we have that $\underline{\Alg}_{\cO, \cT}(\cO_{\inert},\cC) \simeq \underline{\Fun}_{/\cO, \cT}(\cO,\cC)$. By \cref{thm:globalOperadicKanExtn}, we thus obtain the free-forgetful $\cT$-adjunction
\[ \adjunct{\mathrm{F}}{\underline{\Fun}_{/\cO, \cT}(\cO,\cC)}{\underline{\Alg}_{\cO,\cT}(\cC)}{\mathrm{U}} \]
as an instance of $\cT$-operadic left Kan extension along $\cO^\otimes_{\inert} \to \cO^\otimes$. This implies that each fiber $\underline{\Alg}_{\cO,\cT}(\cC)_V$ admits finite coproducts of free objects, and for each morphism $\alpha: V \to W$ in $\cT$, the putative left adjoint $\alpha_!$ to the restriction functor $\alpha^\ast: \underline{\Alg}_{\cO,\cT}(\cC)_W \to \underline{\Alg}_{\cO,\cT}(\cC)_V$ is at least defined on the full subcategory of free objects, using the pointwise criterion for the existence of an adjoint. By part (2) and the observation that $\mathrm{U}$ is fiberwise conservative, the assumptions of \cite[Prop.~4.7.3.14]{HA} are satisfied for each of the adjunctions $\mathrm{F}_V \dashv \mathrm{U}_V$, so for each object $A \in \underline{\Alg}_{\cO,\cT}(\cC)_V$, there exists a simplicial object $A_{\bullet}$ such that each $A_n$ is free and $A \simeq |A_{\bullet}|$. It follows that the requisite finite coproducts and left adjoints exist for $\underline{\Alg}_{\cO,\cT}(\cC)$. Finally, verification of the base-change condition also reduces to free objects in the same way.

(4): Upon replacing $\cT$ by $\cT^{/V}$ this amounts to showing that $\Alg_{\cO,\cT}(\cP,\cC)$ is presentable. Given (3), it remains to show that $\Alg_{\cO,\cT}(\cP,\cC)$ is accessible. But since all the pushforward functors $\alpha_!: \cC^\otimes_x \to \cC^\otimes_{x'}$ indexed by morphisms $\alpha: x \to x' \in \cO^\otimes$ preserve sifted colimits (which reduces to the aforementioned assertion for fiberwise active $\alpha$ given the inert-fiberwise active factorization on $\cO^\otimes$), this follows from \cite[Prop.~5.4.7.11]{HTT} in exactly the same way as in the proof of \cite[Cor.~3.2.3.5]{HA}.
\end{proof}

\begin{corollary}
Suppose $\cC$ is a distributive $\cO$-monoidal $\cT$-$\infty$-category. Then the free-forgetful $\cT$-adjunction
\[ \adjunct{\mathrm{F}}{\underline{\Fun}_{/\cO, \cT}(\cO,\cC)}{\underline{\Alg}_{\cO,\cT}(\cC)}{\mathrm{U}} \]
of \cref{thm:globalOperadicKanExtn} applied to $\cO^\otimes_{\inert} \subset \cO^\otimes$ is fiberwise monadic.
\end{corollary}
\begin{proof}
We verify the hypotheses of the Barr--Beck--Lurie Theorem \cite[Thm.~4.7.3.5]{HA}. After \cref{thm:ColimitsInAlgebras}(2), it only remains to note that $\mathrm{U}$ is fiberwise conservative, which is immediate from the definitions.
\end{proof}

\subsection{Units and initial objects}

In this subsection, we identify $\cT$-initial objects in $\underline{\Alg}_{\cO,\cT}(\cC)$ in the case where $\cO^\otimes$ is a \emph{unital} $\cT$-$\infty$-operad and $\cC^\otimes$ is any $\cO$-monoidal $\cT$-$\infty$-category.

\begin{definition} \label{def:UnitalOperad}
Let $\cO^\otimes$ be a $\cT$-$\infty$-operad. We say that $\cO^\otimes$ is \emph{unital} if for all orbits $V \in \cT$ and objects $x \in \cO_V$, the space of multimorphisms $\Mul_{\cO}(\emptyset, x)$ is contractible.
\end{definition}

For example, $\uFinpT$ is unital and $\Triv^\otimes_{\cT}$ is not unital. We next introduce the minimal $\cT$-suboperad of $\uFinpT$ which remains unital.

\begin{definition}
Let $\EE^{\otimes}_{0, \cT} \subset \uFinpT$ be the $\cT$-suboperad given by the wide subcategory on those morphisms
\[ \begin{tikzcd}[row sep=4ex, column sep=4ex, text height=1.5ex, text depth=0.25ex]
U \ar{d} & Z \ar{l} \ar{d} \ar{r}{m} & X \ar{d} \\
V & Y \ar{l} \ar{r}{=} & Y
\end{tikzcd} \]
for which $m$ is a summand inclusion.
\end{definition}

Given any $\cT$-$\infty$-operad $\cO^\otimes$, we will then write $\cO^\otimes_0$ for the pullback $\EE^\otimes_{0,\cT} \times_{\uFinpT} \cO^\otimes$ in this subsection. Note that the inclusion $\EE^\otimes_{0,\cT} \subset \uFinpT$ is stable under equivalences and is thus a fibration of $\cT$-$\infty$-operads, and the same is true for the pullback $\cO^\otimes_0 \subset \cO^\otimes$.

\begin{remark} \label{rem:ZeroObjectUnitalOperad}
Let $\cO^\otimes$ be a $\cT$-$\infty$-operad and for each orbit $V \in \cT$, let $\ast_V$ be a choice of object in the fiber $\cO^\otimes_{[\emptyset_+ \ra V]} \simeq \ast$, which is unique up to contractible choice. We note that $\ast_V$ is a final object in $\cO^\otimes_V$. Indeed, if $\cO^\otimes = \uFinpT$, then $[\emptyset_+ \ra V]$ is a zero object in $(\uFinpT)_V \simeq \FinpT[V]$, and the general case follows by the definition of a $\cT$-$\infty$-operad. Since for each $\alpha: V \to W \in \cT$ we have that $\alpha^\ast(\ast_W) \simeq \ast_V$, the $\ast_V$ assemble to define a $\cT$-final object $\ast: \cT^\op \to \cO^\otimes$.

Now suppose $\cO^\otimes$ is unital. Then by the same reasoning, we see that $\ast_V$ is a zero object in $\cO^\otimes_V$ and hence $\ast$ is a $\cT$-zero object. In this case, we will also write $0_V$ for $\ast_V$ and $0$ for $\ast$.
\end{remark}

\begin{definition} \label{def:HomotopiesUnital}
We define the $\cT$-functor $\omega: \Delta^1 \times \cT^\op \to \uFinpT$ to be the unique homotopy from $0$ to $I(-)_+$. For a unital $\cT$-$\infty$-operad $\cO^\otimes$, we then define the $\cT$-functor $\omega_{\cO}$ lying in the 
commutative diagram
\[ \begin{tikzcd}
\cO^{\underline{\lhd}} \ar{r}{\omega_\cO} \ar{d} & \cO^\otimes_0 \ar{d} \ar{r} & \cO^\otimes \ar{d} \\
\Delta^1 \times \cT^\op \ar{r}{\omega} & \EE^\otimes_{0,\cT} \ar{r} & \uFinpT
\end{tikzcd} \]
to be the unique $\cT$-functor extending the inclusion $\cO \subset \cO^\otimes_0$ of the underlying $\infty$-category, whose restriction along the cone inclusion $\cT^\op \subset \cO^{\underline{\lhd}}$ is $0$.

For an $\cO$-monoidal $\cT$-$\infty$-category $p: \cC^\otimes \to \cO^\otimes$, we define the $\cT$-functor $\wt{\omega}_{\cC}$ lying in the commutative diagram
\[ \begin{tikzcd}
\cO^{\underline{\lhd}} \ar{r}{\wt{\omega}_{\cC}} \ar{rd}[swap]{\omega_{\cO}} & \cC^\otimes_0 \ar{r} \ar{d} & \cC^\otimes \ar{d}{p} \\
& \cO^\otimes_0 \ar{r} & \cO^\otimes 
\end{tikzcd} \]
to be the unique lift of $\omega_{\cO}$ so that $\wt{\omega}_{\cC}$ sends the edges $(0_V \to x) \in \cO_V^{\lhd}$ to $p$-cocartesian edges (and hence all edges to $p$-cocartesian edges).

We then define the \emph{unit} of $(\cC^\otimes,p)$ to be the $\cT$-functor $1 \coloneq \wt{\omega}_{\cC}|_{\cO}: \cO \to \cC$.
\end{definition}

\begin{remark}
The $\cT$-functors $\omega$, $\omega_{\cO}$, and $\wt{\omega}_{\cC}$ of \cref{def:HomotopiesUnital} may be defined rigorously as follows. For $\omega$, choose a section $\sigma$ of the trivial fibration $\cT^\op \times_{0, \EE_{0,\cT}^\otimes, \ev_0} \Ar_{\cT}(\EE_{0,\cT}^\otimes) \xto{\ev_1} \EE_{0,\cT}^\otimes$ of \cref{lem:InitialObjectTrivialFibration} and define $\omega$ to be the adjoint to the composite $\omega^\perp = \pr \circ \sigma \circ I_+: \cT^\op \to \Ar_{\cT}(\EE_{0,\cT}^\otimes)$.

Then for $\omega_\cO$, choose a section $\sigma_{\cO}$ of the trivial fibration $\psi$ of \cref{lem:InitialObjectTrivialFibration} applied to $\cO^\otimes_0 \to \EE_{0,\cT}^\otimes$, let $\rho: \cO \to \cT^\op$ and $i: \cO \subset \cO^\otimes_0$ denote the structure map and inclusion, and define the $\cT$-functor
$$\omega^{\perp}_{\cO} = \sigma \circ (\omega^\perp \rho, i): \cO \to (\cT^\op \times_{\EE_{0,\cT}^\otimes} \Ar_{\cT}(\EE_{0,\cT}^\otimes)) \times_{\EE_{0,\cT}^\otimes} \cO^\otimes_0 \xto{\simeq} \cT^\op \times_{0, \cO_0^\otimes, \ev_0} \Ar_{\cT}(\cO_0^\otimes).$$
By \cite[Cor.~4.27]{Exp2} and \cite[Prop.~4.30]{Exp2}, for any $\cT$-$\infty$-category $\cD$ and cocartesian section $\phi: \cT^\op \to \cD$ we have natural equivalences $$\cD_{(\phi,\cT)/} \xto{\simeq} \cD^{(\phi,\cT)/} \xto{\simeq} \cT^\op \times_{\phi, \cD} \Ar_{\cT}(\cD),$$ and by \cite[Prop.~4.25]{Exp2} for any $\cT$-$\infty$-category $\cA$ we have the natural $\cT$-join and slice equivalence $$\underline{\Fun}_{\cT}(\cA,\cD_{(\phi,\cT)/}) \xto{\simeq} \underline{\Fun}_{\cT/ /\cT}(\cA^{\underline{\lhd}}, \cD).$$We may thus adjoin $\omega^{\perp}_{\cO}$ to define $\omega_{\cO}$ so that it fits into the indicated commutative diagram over $\omega$.

Finally, to define $\wt{\omega}_{\cC}$, first let $\ast: \cT^\op \to \cC_0^\otimes$ be a lift of $0$ to the $\cT$-final object of \cref{rem:ZeroObjectUnitalOperad}, and let $\sigma_{\cC}: \cC_0^\otimes \times_{\cO_0^\otimes} \Ar(\cO_0^\otimes) \xto{\simeq} \Ar^{\cocart}(\cC_0^\otimes)$ be a choice of section for the trivial fibration. We then have the composite
\[ \cT^\op \times_{0, \cO_0^\otimes} \Ar(\cO_0^\otimes) \xto{\ast \times \id} \cC^\otimes_0 \times_{\cO_0^\otimes} \Ar(\cO_0^\otimes) \xto{\sigma_{\cC}} \Ar^{\cocart}(\cC_0^\otimes) \subset \Ar(\cC_0^\otimes), \]
which restricts to
\[ f: \cT^\op \times_{0, \cO_0^\otimes} \Ar_{\cT}(\cO_0^\otimes) \to \cT^\op \times_{\ast, \cC_0^\otimes} \Ar_{\cT}(\cC_0^\otimes). \]
The adjoint of $f \circ \omega^{\perp}_{\cO}$ then defines the lift $\wt{\omega}_{\cC}$ of $\omega_{\cO}$.

We also verify the uniqueness assertion for $\wt{\omega}_{\cC}$ and leave that for $\omega_{\cO}$ as an exercise for the reader. By \cref{lem:InitialObjectCocartesianEquivalence}, the functor given by restriction along the $\cT$-cone point
\[ \Fun^{\cocart}_{/\cO^\otimes_0}(\cO^{\underline{\lhd}}, \cC^\otimes_0) \to \Fun_{\cT}(\cT^\op, \cC^\otimes \times_{\cO^\otimes, 0} \cT^\op)  \]
is an equivalence, and since $\cC^\otimes \times_{\cO^\otimes, 0} \cT^\op \simeq \cT^\op$, we see that the righthand side is contractible, which shows the claim (and gives another construction of $\wt{\omega}_{\cC}$).
\end{remark}

\begin{lemma} \label{lem:InitialObjectTrivialFibration}
Suppose $q: \cD \to \cB$ is a $\cT$-fibration and $0: \cT^\op \to \cD$ is a $\cT$-functor such that $0$ and $q \circ 0$ are $\cT$-initial objects. Then the $\cT$-functor
\[ \psi: (\cT^\op \times_{0, \cD, \ev_0} \Ar_{\cT}(\cD)) \to (\cT^\op \times_{0, \cB, \ev_0} \Ar_{\cT}(\cB)) \times_{\ev_1,\cB,q} \cD \]
is a trivial fibration.
\end{lemma}
\begin{proof}
Since $q$ is a categorical fibration, it follows that $\psi$ is as well. It thus suffices to prove that $\psi$ is a categorical equivalence, for which we may suppose that $\cB = \cT^\op$ by the two-out-of-three property of equivalences. The claim then follows by our hypothesis that $0_V$ is initial for all $V \in \cT$.
\end{proof}

\begin{lemma} \label{lem:InitialObjectCocartesianEquivalence}
Let $q: \cD \to \cB$ be a $\cT$-cocartesian fibration and let $i: \cT^\op \to \cB$ be a $\cT$-initial object. Then the $\cT$-functor
\[ i^\ast: \underline{\Fun}^{\cocart}_{/\cB, \cT}(\cB, \cD) \xto{\simeq} \cT^{\op} \times_{\cB} \cD \]
is an equivalence.
\end{lemma}
\begin{proof}
It suffices to check the assertion fiberwise. Replacing $\cT$ by $\cT^{/V}$, we may further suppose that $\cT$ has a final object $\ast$, and we reduce to showing that $i^{\ast}: \Fun^{\cocart}_{/\cB}(\cB, \cD) \to \Fun^{\cocart}_{/\cT^\op}(\cT^\op, \cT^\op \times_{\cB} \cD)$ is an equivalence of $\infty$-categories, or equivalently, that $(\cT^\op)^\sharp \to \cB^\sharp$ is a cocartesian equivalence in $\sSet^+_{/\cB}$. But the inclusion of the initial object $\ast \in \cT^\op$ is a cocartesian equivalence to both $(\cT^\op)^\sharp$ and $\cB^\sharp$, so by the two-out-of-three property of the cocartesian equivalences we are done.
\end{proof}

We may canonically endow the unit of an $\cO$-monoidal $\cT$-$\infty$-category with the structure of an $\cO$-algebra in the following way.

\begin{proposition} \label{prop:UnitCanonicalAlgebra}
Let $\cO^\otimes$ be a unital $\cT$-$\infty$-operad and let $p: \cC^\otimes \to \cO^\otimes$ be an $\cO$-monoidal $\cT$-$\infty$-category. Then there is a unique cocartesian section $1^\otimes$ of $p$ such that $1^\otimes$ extends the unit $1: \cO \to \cC$.
\end{proposition}
\begin{proof}
Since $\cO^\otimes$ has the $\cT$-initial object $0$ by assumption, the claim follows from \cref{lem:InitialObjectCocartesianEquivalence}.
\end{proof}


We next identify $\cO^\otimes_0$-monoidal $\cT$-$\infty$-categories and $\cO_0$-algebras therein in more familiar terms.

\begin{proposition} \label{prop:E0MonoidalCats}
Let $\cO^\otimes$ be a unital $\cT$-$\infty$-operad.
\begin{enumerate}
\item Suppose $(\cC^\otimes,p)$ is a $\cO$-monoidal $\cT$-$\infty$-category and let $F^{\un}_{\cC}: \cO^\otimes_0 \to \Cat$ be the functor classifying the cocartesian fibration $p|_{\cO^\otimes_0}$. Then $F^{\un}_{\cC}$ is the right Kan extension of its restriction $f^{\un}_{\cC}$ along $\omega_{\cO}$.
\item Let $f: \cO^{\underline{\lhd}} \to \Cat$ be a functor whose restriction along the cone inclusion $\cT^\op \subset \cO^{\underline{\lhd}}$ is constant at the final object, and let $F$ be the right Kan extension of $f$ along $\omega_{\cO}$. The cocartesian fibration $q: \cD \to \cO^\otimes_0$ classifed by $F$ is then an $\cO_0$-monoidal $\cT$-$\infty$-category.
\item The assignment $(\cC^\otimes,p) \mapsto (\ast \ra f_{\cC})$ (where $f_{\cC}$ denotes the restriction of $f^{\un}_{\cC}$ to $\cC$, regarded as pointed in $\Fun(\cO,\Cat)$ via the unit $1: \cO \to \cC$) implements an equivalence of $\infty$-categories
\[ \mu: \Cat^\otimes_{\cO_0} \xto{\simeq} \Fun(\cO,\Cat)^{\ast/} \]
and an equivalence of $\cT$-$\infty$-categories
\[ \underline{\Cat}^\otimes_{\cO_0} \xto{\simeq} \underline{\Fun}_{\cT}(\cO, (\underline{\Cat}_{\cT})^{(\ast,\cT)/}). \]
\item For any two $\cO$-monoidal $\cT$-$\infty$-categories $\cC^\otimes$ and $\cD^\otimes$, $\mu$ induces an equivalence of $\infty$-categories
\[ \Fun^\otimes_{\cO_0,\cT}(\cC_0, \cD_0) \xto{\simeq} \Nat_{\ast}(f_{\cC}, f_{\cD}). \]
\end{enumerate}
\end{proposition}
\begin{proof}
We first analyze right Kan extension along $\omega_{\cO}$ in general. Let $x \in \cO^\otimes_{[U_+ \ra V]}$, let $U \simeq \coprod_{i=1}^n U_i$ be an orbit decomposition, and let $\rho_i: x \to x_i$ be cocartesian edges lifting the characteristic morphisms $\chi_{[U_i \subset U]}$. Let $\cA, \cB \subset \cO^{\underline{\lhd}} \times_{\cO^\otimes_0} (\cO^\otimes_0)^{x/}$ be the full subcategories on objects over
\[ \left( \begin{tikzcd}
U \ar{d} & Z \ar{l} \ar{r} \ar{d}{} & X \ar{d} \\
V & Y \ar{l}  \ar{r}{=} & Y
\end{tikzcd} \right) \in (\Delta^1 \times \cT^\op) \times_{\EE^\otimes_{0,\cT}} (\EE^\otimes_{0,\cT})^{[U_+ \ra V]/} \]
such that $Z = \emptyset$ and $Z \neq \emptyset$, respectively, and note that $\cO^{\underline{\lhd}} \times_{\cO^\otimes_0} (\cO^\otimes_0)^{x/}$ decomposes as the disjoint union of $\cA$ and $\cB$. Let $\phi = \phi_A \sqcup \phi_B: \{ a \} \sqcup \Orbit(U) \to \cA \sqcup \cB$ be the functor that sends $a$ to $(x \ra 0_V)$ and $U_i$ to $\rho_i$.

We claim that $\phi$ is right cofinal. By the same argument as in the proof of \cref{lem:InertSubcategoryIsRightKanExtended}, $\phi_B$ is right cofinal, so it only remains to show $(x \ra 0_V)$ is an initial object in $\cA$. Since $\cO^{\underline{\lhd}} \times_{\cO^\otimes_0} (\cO^\otimes_0)^{x/} \to \cO^{\underline{\lhd}} \to \cT^\op$ is a cocartesian fibration, it suffices to show that for all morphisms $\alpha: W \to V \in \cT$, $\alpha^\ast (x \ra 0_V) \simeq (x \ra 0_W)$ is an initial object in $\cB_W$. We first check that for objects $(x \ra y) \in \cB$ covering $\gamma: [U_+ \ra V] \to I(W)_+$, the mapping space $\Map_{\cB_W}(x \ra 0_W, x \ra y)$ is contractible. This mapping space fits into the commutative diagram
\[ \begin{tikzcd}
\Map_{\cB_W}(x \ra 0_W, x \ra y) \ar{d} \ar{r} & \Map_{(\cO^\otimes_0)^{x/}}(x \ra 0_W, x \ra y) \ar{r} \ar{d} & \ast \ar{d}{(x \ra y)} \\
\ast \simeq \Map_{\cO^{\lhd}_W}(0_W,y) \ar{r} & \Map_{\cO^\otimes_0}(0_W,y) \ar{r}{(x \ra 0_W)^\ast} & \Map_{\cO^\otimes_0}(x,y),
\end{tikzcd} \]
where the lower horizontal composite selects the inert-fiberwise active factorization of $x \to y$ in the connected component $\Map^{\gamma}_{\cO^\otimes_0}(x,y)$. In fact, since $\Mul_{\cO}(\emptyset, y) \simeq \ast$, this map is an equivalence onto that connected component, and the contractibility of $\Map_{\cB_W}(x \ra 0_W, x \ra y)$ follows. Finally, the argument for $(x \ra 0_W)$ itself proceeds the same way.

We conclude that given a functor $f: \cO^{\underline{\lhd}} \to \Cat$, the right Kan extension $(\omega_{\cO})_{\ast} f$ sends $x$ to $f(0_V) \times \prod_{i=1}^n f(x_i)$. This shows (1) -- more precisely, the unit map $F^{\un}_{\cC} \xto{\simeq} (\omega_{\cO})_{\ast} (\omega_{\cO})^{\ast} F^{\un}_{\cC}$ is seen to be an equivalence. By \cref{prop:SimplifyOMonoidalDef}, this also shows (2). Moreover, we see that $(\omega_{\cO})^\ast (\omega_{\cO})_{\ast} f \xto{\simeq} f$ if and only if $f|_{\cT^\op}$ is constant at $\ast \in \Cat$. We deduce that the adjunction $(\omega_{\cO})^\ast \dashv (\omega_{\cO})_{\ast}$ restricts to an adjoint equivalence
\[ \adjunct{(\omega_{\cO})^\ast}{\Cat^\otimes_{\cO_0}}{\Fun'(\cO^{\underline{\lhd}},\Cat)}{(\omega_{\cO})_{\ast}} \]
where we take the righthand side to consist of the full subcategory on functors $\cO^{\underline{\lhd}} \to \Cat$ that restrict to $\ast$ on $\cT^\op$. Note that since the inclusion of a final object is fully faithful, we have an equivalence $$\Fun'(\cO^{\underline{\lhd}},\Cat) \simeq \Delta^0 \times_{\ast, \Fun(\cT^\op,\Cat)} \Fun(\cO^{\underline{\lhd}}, \Cat),$$ and under the equivalence $\Fun_{\cT}(\cO^{\underline{\lhd}}, \Cat) \simeq \Fun_{\cT}(\cO^{\underline{\lhd}}, \underline{\Cat}_{\cT})$ of \cite[Prop.~3.9]{Exp2}, this yields an equivalence
\[ \Fun'(\cO^{\underline{\lhd}},\Cat) \simeq \Fun_{\cT/ /\cT}(\cO^{\underline{\lhd}}, \underline{\Cat}_{\cT}) \]
to the $\infty$-category of $\cT$-functors $\cO^{\underline{\lhd}} \to \underline{\Cat}_{\cT}$ that restrict on $\cT^\op$ to the $\cT$-final object of $\underline{\Cat}_{\cT}$. By the $\cT$-join and slice adjunction of \cite[Prop.~4.25]{Exp2} (together with \cite[Cor.~4.27]{Exp2}), we have an equivalence
$$\underline{\Fun}_{\cT/ /\cT}(\cO^{\underline{\lhd}}, \underline{\Cat}_{\cT}) \simeq \underline{\Fun}_{\cT}(\cO, (\underline{\Cat}_{\cT})^{(\ast, \cT)/}),$$
which we claim yields an equivalence $\Fun_{\cT/ /\cT}(\cO^{\underline{\lhd}}, \underline{\Cat}_{\cT}) \simeq \Fun(\cO,\Cat)^{\ast/}$ upon passage to cocartesian sections -- for this, just examine the pullback square of $\infty$-categories
\[ \begin{tikzcd}
\Fun_{\cT}(\cO, (\underline{\Cat}_{\cT})^{(\ast,\cT)/}) \ar{r} \ar{d} & \Fun_{\cT}(\cO, \underline{\Fun}_{\cT}(\cT^\op \times \Delta^1, \underline{\Cat}_{\cT})) \simeq \Fun(\Delta^1, \Fun(\cO,\Cat)) \ar{d} \\
\Fun_{\cT}(\cO, \cT^\op) \simeq \Delta^0 \ar{r}{\ast} & \Fun_{\cT}(\cO,\underline{\Cat}_{\cT}) \simeq \Fun(\cO, \Cat).
\end{tikzcd} \]
This shows the first part of (3), and the second follows since we have a comparison $\cT$-functor that we just proved is an equivalence fiberwise. The claim of (4) (i.e., that $\mu$ promotes to an equivalence of $(\infty,2)$-categories) then follows since $\mu$ clearly respects cotensors by $\infty$-categories (cf. \cref{constr:OperadCotensors}).
\end{proof}

\begin{theorem} \label{thm:E0Algebras}
Suppose $\cO^\otimes$ is a unital $\cT$-$\infty$-operad and $\cC^\otimes$ is a $\cO$-monoidal $\cT$-$\infty$-category. Then we have a canonical equivalence of $\cT$-$\infty$-categories
\[ \underline{\Alg}_{\cO,\cT}(\cO_0, \cC) \xto{\simeq} \underline{\Fun}_{/\cO,\cT}(\cO, \cC)^{(1,\cT)/}. \]
\end{theorem}
\begin{proof}
By the usual reduction, it will suffice to prove the statement without the `underlining'. Our strategy is to replace $\cO^\otimes_0$ by its $\cO_0$-monoidal envelope and then invoke \cref{prop:E0MonoidalCats}(4). We first identify $\Env_{\cO_0,\cT}(\cO_0)^\otimes = \Ar^{\act}_{\cT}(\cO^\otimes_0)$ in simpler terms. Let
$$\lambda: \cO \times \Delta^1 \to \Ar^{\act}_{\cT}(\cO^\otimes_0) \times_{\cO_0^\otimes} \cO$$
be the $\cT$-functor which for $x \in \cO_V$ sends $(x,0) \to (x,1)$ to the evident morphism $[0_V \ra x] \to \id_x$ of active arrows. More precisely, we may define the adjoint of $\lambda$ projecting to $ \Ar^{\act}_{\cT}(\cO^\otimes_0)$ as the composite
\[ \cO \times \Delta^1 \times \Delta^1 \xto{h} \cO^{\underline{\lhd}} \xto{\omega_{\cO}} \cO^\otimes_0, \]
where to define $h$, we regard $\cO \times (\Delta^1 \times \Delta^1)$ as fibered over $\cT^\op \times \Delta^1$ via the structure map $\pi$ for $\cO$ and $(i,j) \mapsto \max(i,j)$, and let $h$ be given by $(\pi,\pr_{\cO})$ under the defining universal property of the $\cT$-join (cf. \cite[Prop.~4.3]{Exp2}). We then let
$$\overline{\lambda}: \Ar_{\cT}(\cO) \times \Delta^1 \to \Ar^{\act}_{\cT}(\cO^\otimes_0) \times_{\cO_0^\otimes} \cO$$
be the induced morphism of $\cT$-cocartesian fibrations over $\cO$ extending $\lambda$ under the equivalence of \cite[Ex.~3.8]{Exp2b}. We claim that $\overline{\lambda}$ is an equivalence, for which it suffices to check fiberwise. But for every $x \in \cO_V$, we have that $\lambda_x$ is essentially surjective in view of the unital assumption on $\cO^\otimes$ (since the active edges must be of the form $[f: x' \ra x]$ in $\cO_V$ or $[0_V \ra x]$ factoring through some $f$), and an easy computation of mapping spaces shows that $\lambda_x$ is also fully faithful.

By similar reasoning, we also see that the composition
\[ \cO \times\{0\} \subset \cO \times \Delta^1 \to \Ar_{\cT}(\cO) \times \Delta^1 \xto{\simeq} \Ar^{\act}_{\cT}(\cO^\otimes_0) \times_{\cO_0^\otimes} \cO \]
identifies with the unit map for $\Env_{\cO_0,\cT}(\cO_0)^\otimes$. Now let $\cE: \cO \to \Cat$ be the functor obtained by straightening $\ev_1: \Ar_{\cT}(\cO) \to \cO$. We have shown that under the correspondence of \cref{prop:E0MonoidalCats}, $\Env_{\cO_0,\cT}(\cO_0)^\otimes$ straightens to $\Delta^1 \times \cE$. In the notation of that proposition, consider the pullback square
\[ \begin{tikzcd}
\Nat_{\ast}(\Delta^1 \times \cE, f_{\cC}) \ar{r} \ar{d} & \Nat(\Delta^1 \times \cE, f_{\cC}) \ar{d} \\
\Delta^0 \ar{r} & \Nat(\cE, f_{\cC}).
\end{tikzcd} \]
Using the universal property of the free $\cT$-cocartesian fibration, we deduce an equivalence
$$\Nat_{\ast}(\Delta^1 \times \cE, f_{\cC}) \simeq \Fun_{/\cO, \cT}(\cO, \cC)^{(1,\cT)/}.$$
We may now conclude using \cref{prop:E0MonoidalCats}(4).
\end{proof}

\begin{theorem} \label{thm:InitialObjectsInAlgebrasWhenUnital}
Let $\cO^\otimes$ be a unital $\cT$-$\infty$-operad and let $p: \cC^\otimes \to \cO^\otimes$ be an $\cO$-monoidal $\cT$-$\infty$-category.
\begin{enumerate}
\item The $\cO$-algebra $1^\otimes$ of \cref{prop:UnitCanonicalAlgebra} is an initial object of $\Alg_{\cO,\cT}(\cC)$.
\item $\underline{\Alg}_{\cO,\cT}(\cC)$ admits a $\cT$-initial object given fiberwise by $(1^\otimes)_{\underline{V}}$.
\end{enumerate}
\end{theorem}
\begin{proof}
Since units are stable under base-change, it suffices to prove the first assertion. Let $1^\otimes_{\un}: \cO^\otimes_0 \to \cC^\otimes$ be the unique morphism of $\cT$-$\infty$-operads over $\cO^\otimes$ that sends all edges in $\cO^\otimes_0$ to $p$-cocartesian edges in $\cC^\otimes$ (so $1^\otimes_{\un}$ extends $\ast: \cT^\op \to \cC^\otimes$ under the equivalence of \cref{lem:InitialObjectCocartesianEquivalence}). Then $1^\otimes_{\un}$ corresponds to $1 \in \Fun_{/\cO,\cT}(\cO, \cC)^{(1,\cT)/}$ under the equivalence of \cref{thm:E0Algebras} and is hence an initial object of $\Alg_{\cO,\cT}(\cO_0, \cC)$.

Let $i: \cO^\otimes_0 \to \cO^\otimes$ denote the inclusion. It will suffice to show that the left adjoint $i_!$ to the the forgetful functor $i^\ast: \Alg_{\cO,\cT}(\cC) \to \Alg_{\cO,\cT}(\cO_0, \cC)$ is defined on $1^\otimes_{\un}$ and sends $1^\otimes_{\un}$ to $1^\otimes$. Consider the factorization of $i$ through its $\cO$-monoidal envelope
$$\cO^\otimes_0 \xto{\iota} \Env_{\cO,\cT}(\cO_0)^\otimes = \cO^\otimes_0 \times_{\cO^\otimes} \Ar^{\act}_{\cT}(\cO^\otimes) \xto{\ev_1} \cO^\otimes$$
and the resulting sequence of adjunctions
\[ \begin{tikzcd}
\Alg_{\cO,\cT}(\cO_0, \cC) \simeq \Fun^\otimes_{\cO,\cT}(\Env_{\cO,\cT}(\cO_0), \cC) \ar[hookrightarrow, shift left=2]{r}{\iota_!} \ar[dotted]{rd}{i_!} & \Alg_{\cO,\cT}(\Env_{\cO,\cT}(\cO_0), \cC) \ar{l}{\iota^\ast} \ar[dotted, shift left=2]{d}{(\ev_1)_!} \\
& \Alg_{\cO,\cT}(\cC) \ar{u}{(\ev_1)^\ast} \ar[shift left=2]{ul}{i^\ast}
\end{tikzcd} \]
where the dotted left adjoints are not necessarily defined.
Observe that for every orbit $V \in \cT$, the fiber $\Env_{\cO,\cT}(\cO_0)^\otimes_V$ admits an initial object given by $\id_{0_V}$, and these assemble to define a $\cT$-initial object of $\Env_{\cO,\cT}(\cO_0)^\otimes$. By \cref{lem:InitialObjectCocartesianEquivalence}, the $\cO$-monoidal $\cT$-functor $1^\otimes \circ \ev_1$ restricts to $1^\otimes_{\un}$ as they both send \emph{all} edges to $p$-cocartesian edges and extend $\ast$, so $\iota_!(1^\otimes_{\un}) \simeq 1^\otimes \circ \ev_1$. Now observe that for every $x \in \cO^\otimes_V$, the unique map $0_V \ra x$ is active and is an initial object in the fiber $\Env_{\cO,\cT}(\cO_0)^\otimes \times_{\cO^\otimes} \{ x \}$. Consequently, the ordinary left Kan extension of $1^\otimes \circ \ev_1$ along $\ev_1$ exists and is computed by $1^\otimes$ itself. Since the ordinary left Kan extension is an $\cO$-algebra in this case, we conclude that $(\ev_1)_!(1^\otimes \circ \ev_1) \simeq 1^\otimes$ and hence $i_!(1^\otimes_{\un}) \simeq 1^\otimes$.
\end{proof}

\subsection{Indexed coproducts in the \texorpdfstring{$\cT$}{T}-symmetric monoidal case}

We identify finite $\cT$-indexed coproducts with tensor products in the case of a $\cT$-symmetric monoidal $\cT$-$\infty$-category $\cC$, following the strategy of \cite[\S 3.2.4]{HA}. To precisely articulate this identification, we first discuss how to equip $\underline{\CAlg}_{\cT}(\cC)$ with a $\cT$-symmetric monoidal structure.

\begin{construction} \label{con:BifunctorComm}
We define the smash product $\cT$-functor
$$\wedge: \uFinpT \times_{\cT^\op} \uFinpT \to \uFinpT, \quad ([U_+ \ra V], [U'_+ \ra V]) \mapsto [(U \times_V U')_+ \mapsto V]$$
as follows. Recall that given an $\infty$-category $\cD$ with finite products, we can define the smash product functor $\wedge: \cD_{\ast} \times \cD_{\ast} \to \cD_{\ast}$ as the composition of functors
\[ \begin{tikzcd}
\cD_{\ast} \times \cD_{\ast} \subset \cD^{\Delta^1} \times \cD^{\Delta^1} \ar{r}{\times} & \cD^{\Delta^1 \times \Delta^1} \ar[dotted]{r}{\min_!} & \cD^{\Delta^1} \ar{r}{(d^2)_{\ast}} & \cD^{\Lambda^2_0} \ar[dotted]{r}{\colim} & \cD^{(\Lambda^2_0)^{\rhd}} \ar{r}{\ev_{\{2\}^{\rhd}}} & \cD
\end{tikzcd} \]
provided the functors $\min_!$ and $\colim$ exist pointwise on objects $(\ast \ra x, \ast \ra y)$ and $(\ast \la x \vee y \ra x \times y)$, as they then extend to partially defined left adjoints so that the composition exists. If we then have a presheaf $\cD_{\bullet}: \cT^\op \to \Cat$ such that for every $\alpha: V \to W$ in $\cT$, $\alpha^\ast: \cD_W \to \cD_V$ preserves finite products, wedge sums, and cofibers $(x \times y) / (x \vee y)$, it follows from the existence theorem for relative adjunctions (\cite[Prop.~7.3.2.6]{HA} and \cite[Prop.~7.3.2.11]{HA}) that we obtain a $\cT$-functor
\[ \wedge: \underline{\cD_{\ast}} \times_{\cT^\op} \underline{\cD_{\ast}} \to \underline{\cD_{\ast}} \]
given fiberwise by formation of smash products, where $\underline{\cD_{\ast}}$ is the unstraightening of the pointed presheaf $\cD_{\bullet, \ast}$.
\end{construction}

\begin{definition} \label{def:BifunctorOperads}
Let $\cO^\otimes, \cP^\otimes, \cQ^\otimes$ be $\cT$-$\infty$-operads. We say that a $\cT$-functor $F: \cO^\otimes \times_{\cT^\op} \cP^\otimes \to \cQ^\otimes$ is a \emph{bifunctor of $\cT$-$\infty$-operads} if it sends pairs of inert edges to inert edges and the diagram
\[ \begin{tikzcd}
\cO^\otimes \times_{\cT^\op} \cP^\otimes \ar{r}{F} \ar{d} & \cQ^\otimes \ar{d} \\
\uFinpT \times_{\cT^\op} \uFinpT \ar{r}{\wedge} & \uFinpT
\end{tikzcd} \]
is homotopy commutative.\footnote{If $\cT = \ast$, then the smash product for $\Fin_{\ast}$ can be defined without the ambiguity of a contractible space of choices. Therefore, one can choose the square to strictly commute in $\sSet^+_{/(\Fin_{\ast}, \Inert)}$ for the non-parametrized definition of a bifunctor of $\infty$-operads as in \cite[Def.~2.2.5.3]{HA}.}
\end{definition}


\begin{variant} \label{var:BifunctorBigOperads}
By the same construction as in \cref{con:BifunctorComm}, we have a smash product $\FinT^\op$-functor
$$\wedge: \BigFinT \times_{\FinT^\op} \BigFinT \to \BigFinT$$
that can be chosen to extend the smash product on $\uFinpT$. Likewise, we have the analogous definition of a bifunctor of big $\cT$-$\infty$-operads, and the datum of one determines the other essentially uniquely under the correspondence of \cref{cor:BigVsSmallOperads}.
\end{variant}

\begin{conthm} \label{conthm:MonoidalStructureOnAlgebras}
Suppose $F: \cO^\otimes \times_{\cT^\op} \cP^\otimes \to \cQ^\otimes$ is a bifunctor of $\cT$-$\infty$-operads and let the $\cT$-functor $$P: \cO^\otimes \times_{\cT^\op} \Ar(\cT^\op) \to \cO^\otimes, \quad (x, V \xleftarrow{\beta} W) \mapsto \beta^{\ast}(x)$$ be a choice of cocartesian pushforward. Consider the spans of marked simplicial sets
\begin{equation*} \begin{tikzcd}
(\cO^\otimes, \Inert) &  (\cO^\otimes,\Inert) \times_{\cT^\op} \Ar(\cT^\op)^\sharp \times_{\cT^\op} (\cP^\otimes,\Inert) \ar{l}[swap]{\pi} \ar{r}{G} & (\cQ^\otimes, \Inert),
\end{tikzcd} \end{equation*}
\begin{equation*} \begin{tikzcd}
(\cO^\otimes)^{\sharp} & (\cO^\otimes)^\sharp \times_{\cT^\op} \Ar(\cT^\op)^\sharp \times_{\cT^\op} (\cP^\otimes,\Inert) \ar{l}[swap]{\pi} \ar{r}{G} & (\cQ^\otimes)^\sharp,
\end{tikzcd} \end{equation*}
where $G = F \circ (P \times \id_{\cP^\otimes})$ and $\pi$ is the projection to $\cO^\otimes$. Then these spans determine Quillen adjunctions
\begin{align*}
\adjunct{G_! \pi^\ast}{\sSet^+_{/(\cO^\otimes, \Inert)}}{\sSet^+_{/(\cQ^\otimes, \Inert)}}{\pi_{\ast} G^\ast}, \qquad \adjunct{G_! \pi^\ast}{\sSet^+_{/\cO^\otimes}}{\sSet^+_{/\cQ^\otimes}}{\pi_{\ast} G^\ast}
\end{align*}
with respect to the $\cT$-operadic model structures and $\cT$-monoidal model structures. Given a fibration $\cC^\otimes \to \cQ^\otimes$ of $\cT$-$\infty$-operads, we then let
\[ \underline{\Alg}_{\cQ, \cT}(\cP,\cC)^\otimes \to \cO^\otimes \]
denote the resulting fibration of $\cT$-$\infty$-operads given by $\pi_{\ast} G^\ast (\cC^\otimes, \Inert)$.\footnote{Beware that the notation $\underline{\Alg}_{\cQ, \cT}(\cP,\cC)^\otimes$ hides the dependence of the structure map $\cP^\otimes \to \cQ^\otimes$ on the choice of parameter in $\cO^\otimes$.}

If $\cC^\otimes$ is $\cQ$-monoidal, then $\underline{\Alg}_{\cQ, \cT}(\cP,\cC)^\otimes$ is $\cO$-monoidal, and has cocartesian edges marked as in $\pi_{\ast} G^\ast (\leftnat{\cC^\otimes})$.
\end{conthm}
\begin{proof}
Note that the underlying simplicial sets of $\pi_\ast G^\ast(-)$ are the same regardless of whether we work over $(\cQ^\otimes, \Inert)$ or $(\cQ^\otimes)^\sharp$. We first establish the assertion on the Quillen adjunction between $\cT$-operadic model structures. For the proof, it will be convenient to pass to big $\cT$-$\infty$-operads (cf. \cref{var:BifunctorBigOperads}). Let $$\wt{F}: \wt{\cO}^\otimes \times_{\FinT^\op} \wt{\cP}^\otimes \to \wt{\cQ}^\otimes$$ be the bifunctor of big $\cT$-$\infty$-operads extending $F$, let the $\FinT$-functor $$\wt{P}: \wt{\cO}^\otimes \times_{\FinT^\op} \Ar(\FinT^\op) \to \wt{\cO}^\otimes$$ be a choice of cocartesian pushforward over $\FinT^\op$ extending $P$, let $\wt{G} = \wt{F} \circ (\wt{P} \times \id_{\wt{\cP}^\otimes})$, and consider the span of marked simplicial sets
\[ \begin{tikzcd}
(\wt{\cO}^\otimes, \Inert) &  (\wt{\cO}^\otimes,\Inert) \times_{\FinT^\op} \Ar(\FinT^\op)^\sharp \times_{\FinT^\op} (\wt{\cP}^\otimes,\Inert) \ar{l}[swap]{\wt{\pi}} \ar{r}{\wt{G}} & (\wt{\cQ}^\otimes, \Inert).
\end{tikzcd} \]
We claim that this span satisfies the hypotheses of \cite[Thm.~B.4.2]{HA} with respect to the categorical patterns $\wt{\mathfrak{P}}_{\cO}$ and $\wt{\mathfrak{P}}_{\cQ}$ of \cref{def:BigModelStructure}. (2) is clear and (3) is vacuous. By \cite[Cor.~2.4.7.17]{HTT}, the source functor $\ev_0: \Ar(\FinT^\op) \times_{\FinT^\op} \wt{P}^\otimes \to \FinT^\op$ is a cartesian fibration, so the pullback $\wt{\pi}$ is a cartesian fibration. This proves (1) and (4). Moreover, an edge in $\wt{\cO}^\otimes \times_{\FinT^\op} \Ar(\FinT^\op) \times_{\FinT^\op} \wt{\cP}^\otimes$ is $\wt{\pi}$-cartesian if and only if its projection to $\wt{\cP}^\otimes$ is an equivalence, which implies (7).

Now let $f_{x, \phi, \Sigma}: n^{\lhd} \to \wt{\cO}^\otimes$ be as in \cref{def:BigModelStructure}, so that $\phi: U \to V$ is a morphism in $\FinT$, $\Sigma = \{ \sigma_1, ..., \sigma_n \}$ is a collection of commutative squares
\[ \begin{tikzcd}
U_i \ar{r}{\alpha_i} \ar{d}{\phi_i} & U \ar{d}{\phi} \\
V_i \ar{r}{\beta_i} & V
\end{tikzcd} \]
such that $\alpha_i$ and $U_i \to V_i \times_{V} U$ are summand inclusions and $U \simeq \bigsqcup_{1 \leq i \leq n} U_i$, $f(v) = x$, and each morphism $\overline{\chi}_i \coloneq f(v \to i): x \to x_i \coloneq f(i)$ is an inert edge covering $\chi_{\sigma_i}: [U_+ \ra V] \mapsto [(U_i)_+ \ra V_i]$ in $\BigFinT$. We first prove (5) by showing that the restriction
\[ n^{\lhd} \times_{\FinT^\op} \Ar(\FinT^\op) \times_{\FinT^\op} \wt{\cP}^\otimes \to n^{\lhd} \]
of $\pi$ along $f_{x, \phi, \Sigma}$ is a cocartesian fibration. In fact, by \cite[Lem.~6.1.1.1]{HA}, for any cocartesian fibration $\cX \to \FinT^\op$, the source functor $\ev_0: \Ar(\FinT^\op) \times_{\FinT^\op} \cX \to \FinT^\op$ is a cocartesian fibration, with an edge
\[ \left( \begin{tikzcd}
V & V' \ar{l} \\
W \ar{u} & W' \ar{l} \ar{u}
\end{tikzcd}, \: x \to y \right) \]
cocartesian if and only if the square in $\FinT$ is a pullback square and $x \to y$ is a cocartesian edge. Next, let $$s: n^{\lhd} \to n^{\lhd} \times_{\FinT^\op} \Ar(\FinT^\op) \times_{\FinT^\op} \wt{\cP}^\otimes$$ be a cocartesian section determined by $s(v) = (V \xleftarrow{\gamma} W, y \in \wt{\cP}^\otimes_{[U'_+ \ra W]})$, so that $s(i) = (V_i \xleftarrow{\gamma_i} W_i \coloneq W \times_V V_i, y_i)$ for $y \to y_i$ an inert edge lifting $W_i \to W$. Let $z = \wt{F}(\gamma^\ast x,y)$, let $\sigma'_i$ be
\[ \begin{tikzcd}
U_i \times_{V} U' \ar{r} \ar{d} & U \times_V U' \ar{d}{\phi'} \\
W_i \ar{r} & W,
\end{tikzcd} \]
and let $\Sigma' = \{ \sigma'_1, ..., \sigma'_n \}$. Then a chase of the definitions shows that the composition $\wt{F} \circ \wt{P} \circ f_{x, \phi, \Sigma}$ is of the form $f_{z, \phi', \Sigma'}$, which shows (6). Finally, (8) follows from the right cancellation property of inert edges. This completes the verification of the hypotheses of \cite[Thm.~B.4.2]{HA}. We then deduce the theorem in question by means of \cref{prp:QuillenEquivalenceOfBigAndSmall} and \cref{cor:BigVsSmallOperads}. Finally, repeating this analysis for the other span proves the assertions regarding the monoidality of the construction.
\end{proof}

\begin{proposition} \label{prop:PropertiesOfAlgebraConstrFromBifunctor}
Let $F: \cO^\otimes \times_{\cT^\op} \cP^\otimes \to \cQ^\otimes$ be a bifunctor of $\cT$-$\infty$-operads and let $\cC^\otimes \to \cQ^\otimes$ be a fibration of $\cT$-$\infty$-operads. We then have the following properties of the construction $\underline{\Alg}_{\cQ, \cT}(\cP,\cC)^\otimes \to \cO^\otimes$:
\begin{enumerate}
\item For every object $x \in \cO$ over $V \in \cT^\op$, the parametrized restriction $$F_{\underline{x}}: \underline{x} \times_{\cT^\op} \cP^\otimes \to (\cQ^\otimes)_{\underline{V}}$$ is a morphism of $\cT^{/V}$-$\infty$-operads, and we obtain a canonical equivalence of $\cT^{/V}$-$\infty$-categories
$$\underline{\Alg}_{\cQ, \cT}(\cP,\cC)^\otimes \times_{\cO^\otimes} \underline{x} \simeq \underline{\Alg}_{\cQ_{\underline{V}},\cT^{/V}}(\cP_{\underline{V}}, \cC_{\underline{V}}).$$
Similarly, for every cocartesian section $\tau: \cT^\op \to \cO$, we have a canonical equivalence of $\cT$-$\infty$-categories
\[ \underline{\Alg}_{\cQ,\cT}(\cP, \cC)^\otimes \times_{\cO^\otimes, \tau} \cT^\op \simeq \underline{\Alg}_{\cQ,\cT}(\cP,\cC). \]
\item For every object $y \in \cP$ over $V \in \cT^\op$, the parametrized restriction $$F_{\underline{y}}: \cO^\otimes \times_{\cT^\op} \underline{y} \to (\cQ^\otimes)_{\underline{V}}$$ is a morphism of $\cT^{/V}$-$\infty$-operads, and `evaluation at $y$' furnishes a commutative square of $\cT^{/V}$-$\infty$-operads
\[ \begin{tikzcd}
(\underline{\Alg}_{\cQ, \cT}(\cP,\cC)^\otimes)_{\underline{V}} \ar{r}{\ev_{y}} \ar{d} & (\cC^\otimes)_{\underline{V}} \ar{d} \\
(\cO^\otimes)_{\underline{V}} \simeq \cO^\otimes \times_{\cT^\op} \underline{y} \ar{r}{F_{\underline{y}}} & (\cQ^\otimes)_{\underline{V}}.
\end{tikzcd} \]
Similarly, for every cocartesian section $\tau: \cT^\op \to \cP$, we have a morphism of $\cT$-$\infty$-operads
\[ \ev_{\tau}: \underline{\Alg}_{\cQ, \cT}(\cP,\cC)^\otimes \to \cC^\otimes \]
covering $F_{\tau}: \cO^\otimes \to \cQ^\otimes$ and given fiberwise by evaluation at $\tau(V)$.
\item If $\cC^\otimes$ is $\cQ$-monoidal (so that $\underline{\Alg}_{\cQ, \cT}(\cP,\cC)^\otimes$ is $\cO$-monoidal), then $\ev_y$ and $\ev_{\tau}$ preserve cocartesian edges.
\end{enumerate}
\end{proposition}
\begin{proof}
(1): We prove the assertion about $x$ -- that for $\tau$ will hold by the same reasoning. We have a commutative diagram
\[ \begin{tikzcd}
\underline{x} \times_{\cT^\op} \cP^\otimes \ar{r} \ar{d} & \cO^\otimes \times_{\cT^\op} \cP^\otimes \ar{r}{F} \ar{d} & \cQ^\otimes \ar{d} \\
(\cT^{/V})^\op \times_{\cT^\op} \uFinpT \ar{r} & \uFinpT \times_{\cT^\op} \uFinpT \ar{r}{\wedge} & \uFinpT
\end{tikzcd} \]
where the outer square induces the $\cT^{/V}$-functor $F_{\underline{x}}$. The definition of a bifunctor of $\cT$-$\infty$-operads then immediately shows that $F_{\underline{x}}$ is a morphism of $\cT^{/V}$-$\infty$-operads. Next, consider the commutative diagram
\[ \begin{tikzcd}
& \Ar(\underline{x}) \times_{\underline{x}} (\underline{x} \times_{\cT^\op} \cP^\otimes) \ar[bend right=15]{ld}[swap]{\ev_0} \ar[two heads]{d}{\simeq}[swap]{\rho} \ar[bend left=15]{rd}{\pr} & &  \\
\underline{x} \ar{d} & \underline{x} \times_{\cT^\op} \Ar(\cT^\op) \times_{\cT^\op} \cP^\otimes \ar{d} \ar{l}[swap]{\pi} \ar{r}{P \times \id} & \underline{x} \times_{\cT^\op} \cP^\otimes \ar{r}{F_{\underline{x}}} \ar{d} & \cQ^\otimes_{\underline{V}} \ar{d} \\
\cO^\otimes & \cO^\otimes \times_{\cT^\op} \Ar(\cT^\op) \times_{\cT^\op} \cP^\otimes \ar{r}{P \times \id} \ar{l}[swap]{\pi} & \cO^\otimes \times_{\cT^\op} \cP^\otimes \ar{r}{F} & \cQ^\otimes
\end{tikzcd} \]
where the functor $\rho$ is the trivial fibration used in the definition of the cocartesian pushforward. Using the variant of \cite[Lem.~4.8]{Exp2b} for algebra maps, after marking appropriately this diagram induces a comparison $\cT^{/V}$-functor
\[ \underline{\Alg}_{\cQ, \cT}(\cP,\cC)^\otimes \times_{\cO^\otimes} \underline{x} \to \underline{\Alg}_{\cQ_{\underline{V}},\cT^{/V}}(\cP_{\underline{V}}, \cC_{\underline{V}}), \]
which by \cite[Lem.~2.27]{Exp2} is an equivalence.

(2): By the same logic as in (1), $F_{\underline{y}}$ is a morphism of $\cT^{/V}$-$\infty$-operads. Using the compatibility of the construction $\underline{\Alg}_{\cQ, \cT}(\cP,\cC)^\otimes \to \cO^\otimes$ with base-change in $\cT$, without loss of generality we may replace $\cT^{/V}$ with $\cT$ and suppose $y \in \cP$ lies over a final object in $\cT$. Choosing a section $\sigma$ of the trivial Kan fibration $\underline{y} \xto{\simeq} \cT^\op$, let $j = (\id,\iota,\sigma): \cO^\otimes \to \cO^\otimes \times_{\cT^\op} \Ar(\cT^\op) \times_{\cT^\op} \cP^\otimes$ and consider the morphism of spans
\[ \begin{tikzcd}
& \cO^\otimes \ar{d}{j} \ar[bend right=15]{ld}[swap]{=} \ar[bend left=10]{rrd}{F_{\underline{y}}\sigma} & & \\
\cO^\otimes & \cO^\otimes \times_{\cT^\op} \Ar(\cT^\op) \times_{\cT^\op} \cP^\otimes \ar{r}{P \times \id} \ar{l}[swap]{\pi} & \cO^\otimes \times_{\cT^\op} \cP^\otimes \ar{r}{F} & \cQ^\otimes.
\end{tikzcd} \]
Noting that $j$ respects markings for the first span in \cref{conthm:MonoidalStructureOnAlgebras}, we then see that $j$ induces the desired morphism of $\cT$-$\infty$-operads $\underline{\Alg}_{\cQ, \cT}(\cP,\cC)^\otimes \to \cO^\otimes \times_{\cQ^\otimes} \cC^\otimes$.

(3): This follows from the proof of (2) since the functor $j$ also respects markings for the second span in \cref{conthm:MonoidalStructureOnAlgebras}.
\end{proof}

We now specialize to the bifunctor $\wedge: \uFinpT \times_{\cT^\op} \uFinpT \to \uFinpT$. Fix a choice of cocartesian pushforward $P: \uFinpT \times_{\cT^\op} \Ar(\cT^\op) \to \uFinpT$ and also write
\[ \wedge : \uFinpT \times^{\lax}_{\cT^\op} \uFinpT \coloneq \uFinpT \times_{\cT^\op} \Ar(\cT^\op) \times_{\cT^\op} \uFinpT \to \uFinpT \]
for the composition $\wedge \circ (P \times \id)$.

\begin{construction} \label{con:CanonicalAlgebraStructureOnAlgebra}
Let $\cC^\otimes$ be a $\cT$-symmetric monoidal $\cT$-$\infty$-category. We construct a $\cT$-symmetric monoidal $\cT$-functor
\[ (-)^{\can}: \underline{\CAlg}_{\cT}(\cC)^\otimes \to \underline{\CAlg}_{\cT}(\underline{\CAlg}_{\cT}(\cC))^\otimes \]
that is split by the `forgetful' evaluation $\cT$-functor $\mathrm{U}$ of \cref{prop:PropertiesOfAlgebraConstrFromBifunctor}(2). First observe that we have a commutative diagram of marked simplicial sets
\[ \begin{tikzcd}[column sep=6ex]
\uFinpT^\sharp \times_{\cT^\op}^{\lax} (\uFinpT, \Inert) \times_{\cT^\op}^{\lax} (\uFinpT, \Inert) \ar{r}{\wedge \times^{\lax}_{\cT^\op} \id} \ar{d}{\pr_{12}} & \uFinpT^\sharp \times^{\lax}_{\cT^\op} (\uFinpT, \Inert) \ar{r}{\wedge} \ar{d}{\pr_1} & \uFinpT^\sharp \\
\uFinpT^\sharp \times^{\lax}_{\cT^\op} (\uFinpT, \Inert) \ar{r}{\wedge} \ar{d}{\pr_1} & \uFinpT^\sharp \\
\uFinpT^\sharp
\end{tikzcd} \]
in which the square is a pullback (here, $\pr_{12}$ denotes projection away from the third factor). Also write $\wedge$ for the upper horizontal composite. Then we have that (cf. \cite[Lem.~2.26]{Exp2})
\[ (\pr_1)_{\ast} \wedge^{\ast} \cong (\pr_1)_{\ast} \wedge^{\ast} (\pr_1)_{\ast} \wedge^{\ast} : \sSet^+_{/\uFinpT} \to \sSet^+_{/\uFinpT}, \quad \leftnat{\cC}^\otimes \mapsto \leftnat{\underline{\CAlg}_{\cT}}(\underline{\CAlg}_{\cT}(\cC))^\otimes. \]
Now consider the morphisms of spans
\[ \begin{tikzcd}[column sep=8ex]
& \uFinpT^\sharp \times^{\lax}_{\cT^\op} (\uFinpT, \Inert) \ar[bend right=15]{ld}[swap]{\pr_1} \ar[bend left=15]{rd}{\wedge} \ar{d}{\id \times^\lax_{\cT^\op} \id \times^\lax_{\cT^\op} \iota} & \\
\uFinpT^\sharp & \uFinpT^\sharp \times^{\lax}_{\cT^\op} (\uFinpT, \Inert) \times^{\lax}_{\cT^\op} (\uFinpT, \Inert) \ar{l}[swap]{\pr_1} \ar{r}{\wedge} \ar{d}{\id \times^\lax_{\cT^\op} \wedge} & \uFinpT^\sharp. \\
& \uFinpT^\sharp \times^{\lax}_{\cT^\op} (\uFinpT, \Inert) \ar[bend left=15]{lu}{\pr_1} \ar[bend right=15]{ru}[swap]{\wedge}
\end{tikzcd} \]
Then the lower vertical arrow defines $(-)^{\can}$, and since the upper vertical arrow induces $\mathrm{U}$ and the composite is homotopic to the identity, we see that $\mathrm{U} \circ (-)^{\can} \simeq \id$.
\end{construction}

\begin{theorem} \label{thm:IndexedCoproductsAreTensorProducts}
Let $\cC^\otimes$ be a $\cT$-symmetric monoidal $\cT$-$\infty$-category. Then $\underline{\CAlg}_{\cT}(\cC)$ has all finite $\cT$-coproducts. Moreover, for any map of finite $\cT$-sets $f: U \to V$, we have a canonical equivalence
\[ \coprod_f \simeq f_{\otimes} : \underline{\CAlg}_{\cT}(\cC)_U \to \underline{\CAlg}_{\cT}(\cC)_V, \]
where $f_\otimes$ is furnished by the $\cT$-symmetric monoidal structure on $\underline{\CAlg}_{\cT}(\cC)$ of \cref{conthm:MonoidalStructureOnAlgebras}.
\end{theorem}
\begin{proof}
Since the base-change condition for left adjoints to the restriction functors $\{ f^\ast \}$ of $\underline{\CAlg}_{\cT}(\cC)$ to furnish finite $\cT$-coproducts is already satisfied by the maps $\{ f_{\otimes} \}$, it will suffice to construct unit and counit maps exhibiting $f_{\otimes}$ as left adjoint to $f^\ast$. Without loss of generality, we may suppose $V$ is an orbit, and after replacing $\cT$ by $\cT^{/V}$, we may suppose $V = \ast$ is the final object of $\cT$ so that $\underline{\CAlg}_{\cT}(\cC)_V \simeq \CAlg_{\cT}(\cC)$. In addition, by \cref{thm:InitialObjectsInAlgebrasWhenUnital}, $\CAlg_{\cT}(\cC)$ admits an initial object given by the unit $1$, which is $f_{\otimes} (\ast)$ for $f: \emptyset \to \ast$. We may thus suppose that $U$ is nonempty. Using \cref{con:CanonicalAlgebraStructureOnAlgebra}, given a $\cT$-commutative algebra $A$ we get
\[ A^{\can}: \uFinpT \to \underline{\CAlg}_{\cT}(\cC)^\otimes, \]
and we then get a map $f_{\otimes} f^{\ast} A \to A$ by applying $A^{\can}$ to $f_+: [U_+ \ra \ast] \to [\ast_+ \ra \ast]$ and factoring the resulting map $f^{\ast} A \to A$ through a cocartesian lift over $f_+$ in the base. Using the naturality of this procedure in $A$, we then obtain our candidate for the counit transformation $\epsilon: f_{\otimes} f^{\ast} \to \id$. To define the unit transformation $\eta: \id \to f^\ast f_\otimes$, consider the pullback square
\[ \begin{tikzcd}
U \times U \ar{r}{\pr_2} \ar{d}{\pr_1} & U \ar{d}{f} \\
U \ar{r}{f} & \ast
\end{tikzcd} \]
and the associated equivalence $f^{\ast} f_{\otimes} \simeq (\pr_1)_{\otimes} (\pr_2)^{\ast}$. The summand inclusion $\delta: U \to U \times U$ yields a natural transformation
$$\epsilon_{\pr_2^\ast(-)}:  \delta_\otimes \simeq \delta_{\otimes} (\delta^\ast \pr_2^\ast) = (\delta_{\otimes} \delta^{\ast} \pr_2^{\ast}) \to \pr_2^\ast,$$
which on objects $B \in \underline{\CAlg}_{\cT}(\cC)_U$ may be described as follows: if we write $U \times U \simeq U \coprod U'$ and $g = \pr_2|_{U'}$, then in terms of the decomposition $\underline{\CAlg}_{\cT}(\cC)_{U \times U} \simeq \underline{\CAlg}_{\cT}(\cC)_{U} \times \underline{\CAlg}_{\cT}(\cC)_{U'}$, we have that
$$\epsilon_{\pr_2^\ast B}:  \delta_\otimes(B) = (B, 1_{U'}) \to \pr_2^\ast(B) = (B, g^\ast B)$$
is given by the identity on the first factor and the unique map out of the initial object on the second factor. We then let
$$\eta = (\pr_1)_{\otimes} (\epsilon_{\pr_2^\ast (-)}):\id \simeq (\pr_1)_{\otimes} \delta_{\otimes} \to (\pr_1)_{\otimes} \pr_2^\ast \simeq f^\ast f_{\otimes}.$$

It remains to verify the triangle identities. Let $U \simeq \coprod_{i=1}^n U_i$ be an orbit decomposition of $U$, let $\iota_i: U_i \to U$ denote the inclusion, and let $f_i = f \circ \iota_i$. Observe that after pullback to $U_i$, the map $f$ acquires a canonical section, i.e., for all $1 \leq i \leq n$ we have a factorization of the identity map
\[ \id: U_i \xto{(\id,\iota_i)} U_i \times U \xto{p^i} U_i. \]
where $p^i$ denote the projection. Using $(-)^{\can}$, this furnishes a factorization
\[ \id: f_i^{\ast} A \to (p^i)_{\otimes} (p^i)^\ast (f_i^\ast A) \simeq (p^i)_{\otimes} (\pr_2)^{\ast} f^{\ast} A \to f_i^\ast A, \]
where we use the commutative square
\[ \begin{tikzcd}
U_i \times U \ar{r}{\pr_2} \ar{d}{p^i} & U \ar{d}{f} \\
U_i \ar{r}{f_i} & \ast
\end{tikzcd} \]
for the middle equivalence. To express this in more familiar terms, note that if we write $U_i \times U \simeq U_i \coprod U'_i$ and $q^i = p^i|_{U'_i}$, then we may identify this as
\[ f_i^{\ast} A \simeq f_i^{\ast} A \otimes 1_{U_i} \xto{\id \otimes 1} f_i^{\ast} A \otimes (q^i)_{\otimes} (q^i)^{\ast} A \xto{\id \otimes \epsilon} f_i^{\ast} A \otimes f_i^{\ast} A \xto{\epsilon} f_i^{\ast} A, \]
where $\otimes$ denotes the fiberwise tensor product on $\underline{\CAlg}_{\cT}(\cC)_{U_i}$ induced by the fold map $\nabla: U_i \coprod U_i \to U_i$.

Now regarding the composition $f^\ast A \xto{\eta f^\ast} f^\ast f_{\otimes} f^\ast A \xto{f^\ast \epsilon} f^\ast A$, by an elementary diagram chase we see that after pullback along $\iota_i$ this identifies with the factorization of $f_i^\ast A$ given above, which validates this half of the triangle identities.

Finally, we consider the composition $f_{\otimes} B \xto{f_{\otimes} \eta} f_{\otimes} f^{\ast} f_{\otimes} B \xto{\epsilon f^{\ast}} f_{\otimes} B$. By the $\cT$-symmetric monoidality of $(-)^{\can}$, we get a canonical equivalence $f_{\otimes} (B^{\can}) \simeq (f_{\otimes} B)^{\can}$ in $\CAlg_{\cT}(\underline{\CAlg}_{\cT}(\cC))$. If we then write $B = (A_1, ..., A_n)$ under the decomposition $\underline{\CAlg}_{\cT}(\cC)_{U} \simeq \prod_{i=1}^n \underline{\CAlg}_{\cT}(\cC)_{U_i}$, it follows that we obtain an equivalence
$$f_{\otimes} f^{\ast} f_{\otimes} B \simeq f_{\otimes}((p^1)_\otimes (p^1)^{\ast} A_1, ..., (p^n)_{\otimes} (p^n)^{\ast} A_n)$$
under which $\epsilon_{f_{\otimes} B} \simeq f_{\otimes}(\epsilon_{A_1},...,\epsilon_{A_n})$ and the composite $\epsilon_{f_{\otimes} B} \circ f_{\otimes} \eta_B$ identifies with $f_{\otimes}$ of the composite defined factorwise by the map $A_i \to (p^i)_{\otimes} (p^i)^{\ast} A_i \to A_i$ induced from $\id: U_i \xto{(\id, \iota_i)} U_i \times U \xto{p^i} U_i$. Since these all compose as identities, we deduce the other half of the triangle identities.
\end{proof}

\begin{corollary}
Let $\cC^\otimes$ be a $\cT$-symmetric monoidal $\cT$-$\infty$-category. Then the $\cT$-symmetric monoidal structure on $\underline{\CAlg}_{\cT}(\cC)$ of \cref{conthm:MonoidalStructureOnAlgebras} is cocartesian.
\end{corollary}


Next, we consider the more general case of an arbitrary $\cT$-indexing system $\cI$ (\cref{def:IndexingSystem}).

\begin{theorem} \label{thm:CoproductsAreTensorProductsIndexingSystem}
Let $\cC^\otimes$ be a $\cI$-symmetric monoidal $\cT$-$\infty$-category. Then for all orbits $V \in \cT$, the fiber $\underline{\CAlg}_{\cI}(\cC)_V$ admits finite coproducts, and for all morphisms $f: V \to W \in \cT$, the restriction functor $f^{\ast}: \underline{\CAlg}_{\cI}(\cC)_W \to \underline{\CAlg}_{\cI}(\cC)_V$ preserves finite coproducts. Moreover, finite coproducts are computed as tensor products in terms of the symmetric monoidal structures on the fibers $\underline{\CAlg}_{\cI}(\cC)_V$ constructed via \cref{conthm:MonoidalStructureOnAlgebras} applied to the bifunctor
\[ \wedge_{\cI}: \Com^{\otimes}_{\cT^{\simeq}} \times_{\cT^\op} \Com^{\otimes}_{\cI} \to \Com^{\otimes}_{\cI} \]
obtained by restriction of $\wedge: \uFinpT \times_{\cT^\op} \uFinpT \to \uFinpT$.
\end{theorem}
\begin{proof}
The proof is exactly analogous to that of \cref{thm:IndexedCoproductsAreTensorProducts}, where in place of \cref{con:CanonicalAlgebraStructureOnAlgebra} we instead use the composition of the span involving $\wedge_{\cI}$ with that involving $\wedge_{\cT^{\simeq}}$ to define the $\cT^{\simeq}$-symmetric monoidal $\cT$-functor
\[ \underline{\CAlg}_{\cI}(\cC)^\otimes \to \underline{\CAlg}_{\cT^{\simeq}}(\underline{\CAlg}_{\cI}(\cC))^\otimes, \]
and also the identification of \cref{cor:MinimalIndexingSystemIsFiberwiseCommutative}.
\end{proof}



\begin{corollary}
Let $\cC^\otimes$ be a distributive $\cT$-symmetric monoidal $\cT$-$\infty$-category such that $\cC$ is fiberwise presentable, let $\cI$ be a $\cT$-indexing system, and write $\cC^\otimes_{\cI} = \Com^\otimes_{\cI} \times_{\uFinpT} \cC^\otimes$. Let 
\[ \mathrm{U}: \underline{\CAlg}_{\cT}(\cC) \to \underline{\CAlg}_{\cI}(\cC_{\cI}) \]
be the forgetful $\cT$-functor implemented by restriction along $\Com^\otimes_{\cI} \subset \uFinpT$. Then for all $V \in \cT$, $\mathrm{U}_V$ is a conservative functor of presentable $\infty$-categories that preserves all small limits and colimits, and is hence comonadic. In particular, $\underline{\CAlg}_{\cT}(\cC)_V$ is comonadic over $\CAlg(\cC_V)$. 
\end{corollary}
\begin{proof} To reduce notational clutter, let us replace $\cT$ by $\cT^{/V}$ so that $V = \ast$ is a terminal object of $\cT$. By \cref{thm:ColimitsInAlgebras}(4), both $\CAlg_{\cT}(\cC)$ and $\CAlg_{\cI}(\cC_{\cI})$ are presentable. Since the forgetful functor to $\cC_{\ast}$ is conservative for a reduced $\cT$-$\infty$-operad, $\mathrm{U}_{\ast}$ is conservative as well. By \cref{thm:LimitsInAlgebras}(2), $\mathrm{U}_{\ast}$ preserves all small limits. By \cref{thm:ColimitsInAlgebras}(2), $\mathrm{U}_{\ast}$ preserves all sifted colimits. By \cref{thm:IndexedCoproductsAreTensorProducts} and \cref{thm:CoproductsAreTensorProductsIndexingSystem}, $\mathrm{U}_{\ast}$ preserves all finite coproducts. It follows that $\mathrm{U}_{\ast}$ preserves all small colimits and hence admits a right adjoint by the adjoint functor theorem, so we are entitled to ask about the comonadicity of $\mathrm{U}_{\ast}$. The conclusion then follows from the Barr--Beck--Lurie Theorem \cite[Thm.~4.7.3.5]{HTT}.
\end{proof}

\section{\texorpdfstring{$\cT$}{T}-symmetric monoidal structure on \texorpdfstring{$\cT$}{T}-presheaves}

Let $\cC$ be a $\cT$-symmetric monoidal $\cT$-$\infty$-category. In this section, we construct a $\cT$-symmetric monoidal structure on the $\cT$-$\infty$-category of $\cT$-presheaves $\underline{\PShv}_{\cT}(\cC)$ such that for any $\cT$-distributive $\cT$-symmetric monoidal $\cT$-$\infty$-category $\cD$, the universal mapping property  
\[ \underline{\Fun}^{L}_{\cT}(\underline{\PShv}_{\cT}(\cC), \cD) \xto{\simeq} \underline{\Fun}_{\cT}(\cC, \cD) \]
of \cite[Thm.~11.5]{Exp2} refines to an equivalence
\[ \underline{\Fun}^{L, \otimes}_{\cT}(\underline{\PShv}_{\cT}(\cC), \cD) \xto{\simeq} \underline{\Fun}^{\otimes}_{\cT}(\cC, \cD); \]
cf. \cref{cor:UMP_presheaves}. This result has also been achieved by Hilman as \cite[Thm.~2.3.6]{hilman2022}.

We first extend our discussion of universal constructions from \cite{Exp2b} by proving the parametrized analogue of \cite[Prop.~5.3.6.2]{HTT}.\footnote{For this, we need not suppose that $\cT$ is atomic.} Let $\CMcal{K} = \{ \CMcal{K}_{U} : U \in \cT \}$ be a collection of classes $\CMcal{K}_{U}$ of small $\cT^{/U}$-$\infty$-categories such that for each morphism $f: U \to V$ in $\cT$, $f^*(\CMcal{K}_{V}) \subset \CMcal{K}_{U}$; we call such a collection \emph{closed}. For each $V \in \cT$, let $\CMcal{K}|_{V} = \{ \CMcal{K}_{U} : [f: U \rightarrow V] \in \cT \}$. Recall the following definition from \cite[Def.~2.8]{Exp2b}

\begin{dfn} \label{dfn:classes_of_diagrams}
Let $\cC$ be a $\cT$-$\infty$-category. We say that $\cC$ \emph{strongly admits} $\ccK$-indexed $\cT$-colimits if for all $U \in \cT$, $\cC_{\underline{U}}$ admits $\ccK_U$-indexed $\cT^{/U}$-colimits. Likewise, for any $V \in \cT$, we may refer to $\cC_{\underline{V}}$ strongly admitting $\ccK|_V$-indexed $\cT^{/V}$-colimits.

Given a $\cT$-functor $F: \cC \to \cD$, we say that $F$ \emph{strongly preserves} $\ccK$-indexed $\cT$-colimits if for all $U \in \cT$, the $\cT^{/U}$-functor $F_{\underline{U}}: \cC_{\underline{U}} \to \cD_{\underline{U}}$ preserves all $\ccK_U$-indexed $\cT^{/U}$-colimits. Likewise, for any $V \in \cT$, we may refer to $F_{\underline{V}}$ strongly preserving $\ccK|_V$-indexed $\cT^{/V}$-colimits. We then let
\[ \underline{\Fun}^{\ccK}_{\cT}(\cC, \cD) \subset  \underline{\Fun}_{\cT}(\cC, \cD) \]
be the full $\cT$-subcategory spanned in each fiber over $V \in \cT$ by those $\cT^{/V}$-functors that strongly preserve $\ccK|_V$-indexed $\cT^{/V}$-colimits. Note that the global sections $\Fun^{\ccK}_{\cT}(\cC, \cD)$ of $\underline{\Fun}^{\ccK}_{\cT}(\cC, \cD)$ is then the full subcategory of $\Fun_{\cT}(\cC,\cD)$ spanned by those $\cT$-functors that strongly preserve $\ccK$-indexed $\cT$-colimits.

Suppose given a collection $\CMcal{R} = \{ \CMcal{R}_U: U \in \cT \}$ of classes $\CMcal{R}_U = \{ \overline{p_{\alpha}}: \cK_{\alpha}^{\underline{\rhd}} \to \cC_{\underline{U}} \}$ of $\cT^{/U}$-diagrams in $\cC_{\underline{U}}$ (which are not necessarily $\cT^{/U}$-colimit diagrams), such that for each morphism $f: U \to V$ in $\cT$, $f^*(\CMcal{R}_{V}) \subset \CMcal{R}_{U}$. Then a $\cT$-functor $F: \cC \to \cD$ \emph{strongly preserves} $\ccR$-indexed $\cT$-colimits if for all $U \in \cT$, $F_{\underline{U}}: \cC_{\underline{U}} \to \cD_{\underline{U}}$ sends each $\overline{p_{\alpha}}$ to $\cT^{/U}$-colimit diagram in $\cD_{\underline{U}}$. Likewise, for any $V \in \cT$ we have the same notion for $F_{\underline{V}}$ with respect to $\CMcal{R}|_V$. We then let
\[ \underline{\Fun}^{\ccR}_{\cT}(\cC, \cD) \subset  \underline{\Fun}_{\cT}(\cC, \cD) \]
be defined as before.
\end{dfn}

\begin{prp} \label{prp:ElaboratedYoneda}
Let $\cC$ be a $\cT$-$\infty$-category and let $\CMcal{R}$ be as in \cref{dfn:classes_of_diagrams}, such that each $\cK_{\alpha}$ lies in $\CMcal{K}_U$.

 Then there exists a $\cT$-$\infty$-category $\underline{\PShv}^{\CMcal{K}}_{\CMcal{R}}(\cC)$ and a $\cT$-functor $j: \cC \to \underline{\PShv}^{\CMcal{K}}_{\CMcal{R}}(\cC)$ such that:
\begin{enumerate}
\item $\underline{\PShv}^{\CMcal{K}}_{\CMcal{R}}(\cC)$ strongly admits $\CMcal{K}$-indexed $\cT$-colimits.

\item For all $\cT$-$\infty$-categories $\cD$ such that $\cD$ strongly admits $\CMcal{K}$-indexed $\cT$-colimits, precomposition with $j$ induces an equivalence of $\cT$-$\infty$-categories
\begin{equation*}
j^*: \underline{\Fun}^{\CMcal{K}}_{\cT}(\underline{\PShv}^{\CMcal{K}}_{\CMcal{R}}(\cC), \cD) \xto{\simeq} \underline{\Fun}^{\CMcal{R}}_{\cT}(\cC,\cD)
\end{equation*}
which upon passage to global sections yields an equivalence of $\infty$-categories
\begin{equation*}
j^*: \Fun^{\CMcal{K}}_{\cT}(\underline{\PShv}^{\CMcal{K}}_{\CMcal{R}}(\cC), \cD) \xto{\simeq} \Fun^{\CMcal{R}}_{\cT}(\cC,\cD).
\end{equation*}
\item Suppose that all $\cT^{/U}$-functors $\overline{p_{\alpha}}: \cK_{\alpha}^{\underline{\rhd}} \to \cC_{\underline{U}}$ in $\ccR_U$ are $\cT^{/U}$-colimit diagrams. Then $j$ is fully faithful.
\end{enumerate}
\end{prp}
\begin{proof}
The proof is essentially identical to that of \cite[Prop.~5.3.6.2]{HTT}; we spell out a few details for the reader's benefit. After enlarging universes, we may suppose $\cC$ is small.\footnote{Since we suppose $\cT$ is small, a $\cT$-$\infty$-category $\cC$ is small if and only if it is fiberwise small.} Let $j_0: \cC \to \underline{\PShv}_{\cT}(\cC)$ be the $\cT$-Yoneda embedding. For every $\overline{p_{\alpha}} \in \ccR_U$, let $p_{\alpha} = \overline{p_{\alpha}}|_{\cK_{\alpha}}$, let $Y_{\alpha} \in \underline{\PShv}_{\cT}(\cC)_U$ be the image of the cone point (in the fiber over $U$) under $\overline{p_{\alpha}}$, let $X_{\alpha} \in \underline{\PShv}_{\cT}(\cC)_U$ be a $\cT^{/U}$-colimit for $(j_0)_{\underline{U}} \circ p_{\alpha}: \cK_{\alpha} \to \underline{\PShv}_{\cT}(\cC)_{\underline{U}}$, let $s_{\alpha}: X_{\alpha} \to j_0(Y_{\alpha})$ be the induced map, and let $S_U = \{ s_{\alpha} \}$.

Note that for all morphisms $f: U \to V$ in $\cT$, $f^* S_V \subset S_U$ by our closure hypothesis on $\ccR$. Thus, we may form the full $\cT$-subcategory $S^{-1} \underline{\PShv}_{\cT}(\cC) \subset \underline{\PShv}_{\cT}(\cC)$ given on each fiber over $U \in \cT$ by the full subcategory of $S_U$-local objects of $\PShv_{\cT^{/U}}(\cC_{\underline{U}}) = \underline{\PShv}_{\cT}(\cC)_U$. We then have a $\cT$-left adjoint $L: \underline{\PShv}_{\cT}(\cC) \to S^{-1} \underline{\PShv}_{\cT}(\cC)$ given fiberwise by the usual localization. Finally, we define $\underline{\PShv}^{\CMcal{K}}_{\CMcal{R}}(\cC)$ to be the smallest full $\cT$-subcategory of $S^{-1} \underline{\PShv}_{\cT}(\cC)$ which contains the essential image of $L \circ j_0$ and such that for each $U \in \cT$, $\underline{\PShv}^{\CMcal{K}}_{\CMcal{R}}(\cC)$ is closed under $\ccK_U$-indexed colimits, and we let $j = L \circ j_0$ be the induced map.

Given this construction, the verification of properties (1)--(3) proceeds exactly as in the proof of \cite[Prop.~5.3.6.2]{HTT} (with parametrized analogues of non-parametrized statements involving colimits, left Kan extensions, etc. substituted as appropriate).
\end{proof}

\begin{rem}
If $\ccK = \All$ then we may also write $\underline{\PShv}^{\CMcal{K}}_{\CMcal{R}}(\cC)$ as $\underline{\PShv}_{\CMcal{R}}(\cC)$.
\end{rem}

Now let $\uCatT^{\otimes} \to \uFinpT$ be the $\cT$-symmetric monoidal $\cT$-$\infty$-category given by the $\cT$-cartesian $\cT$-symmetric monoidal structure on $\uCatT$ (\cref{exm:cartesian_monoidal}). Objects of $\uCatT^{\otimes}$ are then given by tuples
$$( [f: U \ra V] \in \uFinpT, \: \cC_1 \in \Cat_{\cT^{/U_1}}, \: ..., \: \cC_n \in \Cat_{\cT^{/U_n}} )$$
where $U \simeq U_1 \sqcup ... \sqcup U_n$ is an orbit decomposition.

To describe morphisms in $\uCatT^{\otimes}$, consider a morphism
\[ \begin{tikzcd}[row sep=4ex, column sep=4ex, text height=1.5ex, text depth=0.25ex]
U \ar{d}{f} & Z \ar{l} \ar{d} \ar{r}{m} & X \ar{d}{g} \\
V & Y \ar{l} \ar{r}{=} & Y
\end{tikzcd} \]
in $\uFinpT$, let $U \simeq \bigsqcup_{i=1}^n U_i$, $X \simeq \bigsqcup_{j=1}^m X_j$ be orbit decompositions, and let $m_j: Z_j \simeq X_j \times_Y Z \to X_j$ be the restriction over the summand $X_j$. Let $\{ \cC_i \in \Cat_{\cT^{/U_i}} \}$ and $\{ \cD_j \in \Cat_{\cT^{/X_j}} \}$. Let $\cC = \bigsqcup_{i=1}^n \cC_i$ denote the given $\cT^{/U}$-$\infty$-category and let $\pr^* \cC$ denote the pullback of $\cC$ to a $\cT^{/U \times_V Y}$-$\infty$-category along the projection $U \times_V Y \to U$. By definition, the map $Z \to U \times_V Y$ is a summand inclusion; let $\cC'$ denote the corresponding summand of $\cC$ regarded as a $\cT^{/Z}$-$\infty$-category, and also let $\cC'_j$ denote the $\cT^{/Z_j}$-$\infty$-category given by the further summand of $\cC'$. Then a morphism 
\[(f, \{ \cC_i \} ) \to (g, \{ \cD_j \}) \]
is given by a collection of $\cT^{/X_j}$-functors $F_j: (m_j)_*(\cC'_j) \to \cD_j$, or written more concisely, a $\cT^{/X}$-functor $F: m_* (\cC') \to \cD$.

\begin{dfn}
Let $\cM \subset \uCatT^{\otimes} \times \Delta^1$ be the subcategory defined as follows:

\begin{itemize}
\item $\cM_0 = \uCatT^{\otimes}$.
\item An object $([f: U \ra V], \: \cC_1 \in \Cat_{\cT^{/U_1}}, \: ..., \: \cC_n \in \Cat_{\cT^{/U_n}} )$ belongs to $\uCatT^{\otimes} \times \{ 1 \}$ if and only if each $\cC_i$ strongly admits $\cT^{/U_i}$-colimits.

\item All morphisms $(f, \{ \cC_i \}, 0) \to (g, \{ \cD_j \}, 1)$ belong to $\cM$.

\item A morphism $(f, \{ \cC_i \}, 1) \to (g, \{ \cD_j \}, 1)$ belongs to $\cM$ if and only if each $F_j$ is $\cT^{/X_i}$-distributive.
\end{itemize}

Let $p: \cM \to \uFinpT \times \Delta^1$ denote the composite of the inclusion and the structure map.

\end{dfn}

For the following, we have implicitly extended the notion of (closed) collection $\ccK = \{ \ccK_U \}$ to be indexed over all finite $\cT$-sets $U$.

\begin{ntn}
Let $f: U \to V$ be a morphism in $\FinT$ and let $\cC$ be a $\cT^{/U}$-$\infty$-category. Let $\ccR = \{ \ccR_{\alpha} : [U' \xrightarrow{\alpha} U] \in \FinT \}$ be a closed collection of diagrams in $\cC$, $\ccR_{\alpha} = \{ \overline{p_{\alpha}}: \cK_{\alpha}^{\underline{\rhd}} \to \cC_{\underline{U'}} \}$. We let $f_* \ccR$ denote the closed collection of diagrams in $f_* \cC$ specified at each morphism $\gamma: V' \to V$ as follows:

\begin{itemize}
\item[($\ast$)] Let $U' = U \times_V V'$ and let $f': U' \to V'$ be the pullback of $f$. For every $\cT^{/U'}$-diagram $\overline{p_{\alpha}}: \cK_{\alpha}^{\underline{\rhd}} \to \cC_{\underline{U'}}$ in $\ccR_{U'}$, we may form the $\cT^{/V'}$ diagram
\[ f'_*(\cK_{\alpha})^{\underline{\rhd}} \xto{\can} f'_* ( \cK_{\alpha}^{\underline{\rhd}}) \xto{f'_*(\overline{p_{\alpha}})} f'_*(\cC_{\underline{U'}}) \simeq (f_* \cC)_{\underline{V'}}. \]
Let $(f_* \cC)_{\gamma}$ be the set of these diagrams.
\end{itemize}

\end{ntn}

\begin{ntn}
Given a finite $\cT$-set $U$ with orbit decomposition $U \simeq \bigsqcup_{i=1}^n U_i$, let $\underline{\PShv}_{\cT^{/U}}(-)$ be given by the coproduct of the $\underline{\PShv}_{\cT^{/U_i}}(-)$.
\end{ntn}

\begin{lem} \label{lem:UniversallyLKEexample}
Let $f: U \to V$ be a morphism of finite $\cT$-sets and let $\cC$ be a $\cT^{/U}$-$\infty$-category. Consider the $\cT^{/V}$-functor
\[ \phi: f_* \underline{\PShv}_{\cT^{/U}}(\cC) \to \underline{\PShv}_{\cT^{/V}}(f_* \cC) \]
given by the composite of the $\cT^{/U}$-Yoneda embedding (for $f_* \underline{\PShv}_{\cT^{/U}}(\cC)$) and restriction along $f_*$ of the $\cT^{/V}$-Yoneda embedding (for $\cC$). Then $\phi$ is $\cT^{/V}$-distributive.
\end{lem}
\begin{proof}
Suppose $p: \cK \to \underline{\PShv}_{\cT^{/U}}(\cC)$ is a $\cT^{/U}$-diagram. We need to show that the $\cT^{/V}$-colimit of the composite
\[ f_* \cK \xto{f_*(p)} f_* \underline{\PShv}_{\cT^{/U}}(\cC) \xto{\phi} \underline{\PShv}_{\cT^{/V}}(f_* \cC) \]
evaluates to $f_* (\colim^{\cT^{/U}} p)$. It suffices to check this after evaluation at all objects of $f_* \cC$, and without loss of generality it suffices to consider $x \in (f_* \cC)_V \simeq \cC_U$ (by the usual base-change argument). A diagram chase then shows that the composite
\[  f_* \underline{\PShv}_{\cT^{/U}}(\cC) \xto{\phi} \underline{\PShv}_{\cT^{/V}}(f_* \cC) \xto{\ev_x} \underline{\Spc}_{\cT^{/V}} \]
identifies with the composite
\[ f_* \underline{\PShv}_{\cT^{/U}}(\cC) \xto{f_* \ev_x} f_* \underline{\Spc}_{\cT^{/U}} = f_* f^* \underline{\Spc}_{\cT^{/V}} \xto{m} \underline{\Spc}_{\cT^{/V}} \]
where $m$ is the multiplication given by the $\cT$-distributive $\cT$-cartesian $\cT$-symmetric monoidal structure on $\underline{\Spc}_{\cT}$. Thus, $m$ is $\cT^{/V}$-distributive, which implies that $\ev_x \phi$ is distributive. 
\end{proof}

\begin{lem} \label{lem:distributiveLKEequivalence}
Let $f: U \to V$ be a morphism of finite $\cT$-sets, let $\cC$ be a $\cT^{/U}$-$\infty$-category, and let $\cD$ be a $\cT^{/U}$-cocomplete $\cT^{/U}$-$\infty$-category. Then a $\cT^{/V}$-functor $F: f_* \underline{\PShv}_{\cT^{/U}}(\cC) \to \cD$ is $\cT^{/V}$-distributive if and only if it is the $\cT^{/V}$-left Kan extension of its restriction along $f_*(j): f_* \cC \subset f_* \underline{\PShv}_{\cT^{/U}}(\cC)$.
\end{lem}
\begin{proof}
For the `if' direction, it suffices to consider the universal example given by the $\cT^{/V}$-functor $\phi$ of \cref{lem:UniversallyLKEexample}, which we showed to be $\cT^{/V}$-distributive.\footnote{cf. \url{https://math.stackexchange.com/questions/4094599/when-left-kan-extension-preserve-colimits} for a description of this standard reduction, which also works in the parametrized context.} For the `only if' direction, let $F' = (f_* j)_! (f_* j)^* F$ and consider the comparison map $\theta: F' \Rightarrow F$. Without loss of generality, it suffices to show that for all $x \in (f_* \underline{\PShv}_{\cT^{/U}}(\cC))_V \simeq (\underline{\PShv}_{\cT^{/U}}(\cC))_U$, $\theta_x$ is an equivalence. On the one hand, by the pointwise formula for $\cT^{/V}$-left Kan extensions, we have that $F'(x)$ is given by the $\cT^{/V}$-colimit of $p: (f_* \cC)^{/\underline{x}} \to f_* \cC \to \cD$. On the other hand, since $f_*: \Cat_{\cT^{/U}} \to \Cat_{\cT^{/V}}$ preserves cotensors and limits, we have a natural equivalence $(f_* \cC)^{/\underline{x}} \simeq f_*(\cC^{/\underline{x}})$ such that $[(f_* \cC)^{/\underline{x}} \to f_* \cC] \simeq [f_*(\cC^{/\underline{x}} \to \cC)]$, and using that $F$ is $\cT^{/V}$-distributive, we get that $F(x)$ is also given by the $\cT^{/V}$-colimit of $p$. The naturality of all operations considered shows further that $\theta_x$ implements this equivalence.
\end{proof}

In the following corollary, we disambiguate our terminology for distributive functors by referring to the morphism of finite $\cT$-sets and not just the target. 

\begin{cor} \label{cor:distr_transitivity}
Let $U \xto{f} V \xto{g} W$ be a composite of morphisms of finite $\cT$-sets, let $\cC$ be a $\cT^{/U}$-cocomplete $\cT^{/U}$-$\infty$-category, and let $\cD$ be a $\cT^{/W}$-cocomplete $\cT^{/W}$-$\infty$-category. Consider the composite $\cT^{/W}$-functor
\[j:  g_* f_* \cC \xto{j_0} g_* \underline{\PShv}_{\cT^{/V}}(f_* \cC) \to g_* \underline{\PShv}_{f_* \All}(f_* \cC) \]
Then a $\cT^{/W}$-functor $F: g_* \underline{\PShv}_{f_* \All}(f_* \cC) \to \cD$ is $g$-distributive if and only if $j^* F$ is $g f$-distributive. Consequently, we have an equivalence
\[ \underline{\PShv}_{(g f)_* \All}((g f)_* \cC ) \simeq \underline{\PShv}_{g_* \All}(g_* \underline{\PShv}_{f_* \All}(f_* \cC)). \]
\end{cor}
\begin{proof}
The first claim follows immediately by restricting the equivalence of \cref{lem:distributiveLKEequivalence}, noting that by definition $g_* \underline{\PShv}_{f_* \All}(f_* \cC)$ is a localization of $g_* \underline{\PShv}_{\cT^{/V}}(f_* \cC)$ at the relevant class of morphisms. The equivalence then follows by universal property.
\end{proof}

\begin{prp}
The map $p$ is a cocartesian fibration.
\end{prp}
\begin{proof}
We adapt the proof of \cite[Prop.~4.8.1.3]{HA} to the parametrized context. We first show that $p$ is a locally cocartesian fibration. This is clear if we restrict to the fiber over $0$. For the other cases, first suppose $(f: U \ra V, \{ \cC_i \}) \in \cM_0$ and let $(f, 0) \to (g, 1)$ be a morphism in $\uFinpT \times \Delta^1$, with notation as above. Let $\cD_j = \underline{\PShv}_{\cT^{/X_j}}((m_j)_*(\cC'_j))$ and take $F_j$ to be the identity. We then have a morphism $(f, \{ \cC_i\},0) \to (g, \{ \cD_j \},1)$ which is a locally cocartesian edge by the universal property of the $\cT^{/X_j}$-presheaves.

Next, suppose $(f: U \ra V, \{ \cC_i \}) \in \cM_1$ and let $(f,1) \to (g,1)$ be a morphism in $\uFinpT \times \Delta^1$. Let $\cC'$ be the $\cT^{/Z}$-$\infty$-category as above. Let $\ccR$ be the closed collection of parametrized colimit diagrams in $\cC'$, i.e., for each morphism $\alpha: Z' \to Z$ in $\FinT$, $\ccR_{\alpha}$ is the collection of $\cT^{/Z'}$-colimit diagrams in $\cC'_{\underline{Z'}}$. We let $\cD = \underline{\PShv}^{\All}_{m_* \ccR}(m_* \cC')$ and $F = j: m_* \cC' \to \cD$ be the $\cT^{/X}$-functor as in \cref{prp:ElaboratedYoneda}. We then have that the morphism $j: (f,\cC,1) \to (g, \cD, 1)$ lies in $\cM$ by definition. Moreover, it is a locally cocartesian edge in view of the universal property supplied by \cref{prp:ElaboratedYoneda}.

To then see that $p$ is a cocartesian fibration, we need to see that the composite of locally cocartesian edges is again locally cocartesian. We already know the restriction over $0$ is a cocartesian fibration. If the first edge lies over $[0 \ra 1]$, we may apply the parametrized analogue of \cite[Prop.~5.3.6.11]{HTT}; since this step is straightforward we leave the details to the reader. If both edges lie over $1$, then without loss of generality we may suppose both edges are fiberwise active as edges over $\uFinpT$, in which case the claim follows from the transitivity property established in \cref{cor:distr_transitivity}.
\end{proof}

Let $p_0$ and $p_1$ denote the two fibers of $p$ over $0,1 \in \Delta^1$. 

\begin{cor}
The maps $p_0$ and $p_1$ exhibit $\cM_0$ and $\cM_1$ as $\cT$-symmetric monoidal $\cT$-$\infty$-categories.
\end{cor}
\begin{proof}
We already have that the map $p_0$ is the structure map of $\uCatT^{\otimes}$. As for $p_1$, since it is a cocartesian fibration it remains to verify the parametrized Segal condition. But in the definition of $\cM$, all the inert morphisms in $\uCatT^{\otimes}$ continue to lie in $\cM$, so we see that $\cM_1$ inherits the parametrized Segal condition from $\uCatT^{\otimes}$.
\end{proof}

Now write $(\uCatT^{L})^{\otimes} = \cM_1$. The cocartesian fibration $p$ classifies the $\cT$-symmetric monoidal $\cT$-functor
\[ \underline{\PShv}^{\otimes}_{\cT^{/-}}: \uCatT^{\otimes} \to  (\uCatT^{L})^{\otimes} \]
whose underlying $\cT$-functor is given by the usual $\cT$-presheaf construction $\underline{\PShv}_{\cT^{/-}}$. Since $\underline{\PShv}_{\cT^{/-}}$ admits a right $\cT$-adjoint given by the forgetful $\cT$-functor $U$, $U$ canonically inherits a lax $\cT$-symmetric monoidal structure. Passing to $\cT$-commutative algebra objects and considering the unit of the adjunction, we obtain:

\begin{cor} \label{cor:UMP_presheaves}
Let $\cC$ be a $\cT$-symmetric monoidal $\cT$-$\infty$-category. Then $\underline{\PShv}_{\cT}(\cC)$ and the $\cT$-Yoneda embedding $j: \cC \to \underline{\PShv}_{\cT}(\cC)$ inherit a $\cT$-symmetric monoidal structure such that:
\begin{enumerate}
\item $\underline{\PShv}_{\cT}(\cC)$ is $\cT$-distributive.
\item For every $\cT$-distributive $\cT$-symmetric monoidal $\cT$-$\infty$-category $\cD$, restriction along $j$ yields an equivalence
\[ \underline{\Fun}^{L,\otimes}_{\cT}(\underline{\PShv}_{\cT}(\cC), \cD) \xto{\simeq} \underline{\Fun}^{\otimes}_{\cT}(\cC, \cD). \]
\end{enumerate}
\end{cor}

\begin{rem}
Let $\cC$ be a $\cT$-symmetric monoidal $\cT$-$\infty$-category. Consider the $\cT$-distributive $\cT$-symmetric monoidal structure on $\underline{\PShv}_{\cT}(\cC) = \underline{\Fun}_{\cT}(\cC^{\vop}, \underline{\Spc}_{\cT})$ given by $\cT$-Day convolution (\cref{thm:DayConvolutionCocartesian}), where we have the $\cT$-symmetric monoidal structure on $\cC^{\vop}$ induced by the opposite automorphism on $\Cat$ under the equivalence of \cref{thm:TwoPresentationsOfTSMCs} and the $\cT$-cartesian $\cT$-symmetric monoidal structure on $\underline{\Spc}_{\cT}$. Then one may show directly that the full $\cT$-subcategory $\cC \subset \underline{\PShv}_{\cT}(\cC)$ is closed under this $\cT$-symmetric monoidal structure. Using \cref{cor:UMP_presheaves}(2), it then follows that the $\cT$-Day convolution $\cT$-symmetric monoidal structure agrees with that defined above.
\end{rem}


\bibliographystyle{amsalpha}
\bibliography{Gcats}

\newcommand{\etalchar}[1]{$^{#1}$}
\providecommand{\bysame}{\leavevmode\hbox to3em{\hrulefill}\thinspace}
\providecommand{\MR}{\relax\ifhmode\unskip\space\fi MR }
\providecommand{\MRhref}[2]{%
  \href{http://www.ams.org/mathscinet-getitem?mr=#1}{#2}
}
\providecommand{\href}[2]{#2}
\begin{thebibliography}{BDG{\etalchar{+}}16b}

\bibitem[Bar17]{M1}
Clark Barwick, \emph{Spectral {M}ackey functors and equivariant algebraic
  {K}-theory ({I})}, Advances in Mathematics \textbf{304} (2017), 646 -- 727.

\bibitem[BDG{\etalchar{+}}16a]{Exp0}
Clark Barwick, Emanuele Dotto, Saul Glasman, Denis Nardin, and Jay Shah,
  \emph{Parametrized higher category theory and higher algebra: {A} general
  introduction}, arXiv:1608.03654, 2016.

\bibitem[BDG{\etalchar{+}}16b]{Exp1}
\bysame, \emph{Parametrized higher category theory and higher algebra:
  Expos{\'e} {I} - {E}lements of parametrized higher category theory},
  arXiv:1608.03657, 2016.

\bibitem[BGN18]{BGN}
Clark Barwick, Saul Glasman, and Denis Nardin, \emph{Dualizing cartesian and
  cocartesian fibrations}, Theory and Applications of Categories \textbf{33}
  (2018), no.~4, 67--94.

\bibitem[BGS20]{BarwickGlasmanShah}
Clark Barwick, Saul Glasman, and Jay Shah, \emph{Spectral {M}ackey functors and
  equivariant algebraic {K}-theory, {II}}, Tunisian J. Math. \textbf{2} (2020),
  no.~1, 97--146.

\bibitem[BH15]{MR3406512}
Andrew~J. Blumberg and Michael~A. Hill, \emph{Operadic multiplications in
  equivariant spectra, norms, and transfers}, Adv. Math. \textbf{285} (2015),
  658--708. \MR{3406512}

\bibitem[BH20]{Blumberg2020}
Andrew~J. Blumberg and Michael~A. Hill, \emph{G-symmetric monoidal categories
  of modules over equivariant commutative ring spectra}, Tunisian Journal of
  Mathematics \textbf{2} (2020), no.~2, 237--286.

\bibitem[BH21]{BACHMANN2021}
Tom Bachmann and Marc Hoyois, \emph{Norms in motivic homotopy theory},
  Ast{\'{e}}risque \textbf{425} (2021).

\bibitem[BP21]{Bonventre2021}
Peter Bonventre and Lu{\'{\i}}s~A. Pereira, \emph{Genuine equivariant operads},
  Advances in Mathematics \textbf{381} (2021), 107502.

\bibitem[CH21]{chu2021free}
Hongyi Chu and Rune Haugseng, \emph{Free algebras through day convolution},
  arXiv:2006.08269, 2021.

\bibitem[GMMO18]{guillou2018symmetric}
Bertrand Guillou, J.~Peter May, Mona Merling, and Angélica~M. Osorno, \emph{A
  symmetric monoidal and equivariant segal infinite loop space machine}, 2018,
  arXiv:1711.09183.

\bibitem[GW18]{Gutirrez2018}
Javier~J Guti{\'{e}}rrez and David White, \emph{Encoding equivariant
  commutativity via operads}, Algebraic {\&} Geometric Topology \textbf{18}
  (2018), no.~5, 2919--2962.

\bibitem[HH16]{hillhopkins}
M.~A. Hill and M.~J. Hopkins, \emph{Equivariant symmetric monoidal structures},
  arXiv:1610.03114, 2016.

\bibitem[HHK{\etalchar{+}}20]{hahn2020quivariant}
Jeremy Hahn, Asaf Horev, Inbar Klang, Dylan Wilson, and Foling Zou,
  \emph{Equivariant nonabelian {P}oincaré duality and equivariant
  factorization homology of {T}hom spectra}, arXiv preprint arXiv:2006.13348
  (2020).

\bibitem[Hil22a]{hilman2022}
Kaif Hilman, \emph{Parametrised noncommutative motives and equivariant
  algebraic {K}-theory}, arXiv:2202.02591, 2022.

\bibitem[Hil22b]{hilman2022presentability}
Kaif Hilman, \emph{Parametrised presentability over orbital categories}, 2022.

\bibitem[Hin15]{Hinich}
Vladimir Hinich, \emph{Rectification of algebras and modules}, Documenta
  Mathematica (2015), no.~20, 879--926.

\bibitem[Hin20]{HINICH2020107129}
Vladimir Hinich, \emph{Yoneda lemma for enriched $\infty$-categories}, Advances
  in Mathematics \textbf{367} (2020), 107129.

\bibitem[Hor19]{horev2019enuine}
Asaf Horev, \emph{Genuine equivariant factorization homology}, arXiv preprint
  arXiv:1910.07226 (2019).

\bibitem[Lur09]{HTT}
Jacob Lurie, \emph{Higher topos theory}, Annals of Mathematics Studies, vol.
  170, Princeton University Press, Princeton, NJ, 2009. \MR{2522659
  (2010j:18001)}

\bibitem[Lur17]{HA}
\bysame, \emph{Higher algebra}, Preprint from the web page of the author, May
  2017.

\bibitem[MMO21]{may2021equivariant}
J.~Peter May, Mona Merling, and Angélica~M. Osorno, \emph{Equivariant infinite
  loop space theory, the space level story}, 2021, arXiv:1704.03413.

\bibitem[Nar16]{Exp4}
Denis Nardin, \emph{Parametrized higher category theory and higher algebra:
  Expos{\'e} {IV} -- {S}tability with respect to an orbital $\infty$-category},
  arXiv:1608.07704, 2016.

\bibitem[Nar17]{nardin}
Denis Nardin, \emph{Stability and distributivity over orbital
  $\infty$-categories}, Ph.D. thesis, 2017.

\bibitem[QS21a]{quigleyshah_tate}
J.~D. Quigley and Jay Shah, \emph{On the parametrized {T}ate construction},
  arXiv:2110.07707, 2021.

\bibitem[QS21b]{QS21b}
J.D. Quigley and Jay Shah, \emph{On the equivalence of two theories of real
  cyclotomic spectra}, 2021.

\bibitem[Rub21]{Rubin2021}
Jonathan Rubin, \emph{Combinatorial ${N}_\infty$ operads}, Algebraic {\&}
  Geometric Topology \textbf{21} (2021), no.~7, 3513--3568.

\bibitem[Sha21a]{Exp2}
Jay Shah, \emph{Parametrized higher category theory}, to appear in Algebraic
  {\&} Geometric Topology (2021), arXiv:1809.05892.

\bibitem[Sha21b]{Exp2b}
\bysame, \emph{Parametrized higher category theory {II}: Universal
  constructions}, arXiv:2109.11954, 2021.

\end{thebibliography}

\end{document}